\documentclass[11pt,reqno]{amsart}

\newtheorem{proposition}{Proposition}[section]
\newtheorem{lemma}[proposition]{Lemma}
\newtheorem{corollary}[proposition]{Corollary}
\newtheorem{theorem}[proposition]{Theorem}

\theoremstyle{definition}
\newtheorem{definition}[proposition]{Definition}
\newtheorem{example}[proposition]{Example}
\newtheorem{examples}[proposition]{Examples}
\newtheorem{remark}[proposition]{Remark}
\newtheorem{remarks}[proposition]{Remarks}

\newcommand{\thlabel}[1]{\label{th:#1}}
\newcommand{\thref}[1]{Theorem~\ref{th:#1}}
\newcommand{\selabel}[1]{\label{se:#1}}
\newcommand{\seref}[1]{Section~\ref{se:#1}}
\newcommand{\lelabel}[1]{\label{le:#1}}
\newcommand{\leref}[1]{Lemma~\ref{le:#1}}
\newcommand{\prlabel}[1]{\label{pr:#1}}
\newcommand{\prref}[1]{Proposition~\ref{pr:#1}}
\newcommand{\colabel}[1]{\label{co:#1}}
\newcommand{\coref}[1]{Corollary~\ref{co:#1}}
\newcommand{\relabel}[1]{\label{re:#1}}
\newcommand{\reref}[1]{Remark~\ref{re:#1}}
\newcommand{\exlabel}[1]{\label{ex:#1}}
\newcommand{\exref}[1]{Example~\ref{ex:#1}}
\newcommand{\delabel}[1]{\label{de:#1}}
\newcommand{\deref}[1]{Definition~\ref{de:#1}}
\newcommand{\eqlabel}[1]{\label{eq:#1}}
\newcommand{\equref}[1]{(\ref{eq:#1})}

\def\ot{\otimes}

\def\RR{{\mathbb R}}

\def\CC{{\mathbb C}}
\def\ZZ{{\mathbb Z}}

\newcommand{\Cc}{\mathcal{C}}

\newcommand{\Mm}{\mathcal{M}}

\def\*C{{}^*\hspace*{-1pt}{\Cc}}
\def\text#1{{\rm {\rm #1}}}

\input xy
\xyoption {all} \CompileMatrices

\usepackage{amssymb}
\usepackage{color,amssymb,graphicx,amscd,amsmath}
\usepackage[colorlinks,urlcolor=blue,linkcolor=blue,citecolor=blue]{hyperref}
\usepackage{caption}

\begin{document}
\title[Jacobi and Poisson algebras]
{Jacobi and Poisson algebras}

\author{A. L. Agore}
\address{Faculty of Engineering, Vrije Universiteit Brussel, Pleinlaan 2, B-1050 Brussels, Belgium
\textbf{and} Department of Applied Mathematics, Bucharest
University of Economic Studies, Piata Romana 6, RO-010374
Bucharest 1, Romania} \email{ana.agore@vub.ac.be and
ana.agore@gmail.com}

\author{G. Militaru}
\address{Faculty of Mathematics and Computer Science, University of Bucharest, Str.
Academiei 14, RO-010014 Bucharest 1, Romania}
\email{gigel.militaru@fmi.unibuc.ro and gigel.militaru@gmail.com}

\thanks{A.L. Agore is Postoctoral Fellow of the Fund for Scientific
Research-Flanders (Belgium) (F.W.O. Vlaanderen). This work was
supported by a grant of the Romanian National Authority for
Scientific Research, CNCS-UEFISCDI, grant no. 88/05.10.2011.}



\maketitle

\begin{abstract}
Jacobi/Poisson algebras are algebraic counterparts of
Jacobi/Poisson manifolds. We introduce representations of a Jacobi
algebra $A$ and Frobenius Jacobi algebras as symmetric objects in
the category. A characterization theorem for Frobenius Jacobi
algebras is given in terms of integrals on Jacobi algebras. For a
vector space $V$ a non-abelian cohomological type object
${\mathcal J}{\mathcal H}^{2} \, (V, \, A)$ is constructed: it
classifies all Jacobi algebras containing $A$ as a subalgebra of
codimension equal to ${\rm dim} (V)$. Representations of $A$ are
used in order to give the decomposition of ${\mathcal J}{\mathcal
H}^{2} \, (V, \, A)$ as a coproduct over all Jacobi $A$-module
structures on $V$. The bicrossed product $P \bowtie Q$ of two
Poisson algebras recently introduced by Ni and Bai appears as a
special case of our construction. A new type of deformations of a
given Poisson algebra $Q$ is introduced and a cohomological type
object $\mathcal{H}\mathcal{A}^{2} \bigl(P,\, Q ~|~
(\triangleleft, \, \triangleright, \, \leftharpoonup, \,
\rightharpoonup)\bigl)$ is explicitly constructed as a classifying
set for the bicrossed descent problem for extensions of Poisson
algebras. Several examples and applications are provided.
\end{abstract}

\section*{Introduction}
Noncommutative geometry as well as quantum group theory is based
on the same idea: instead of working with points on a given
\emph{space} $M$ which can be a compact topological group, a Lie
group, an algebraic group, a manifold, etc. we can work
equivalently with the algebra ${\rm Fun} \, (M)$ of functions on
$M$. The definition of the algebra of functions depends on the
category of spaces that we are dealing with: for instance, if $M$
is a compact topological group then ${\rm Fun} \, (M) := {\mathcal
R} \, (M)$, the algebra of real-valued continuous representative
functions on $M$, while if $M$ is a manifold then ${\rm Fun} \,
(M) := C^{\infty} (M)$, the algebra of all real smooth functions
on $M$. Thus there exists a contravariant functor ${\rm Fun} \, (
- )$, which very often is in fact a duality of categories, between
a certain category of spaces and a certain category of algebras
endowed with additional structures compatible with the algebra
structure such as coalgebras (i.e. Hopf algebras), Lie brackets
(i.e. Poisson algebras) etc. Adopting this categorical viewpoint,
purely geometric concepts sat at the foundation of a research
program started at the end of the 80's which deals with their
algebraic counterparts: the functor ${\rm Fun} \, ( - )$ being the
tool through which geometrical problems can be restated and
approached in algebraic language by way of a certain category of
algebras. The present paper fits within this context: we deal with
Jacobi algebras which are abstract algebraic counterparts of
Jacobi manifolds. Jacobi manifolds, as generalizations of
symplectic or more generally Poisson manifolds, were introduced
independently by Kirillov \cite{kirillov} and Lichnerowicz
\cite{lich}. A Jacobi manifold is a smooth manifold endowed with a
bivector field $\Lambda$ and a vector field $E$ satisfying some
compatibility conditions. When the vector field $E$ identically
vanishes, the Jacobi manifold is just a Poisson manifold.
Equivalently, a Jacobi manifold is a smooth manifold $M$ such that
the commutative algebra $A := C^{\infty} (M)$ of real smooth
functions on $M$ is endowed with a Lie bracket $[-, -]$ such that
$[ab, \, c ] =  a \, [b, \, c] + [a, \, c] \, b  - ab \, [1_A, \,
c]$, for all $a$, $b$, $c\in A$. Such an algebraic object $A$ is
called a Jacobi algebra while a Poisson algebra is just a Jacobi
algebra such that $[1_A, \, a] = 0$, for all $a\in A$. Poisson
algebras appear in several areas of research situated at the
border between mathematics and physics such as: Hamiltonian
mechanics, differential geometry, Lie groups and representation
theory, noncommutative algebraic/diferential geometry,
(super)integrable systems, quantum field theory, vertex operator
algebras, quantum groups and so on -- see the recent papers
\cite{CM, gra2013, LPV} and the references therein. If we look at
Poisson algebras as the 'differential' version of Hopf algebras,
then, mutatis-mutandis, Jacobi algebras can be seen as
generalizations of Poisson algebras in the same way as weak Hopf
algebras \cite{bohm} generalize Hopf algebras. It is therefore
natural to expect that Jacobi algebras will play an important role
in all fields enumerated above. For further details on the study
of Jacobi algebras from geometric view point we refer to
\cite{grab92, grabM, gra2013, marle, rovi}.

The paper is organized as follows: in \seref{prel} we survey the
basic concepts that will be used throughout the paper. In
particular, we recall the classical concept of a Frobenius algebra
together with the corresponding notion for Lie algebras which we
will call \emph{Frobenius Lie} algebra. The latter were previously
studied both in mathematics \cite{kat, medina} and physics
\cite{fig, pelc} under the name of self-dual or metric Lie
algebras. The property of being Frobenius reflects a certain
natural symmetry: for instance, a functor $F : {\mathcal C} \to
{\mathcal D}$ is called Frobenius \cite{CMZ2002} if $F$ has the
same left and right adjoint functor, while a finite dimensional
algebra $A$ is called Frobenius \cite{kadison} if $A \cong A^*$,
as right $A$-modules, which is the module version of the classical
problem of Frobenius asking when two canonical representations of
$A$ are equivalent. This idea will be used in \seref{frob} in the
definition of Frobenius Jacobi algebras. For more details on the
importance of Frobenius algebras as well as their applications to
topology, algebraic geometry and 2D topological quantum field
theories we refer to \cite{Kock}, for their categorical
generalization at the level of various types of (co)modules to
\cite{BW, CMZ2002} and for recent contributions and their
relevance to other fields (category theory, Hochschild cohomology
or graph theory, etc.) see \cite{bai, Miu2, lambre} and their
references. In \seref{frob} we present basic properties of Jacobi
algebras: in particular, we classify all complex Jacobi algebras
of dimension $2$ or $3$. The \emph{conformal deformation} of a
Jacobi algebra $A$ (\prref{conformal}) is the Jacobi version of
the Drinfel'd twist deformation of a quantum group. We introduce
the category ${\mathcal J} \Mm^A_A$ of \emph{Jacobi $A$-modules}
which is an equivalent way of defining representations of a Jacobi
algebra $A$ and based on this concept we define a \emph{Frobenius
Jacobi} algebra as a Jacobi algebra $A$ such that $A \cong A^*$,
as Jacobi $A$-modules. A weaker version of this notion, at the
level of Poisson algebras, was recently introduced in \cite{ZVZ}
where the term Frobenius Poisson algebra is used to denote Poisson
algebras $A$ such that $A \cong A^*$, isomorphism of right
$A$-modules. The concept of \emph{integral} on a Jacobi algebra is
introduced in \deref{integrale} having the Hopf algebra theory as
source of inspiration and it is used in the characterization
theorem of Frobenius Jacobi algebras (\thref{caracfrj}): a Jacobi
algebra $A$ is Frobenius if and only if there exists a
non-degenerate integral on $A$; in the finite dimensional case,
this is also equivalent to the existence of a so-called
\emph{Jacobi-Frobenius pair} of $A$, which allows us to define the
Euler-Casimir element associated to any finite dimensional
Frobenius Jacobi algebra.

The classification of finite dimensional Poisson manifolds is a
difficult task: the first steps towards the classification of low
dimensional Poisson manifolds were taken in \cite{gra1993, LPV}
using mainly differential geometry tools. It is natural to ask a
more general question namely that of classifying Jacobi manifolds
of a given dimension. Following the viewpoint of this paper we
look at its algebraic counterpart by asking about the
classification of all Jacobi algebras of a given dimension. The
problem is very difficult since it contains as subsequent
questions the classical problems of classifying all associative
(resp. Lie) algebras of a given dimension. For updates on the
progress made so far and the geometrical significance we refer to
\cite{am-2013d, mazzola} (resp. \cite{popovici}) and their list of
references - we just mention that the classification of all
complex associative (resp. Lie) algebras is known only up to
dimension $5$ (resp. $7$). One of the main tools which was
intensively used in the classification of finite 'objects' is the
famous \emph{extension problem} initiated at the level of groups
by H\"{o}lder and developed later on for Lie algebras, associative
algebras, Hopf algebras, Poisson algebras, etc. A more general
version of the extension problem is the \emph{extending structures
problem} (ES-problem) introduced in \cite{am-2011} for arbitrary
categories. \seref{unifiedprod} is devoted to the study of the
ES-problem for Jacobi algebras which consists of the following
question:

\textit{Let $A$ be a Jacobi algebra and $E$ a vector space
containing $A$ as a subspace. Describe and classify the set of all
Jacobi algebra structures that can be defined on $E$ such that $A$
becomes a Jacobi subalgebra of $E$.}

If we fix $V$ a complement of $A$ in the vector space $E$ then the
ES-problem can be rephrased equivalently as follows: describe and
classify all Jacobi algebras containing $A$ as a subalgebra of
codimension equal to ${\rm dim} (V)$. The answer to the ES-problem
is given in three steps: in the first step (\thref{1}) we shall
construct the \emph{unified product} $A \ltimes V$ associated to a
Jacobi algebra $A$ and a vector space $V$ connected through four
'actions' and two 'cocycles'. The unified product is a very
general construction containing as special cases the semi-direct
product, the crossed product or the bicrossed product. The second
step (\prref{classif}) shows that a Jacobi algebra structure $(E,
\ast_E, [-, \, -]_E )$ on $E$ contains $A$ as a Jacobi subalgabra
if and only if there exists an isomorphism of Jacobi algebras $(E,
\ast_E, [-, \, -]_E ) \cong A \ltimes V$. Finally, the theoretical
answer to the ES-problem is given in \thref{main1}: for a fixed
complement $V$ of $A$ in $E$, a non-abelian cohomological type
object ${\mathcal J}{\mathcal H}^{2} \, (V, \, A)$ is explicitly
constructed; it parameterizes and classifies all Jacobi algebras
containing $A$ as a subalgebra of codimension equal to ${\rm dim}
(V)$. The classification obtained in this theorem follows the
viewpoint of the extension problem: that is, up to an isomorphism
of Jacobi algebras that stabilizes $A$ and co-stabiliezes $V$. For
the sake of completeness we also write down the corresponding
results for Poisson algebras. Computing the classifying object
${\mathcal J}{\mathcal H}^{2} \, (V, \, A)$, for a given Jacobi
algebra $A$ and a vector space $V$ is a highly nontrivial
question. If $A := k$, the base field viewed as a Jacobi algebra
with the trivial bracket, then ${\mathcal J}{\mathcal H}^{2} \,
(V, \, k)$ classifies in fact all Jacobi algebras of dimension $1
+ {\rm dim} (V)$, which is of course a hopeless question for
vector spaces of large dimension. For this reason we shall assume
from now on that $A \neq k$. A very important step in computing
${\mathcal J}{\mathcal H}^{2} \, (V, \, A)$ is given in
\coref{descompunere}, where a decomposition of it as a coproduct
over all Jacobi $A$-modules structures on $V$ is given.
\seref{exemple} is devoted to computing ${\mathcal J}{\mathcal
H}^{2} \, (V, \, A)$ for what we have called \emph{flag Jacobi}
algebras over $A$: that is, a Jacobi algebra structure on $E$ such
that there exists a finite chain of Jacobi subalgebras $E_0 := A
\subset E_1 \subset \cdots \subset E_m := E$, such that each $E_i$
has codimension $1$ in $E_{i+1}$. All flag Jacobi algebras over
$A$ can be completely described by a recursive reasoning where the
key step is settled in \thref{clasdim1}: several applications and
examples are given at the end of the section. In particular, if
$A$ is a Poisson algebra we indicate the difference between
${\mathcal J}{\mathcal H}^{2} \, (V, \, A)$ and ${\mathcal
P}{\mathcal H}^{2} \, (V, \, A)$, the latter being the classifying
object of the ES-problem for Poisson algebras.

\seref{cazurispeciale} has as starting point the following remark:
the necessay and sufficient axioms for the construction of the
unified product, applied for (not necessarily unital) Poisson
algebras and for trivial cocycles reduces precisely to the
definition of the \emph{matched pairs} of Poisson algebras
(\deref{matchedposs}) which were recently introduced by Ni and Bai
\cite[Theorem 1]{NiBai} using the equivalent language of
representations. Moreover, the associated unified product in this
case is precisely the \emph{bicrossed product} of Poisson
algebras. \prref{factprob} shows that the bicrossed product is the
construction responsible for the \emph{factorization problem} at
the level of Poisson algebras. The rest of the section is devoted
to the converse of the factorization problem -- we call it the
\emph{bicrossed descent} (or the \emph{classification of
complements}) problem \cite{a-2013, am-2013a} which for Poisson
algebras comes down to the following question:

\emph{Let $P \subset R$ be an extension of Poisson algebras. If a
$P$-complement in $R$ exists (i.e. a Poisson subalgebra $Q \subset
R$ such that $R = P + Q$ and $P \cap Q = \{0\}$), describe
explicitly, classify all $P$-complement in $R$ and compute the
cardinal of the isomorphism classes of all $P$-complements in $R$
(which will be called the factorization index $[R: P]^f$ of $P$ in
$R$).}

Let $Q$ be a given $P$-complement and $\bigl(P, \, Q, \,
\triangleleft, \, \triangleright, \, \leftharpoonup, \,
\rightharpoonup \bigl)$ the associated canonical matched pair
(\prref{factprob}). In \thref{comp} a general deformation of the
Poisson algebra $Q$ is introduced: it is associated to a
deformation map $r: Q \to P$ in the sense of \deref{defmap}.
\thref{comp} proves that in order to find all complements of $P$
in $R$ it is enough to know only one $P$-complement: all the other
$P$-complements are deformations of it. The answer to the
bicrossed descent problem is given in \thref{clasif100}: there
exists a bijection between the isomorphism classes of all
$P$-complements of $R$ and a new cohomological type object
$\mathcal{H}\mathcal{A}^{2} \bigl(P,\, Q ~|~ (\triangleleft, \,
\triangleright, \, \leftharpoonup, \, \rightharpoonup)\bigl)$
which is constructed and the formula for computing the
factorization index $[R: P]^f$ is given. Examples are also
provided: in particular, an extension $P \subseteq R$ of Poisson
algebras such that $P$ has infinitely many non-isomorphic
complements in $R$ is constructed in \prref{ultimapr}.

\section{Preliminaries}\selabel{prel}
For a family of sets $(X_i)_{i\in I}$ we shall denote by
$\sqcup_{i\in I} \, X_i$ their coproduct in the category of sets,
that is $\sqcup_{i\in I} \, X_i$ is the disjoint union of $X_i$,
for all $i\in I$. Unless otherwise specified all vector spaces,
(bi)linear maps, tensor products are over an arbitrary field $k$.
A map $f: V \to W$ between two vector spaces is called trivial if
$f (v) = 0$, for all $v\in V$. $V^* = {\rm Hom}_k (V, \, k)$ and
${\rm End}_k (V)$ denote the dual, respectively the endomorphisms
ring of a vector space $V$. Throughout we use the following
convention: except for the situation when we deal with Poisson
algebras as in \seref{cazurispeciale}, by an algebra $A = (A,
m_A)$ we will always mean an associative, commutative and unital
algebra over $k$; the unit of $A$ will be denoted by $1_A$ while
the multiplication $m_A$ is denoted by juxtaposition $m_A (a, b) =
ab$. All morphisms of algebras preserve units and any left/right
$A$-module is unitary. For an algebra $A$ we shall denote by
${}_A\Mm_A$ the category of all $A$-bimodules, i.e. triples $(V,
\, \rightharpoonup, \, \triangleleft)$ consisting of a vector
space $V$ and two bilinear maps $\rightharpoonup \, : A \times V
\to V$, $\triangleleft : V \times A \to V$ such that $(V,
\rightharpoonup)$ is a left $A$-module, $(V, \triangleleft)$ is a
right $A$-module and $a \rightharpoonup (x \triangleleft b) = (a
\rightharpoonup x) \triangleleft b$, for all $a$, $b\in A$ and
$x\in V$. Although all algebras considered in this paper are
commutative we will maintain the adjectives left/right when
defining modules in order to clearly indicate the way actions are
defined. For unexplained concepts pertaining to Lie algebra theory
we refer to \cite{H}. In particular, representations of a Lie
algebra $\mathfrak{g}$ will be viewed as right Lie
$\mathfrak{g}$-modules. Explicitly, a right Lie
$\mathfrak{g}$-module is a vector space $V$ together with a
bilinear map $ \leftharpoonup \, : V \times \mathfrak{g} \to V$
such that $x \leftharpoonup [a, \, b] = (x \leftharpoonup a)
\leftharpoonup b \,  - \, (x \leftharpoonup b) \leftharpoonup a$,
for all $a$, $b \in \mathfrak{g}$ and $x\in V$. Left Lie
$\mathfrak{g}$-modules are defined analogously and the category of
right Lie $\mathfrak{g}$-modules will be denoted by ${\mathcal
L}\Mm^{\mathfrak{g}}$.

An algebra $A$ is called a \emph{Frobenius} algebra if $A \cong
A^*$ as right $A$-modules, where $A^*$ is viewed as a right
$A$-module via $(a^* \cdot a ) (b) := a^* (ab)$, for all $a^* \in
A^*$ and $a$, $b\in A$. For the basic theory of Frobenius algebras
we refer to \cite{kadison}. The Lie algebra counterpart of
Frobenius algebras was studied under different names such as
\emph{self-dual}, \emph{metric} or Lie algebras having a
non-degenerate invariant bilinear form. In this paper we will call
them Frobenius Lie algebras: a \emph{Frobenius Lie} algebra is a
Lie algebra $\mathfrak{g}$ such that $\mathfrak{g} \cong
\mathfrak{g}^*$ as right Lie $\mathfrak{g}$-modules, where
$\mathfrak{g}$ and $\mathfrak{g}^*$ are right Lie
$\mathfrak{g}$-modules via the canonical actions: $b
\leftharpoonup a := [b, \, a]$ and $\bigl( a^* \curvearrowleft a
\bigl) (b) := a^* \bigl([a, \, b] \bigl)$, for all $a$, $b\in
\mathfrak{g}$ and $a^* \in \mathfrak{g}^*$. We can easily see that
a Lie algebra $\mathfrak{g}$ is Frobenius if and only if there
exists a non-degenerate invariant bilinear form $B: \mathfrak{g}
\times \mathfrak{g} \to k$, i.e. $B([a, \, b], \, c) = B(a, \, [b,
\, c])$, for all $a$, $b$, $c\in \mathfrak{g}$. In light of this
reformulation, the second Cartan's criterion shows that any finite
dimensional complex semisimple Lie algebra is Frobenius since its
Killing form is non-degenerate and invariant. Let $\mathfrak{h}
(2n + 1, k)$ be the $(2n+1)$-dimensional Heisenberg algebra: it
has a basis $\{x_1, \cdots, x_n, y_1, \cdots, y_n, z \}$ and the
only non-zero Lie brackets are $[x_i, y_i] := z$, for all $i = 1,
\cdots, n$. Then $\mathfrak{h} (2n + 1, k)$ is not Frobenius: if
$B: \mathfrak{h} (2n + 1, k) \times \mathfrak{h} (2n + 1, k) \to
k$ is an invariant bilinear form then we can see that $B(z, \, -)
= 0$, that is $B$ is degenerate. Besides the mathematical interest
in studying Frobenius Lie algebras \cite{kat, medina}, they are
also important and have been intensively studied in physics
\cite{fig, pelc} - in particular for the construction of
Wess-Zumino-Novikov-Witten models.

A \emph{Poisson} algebra is a triple $A = (A, \, m_A, \, [-, \,
-])$, where $(A, m_A)$ is a (not necessarily unital) commutative
algebra, $(A, \, [-, \, -])$ is a Lie algebra such that the
Leibniz law
$$
[ab, \, c ] = a \, [b, \, c] + [a, \, c] \, b
$$
holds for any $a$, $b$, $c\in A$. For further details concerning
the study of Poisson algebras arising from differential geometry
see \cite{LPV} and the references therein. If a Poisson algebra
$A$ has a unit $1_A$, then by taking $a = b = 1_A$ in the Leibniz
law we obtain that $[1_A, \, c] = [c, \, 1_A] = 0$, for all $c\in
A$. Any non unital Poisson algebra embeds into a unital Poisson
algebra. If $A$ is a unital Poisson algebra, then using that $[a,
\, 1_A ] = 0$ and the Jacobi identity, we can easily prove that
the map
\begin{equation*}
R: A\ot A \to A\ot A, \qquad R (a \ot b) := b\ot a + 1_A \ot [a,
\, b]
\end{equation*}
for all $a$, $b\in A$ is a solution of the quantum Yang-Baxter
equation $R^{12} R^{23} R^{12} = R^{23} R^{12} R^{23}$ in ${\rm
End}_k (A\ot A \ot A)$. A (right) \emph{Poisson $A$-module}
\cite{LWZ, ZVZ} is a vector space $V$ equipped with two bilinear
maps $\triangleleft : V \times A \to V$ and $\leftharpoonup : V
\times A \to V$ such that $(V, \, \triangleleft)$ is a right
$A$-module, $(V, \, \leftharpoonup)$ is a right Lie $A$-module
satisfying the following two compatibility conditions for any $a$,
$b\in A$ and $x\in V$:
\begin{eqnarray}
x \leftharpoonup (ab) &=& (x \leftharpoonup a) \triangleleft b +
(x \leftharpoonup b) \triangleleft a, \quad x  \triangleleft [a,
\, b] = (x \triangleleft a) \leftharpoonup b - (x \leftharpoonup
b) \triangleleft a \eqlabel{Pmod1}
\end{eqnarray}
We denote by ${\mathcal P} \Mm^A_A$ the category of right Poisson
$A$-modules having as morphisms all linear maps which are
compatible with both actions.

\subsection*{Unified products for associative/Lie algebras.}
We recall some concepts and constructions from \cite{am-2013,
am-2013d} that will be used from \seref{unifiedprod} on.

\begin{definition}\delabel{comexdatum}
Let $A$ be an algebra and $V$ a vector space. An algebra
\textit{extending system of $A$ through $V$} is a system
$\Omega(A, \, V) = \bigl(\triangleleft, \, \triangleright, \, f,
\, \cdot \bigl)$ consisting of four bilinear maps $\triangleleft :
V \times A \to V$, \,  $\triangleright : V \times A \to A$, \, $f:
V\times V \to A$, \, $\cdot \, : V\times V \to V$ satisfying the
following six compatibility conditions for any $a$, $b \in A$,
$x$, $y$, $z \in V$:
\begin{enumerate}
\item[(A1)] $f$ and $\cdot$ are symmetric, $(V, \triangleleft)$ is
a right $A$-module and $x \triangleright 1_A = 0$

\item[(A2)] $x \cdot ( y \cdot z) - (x \cdot y) \cdot z = z
\triangleleft f (x, \, y) \, - \,  x \triangleleft f (y, \, z)$

\item[(A3)] $(x \cdot y) \triangleleft a = x \triangleleft (y
\triangleright a) + x \cdot (y \triangleleft a) $

\item[(A4)] $ x \triangleright (ab) = a (x \triangleright b) + (x
\triangleleft b) \triangleright a$

\item[(A5)] $ (x \cdot y) \triangleright a = x \triangleright (y
\triangleright a) + f(x, \, y \triangleleft a) - f (x, \, y) a$

\item[(A6)] $f(x, \, y\cdot z) - f(x\cdot y, \, z) = z
\triangleright f (x, \, y) - x \triangleright f(y, \, z)$
\end{enumerate}
\end{definition}

Let $\Omega(A, \, V) = \bigl(\triangleleft, \, \triangleright, \, f,
\, \cdot \bigl)$ be an extending system of $A$ through $V$ and $A
\, \ltimes_{\Omega(A, V)} V := A \, \times V $ with the
multiplication $\bullet$ defined for any $a$, $b \in A$ and $x$,
$y \in V$ by:
\begin{equation}\eqlabel{produnifcom}
(a, \, x) \bullet (b, \, y) := \bigl( ab + x \triangleright b + y
\triangleright a + f(x, y), \,\, x\triangleleft b +  y
\triangleleft a  + x \cdot y \bigl)
\end{equation}

Then $A \, \ltimes_{\Omega(A, V)} V = (A \, \ltimes_{\Omega(A, V)}
V, \, \bullet)$ is a commutative algebra having $(1_A, \, 0_V)$ as
a unit, called the \emph{unified product} of $A$ and $\Omega(A, \,
V)$. In fact, there is more to be said: $(A \, \ltimes_{\Omega(A,
V)} V, \, \bullet)$ is a commutative algebra with the unit $(1_A,
\, 0_V)$ if and only if $\Omega(A, \, V)$ is an extending system
of $A$ through $V$ -- this is the commutative version of
\cite[Theorem 2.2]{am-2013d}. Any unified product $A \,
\ltimes_{\Omega(A, V)} V$ contains $A \cong A \times \{0\}$ as a
subalgebra. Conversely, let $A$ be an algebra and $E$ a vector
space containing $A$ as a subspace. Then, any algebra structure
$\cdot$ on $E$ containing $A$ as a subalgebra is isomorphic to a
unified product, that is $(E, \cdot) \cong A \, \ltimes_{\Omega(A,
V)} V$, for some extending system of $A$ through $V$ -- this is
\cite[Theorem 2.4]{am-2013d} applied for the special case of
commutative algebras.

\begin{example} \exlabel{flagalg}
Let $A$ be an algebra. Then there is a bijection between the set
of all algebra extending systems of $A$ through $k$ and the set of
all $4$-tuples $(\Lambda, \, \Delta, \, f_0, \, u) \in A^* \times
{\rm End}_k (A) \times A \times k$ satisfying the following
compatibilities for any $a$, $b\in A$:
\begin{enumerate}
\item[(FA1)] $\Lambda : A \to k$ is an algebra map and $\Lambda
\circ \Delta = 0$

\item[(FA2)] $\Delta (a b) = a \, \Delta (b) + \Lambda(b) \,
\Delta (a)$

\item[(FA3)] $\Delta^2 (a) = u \, \Delta (a) + f_0 \, a - \Lambda (a) \, f_0 $
\end{enumerate}
The bijection is given such that the algebra extending system
$\Omega(A, \, k) = \bigl(\triangleleft, \, \triangleright, \, f, \,
\cdot \bigl)$ associated to $(\Lambda, \, \Delta, \, f_0, \, u)$
is defined for any $x$, $y\in k$ and $a\in A$ by:
\begin{equation}\eqlabel{bijflag}
x \triangleleft a := x\, \Lambda (a), \quad x \triangleright a :=
x \, \Delta (a), \quad f(x, \, y) := xy \, f_0, \quad x\cdot y :=
xy u
\end{equation}
A $4$-tuple $(\Lambda, \, \Delta, \, f_0, \, u)$ satisfying
(FA1)-(FA3) is called a \emph{flag datum} of $A$ and we denote by
${\mathcal F} (A)$ the set of all flag datums of $A$. The unified
product $A \, \ltimes_{\Omega(A, k)} \, k$ associated to a flag
datum $(\Lambda, \, \Delta, \, f_0, \, u)$ will be denoted by
$A_{(\Lambda, \, \Delta, \, f_0, \, u)}$ and coincides with the
vector space $A \times k$ having the multiplication given for any
$a$, $b \in A$, $x$, $y \in k$ by:
$$
(a, \, x) \bullet (b, \, y) := \bigl( ab + x \Delta(b) + y \Delta
(a) + xy \, f_0, \,\, x\Lambda (b) +  y \Lambda (a)  + x y u
\bigl)
$$
An algebra $B$ contains $A$ as a subalgebra of codimension $1$ if
and only if $B \cong A_{(\Lambda, \, \Delta, \, f_0, \, u)}$, for
some flag datum $(\Lambda, \, \Delta, \, f_0, \, u) \in {\mathcal
F} (A)$ \cite[Section 4]{am-2013d}.
\end{example}

The Lie algebra counterpart of the extending structures were
introduced in \cite{am-2013} as follows:

\begin{definition}\delabel{exdatum}
Let $A = (A, \, [-, \, -])$ be a Lie algebra and $V$ a vector
space. A Lie \textit{extending system of $A$ through $V$} is a
system $\Lambda (A, \, V) = \bigl(\leftharpoonup, \,
\rightharpoonup, \, \theta, \{-, \, -\} \bigl)$ consisting of four
bilinear maps $\leftharpoonup: V \times A \to V$, \,
$\rightharpoonup : V \times A \to A$, \, $\theta: V\times V \to
A$, \, $\{-, \, -\} : V\times V \to V$ satisfying the following
compatibility conditions for any $a$, $b \in A$, $x$, $y$, $z \in
V$:
\begin{enumerate}
\item[(L1)] $(V, \, \leftharpoonup)$ is a right Lie $A$-module, $
\theta (x, \, x) = 0$ and $\{x, \, x \} = 0$

\item[(L2)] $x \rightharpoonup [a, \, b] = [x \rightharpoonup a,
\, b] + [a, \, x \rightharpoonup b] + (x \leftharpoonup a)
\rightharpoonup b - (x \leftharpoonup b) \rightharpoonup a$

\item[(L3)] $\{x, \, y \} \leftharpoonup a = \{x, \, y
\leftharpoonup a \} + \{x \leftharpoonup a, \, y \} + x
\leftharpoonup (y \rightharpoonup a) - y \leftharpoonup (x
\rightharpoonup a)$

\item[(L4)] $\{x,\, y \} \rightharpoonup a = x \rightharpoonup (y
\rightharpoonup a) - y \rightharpoonup (x \rightharpoonup a) + [a,
\, \theta (x,\, y)] + \theta (x, y \leftharpoonup a) + \theta (x
\leftharpoonup a, y)$

\item[(L5)] $\sum_{(c)} \theta \bigl(x, \, \{y, \, z \}\bigl) \, +
\,  \sum_{(c)} x \rightharpoonup \theta (y, z) = 0$

\item[(L6)] $\sum_{(c)} \{x, \, \{y, \, z\}\} \, + \,  \sum_{(c)}
x \leftharpoonup \theta (y, z) = 0$
\end{enumerate}
where $\sum_{(c)}$ denotes the circular sum.
\end{definition}

The concept of extending system of a Lie algebra $A$ through a
vector space $V$ generalizes the concept of a matched pair of Lie
algebras as defined in \cite{LW, majid}: if $\theta$ is the
trivial map, then $\Lambda (A, \, V) = \bigl(\leftharpoonup, \,
\rightharpoonup, \, \theta := 0, \, \{-, \, -\} \bigl)$ is a Lie
extending system of $A$ through $V$ if and only if $(V, \, \{-, \,
-\})$ is a Lie algebra and $(A, V, \leftharpoonup, \,
\rightharpoonup)$ is a matched pair of Lie algebras. Let $\Lambda
(A, \, V) = \bigl(\leftharpoonup, \, \rightharpoonup, \, \theta,
\{-, \, -\} \bigl)$ be an extending system of a Lie algebra $A$
through a vector space $V$ and let $A \, \ltimes_{\Lambda (A, V)}
V $ be the vector space $A \, \times V$ with the bracket $[ -, \,
-]$ defined for any $a$, $b \in A$ and $x$, $y \in V$ by:
\begin{equation}\eqlabel{brackunif}
[(a, x), \, (b, y)] := \bigl( [a, \, b] + x \rightharpoonup b - y
\rightharpoonup a + \theta (x, y), \,\, \{x, \, y \} +
x\leftharpoonup b - y \leftharpoonup a \bigl)
\end{equation}
Then $A \, \ltimes_{\Lambda (A, V)} V $ is a Lie algebra called
the \textit{unified product} of $A$ and $\Lambda (A, \, V)$.
Moreover, \cite[Theorem 2.2]{am-2013} proves that $(A \,
\ltimes_{\Lambda (A, V)} V, \, [ -, \, -])$ with the bracket given
by \equref{brackunif} is a Lie algebra if and only if $\Lambda (A,
\, V)$ is a Lie extending system of $A$ through $V$. The Lie
algebra $A \, \ltimes_{\Lambda (A, V)} V $ contains $A \cong A
\times \{0\}$ as a Lie subalgebra. Conversely, let $A$ be a Lie
algebra, $E$ a vector space such that $A$ is a subspace of $E$.
Then, any Lie algebra structure $[-, \, -]$ on $E$
containing $A$ as a Lie subalgebra is isomorphic to a unified
product: i.e., $(E, [-, \, -]) \cong A \, \ltimes_{\Lambda (A, V)}
V$, for some $\Lambda (A, \, V) = \bigl(\leftharpoonup, \,
\rightharpoonup, \, \theta, \{-, \, -\} \bigl)$, an extending
system of $A$ through $V$ (\cite[Theorem 2.4]{am-2013}).

\begin{example} \exlabel{flagLie}
Let $A = (A, [-, \, -])$ be a Lie algebra. \cite[Proposition
4.4]{am-2013} proves that there is a bijection between the set of
all Lie extending system of $A$ through $k$ and the set ${\rm
TwDer} (A)$ of all \emph{twisted derivations} of $A$, which is the
set of all pairs $(\lambda, \, D) \in A^* \times {\rm End}_k (A)$
satisfying the following compatibilities for any $a$, $b\in A$:
\begin{enumerate}
\item[(FL1)] $\lambda ([a, \, b]) = 0$

\item[(FL2)] $D ([a, \, b]) = [D(a), \, b] + [a, \, D(b)] +
\lambda(a) D (b) - \lambda(b) D(a)$
\end{enumerate}
The bijection is given such that the Lie extending system $\Lambda
(A, \, k) = \bigl(\leftharpoonup, \, \rightharpoonup, \, \theta,
\{-, \, -\} \bigl)$ associated to a twisted derivation $(\lambda,
\, D)$ is defined for any $x$, $y\in k$ and $a\in A$ by:
\begin{equation}\eqlabel{bijflaglie}
x \leftharpoonup a := x\, \lambda (a), \quad x \rightharpoonup a
:= x \, D (a), \quad \theta(x, \, y) := 0, \quad \{x, \, y \} := 0
\end{equation}
${\rm TwDer} (A)$ contains the usual space of derivations ${\rm
Der} (A)$ via the canonical embedding ${\rm Der} (A)
\hookrightarrow {\rm TwDer} (A), \, D \mapsto (0, D) $. We point
out that the above canonical embedding is bijective if $A$ is a
perfect Lie algebra. The unified product $A \, \ltimes_{\Lambda
(A, k)} k $ associated to $(\lambda, \, D) \in {\rm TwDer} (A)$ is
denoted by $A_{(\lambda, \, D)}$ and it is the vector space $A
\times k$ with the bracket defined for any $a$, $b \in A$ and $x$,
$y \in k$ by:
$$
\left[(a, x), \, (b, y)\right] := \bigl( \left[a, \, b\right] + x
D (b) - y D(a), \,\, x\lambda (b) - y \lambda (a) \bigl)
$$
A Lie algebra $\mathfrak{g}$ contains $A$ as a Lie subalgebra of
codimension $1$ if and only if $\mathfrak{g} \cong A_{(\lambda, \,
D)}$, for some $(\lambda, \, D) \in {\rm TwDer} (A)$ \cite[Section
4]{am-2013}.
\end{example}

\section{Basic concepts, (bi)modules, integrals and Frobenius Jacobi algebras} \selabel{frob}
We recall the definition of Jacobi algebras \cite{grab92, marle}
as the abstract algebraic counterpart of Jacobi manifolds:

\begin{definition}\delabel{Jacalg}
A \emph{Jacobi algebra} is a quadruple $A = (A, \, m_A, \, 1_{A},
\, [-, \, -])$, where $(A, m_A, 1_{A})$ is a unital algebra, $(A,
\, [-, \, -])$ is a Lie algebra such that for any $a$, $b$, $c\in
A$:
\begin{equation}\eqlabel{jac1}
[ab, \, c ] =  a \, [b, \, c] + [a, \, c] \, b  - ab \, [1_A, \,
c]
\end{equation}
\end{definition}

Any unital Poisson algebra is a Jacobi algebra. Any algebra $A$ is
a Jacobi algebra with the trivial bracket $[a, \, b] = 0$, for all
$a$, $b\in A$ -- such an Jacobi algebra will be called
\emph{abelian} and will be denoted by $A_0$. A morphism between
two Jacobi (resp. Poisson) algebras $A$ and $B$ is a linear map
$\varphi: A \to B$ which is both a morphism of algebras as well as
a morphism of Lie algebras. We denote by ${}_ k{\rm Jac}$ (resp.
${}_ k{\rm Poss}$) the category of Jacobi (resp. unitary Poisson)
algebras over a field $k$. A \emph{Jacobi ideal} of a Jacobi
algebra $A$ is a linear subspace $I$ which is both an ideal with
respect to the associative product as well as a Lie ideal of $A$.
If $I$ is a Jacobi ideal of $A$ then $A/I$ inherits a Jacobi
algebra structure in the obvious way.

\begin{remarks}\relabel{reflective}
$(1)$ The category ${}_ k{\rm Poss}$ of unital Poisson algebras is
a reflective subcategory in the category of Jacobi algebras, i.e.
the inclusion functor $\iota : {}_ k{\rm Poss} \hookrightarrow {}_
k{\rm Jac}$ has a left adjoint which we will denote by $(-)_{\rm
poss} : {}_ k{\rm Jac} \to {}_ k{\rm Poss}$ and whose construction
goes as follows: for any Jacobi algebra $A$ we define $A_{\rm
poss} := A/ I_{\rm poss}$, where $I_{\rm poss}$ is the Jacobi
ideal of $A$ generated by all brackets $[1_A, \, a]$, for all
$a\in A$. Then, $A_{\rm poss}$ is a Poisson algebra and the
quotient map $A \to A_{\rm poss}$, $a \mapsto \widehat{a}$ is
universal among the maps from $A$ to any unital Poisson algebra
which are morphisms of Jacobi algebras. We mention that it is
possible that $A_{\rm poss} = 0$ - this happens when $I_{\rm
poss}$ contains an invertible element of $A$. An example is the
Jacobi algebra $J^{2, \, 2}$ defined in \exref{dim2} below.

$(2)$ The category ${}_ k{\rm Jac} = \bigl({}_ k{\rm Jac}, \,  -
\ot -, \, k, \, \tau_{-,-} \bigl)$ is braided monoidal: if $A$ and
$B$ are Jacobi algebras, then $A\ot B$ is a Jacobi algebra via
$$
(a \ot b) \cdot (a' \ot b') := aa' \ot bb', \quad [a\ot b, \, a' \ot b'] := aa' \ot [b, \, b'] + [a, \, a'] \ot bb'
$$
for all $a$, $a'\in A$ and $b$, $b'\in B$, the base field $k$ viewed with the
abelian Lie bracket is the unit and the braiding is the usual flip $\tau_{A, \, B} : A\ot B \to B \ot A$,
$\tau_{A, \, B} (a \ot b) := b\ot a$.
\end{remarks}

The classification of Jacobi algebras of a given finite dimension
is the first non-trivial question which arises as the algebraic
counterpart of the classification of finite dimensional Jacobi
manifolds. Of course, any $1$-dimensional Jacobi algebra is
isomorphic to the abelian Jacobi algebra $k_0$. The
$2$-dimensional case is covered below and it reveals an
interesting fact namely that two Jacobi algebras can be isomorphic
both as algebras and Lie algebras (through different isomorphisms)
but not isomorphic as Jacobi algebras. A similar result holds for
Hopf algebras as well.

\begin{example} \exlabel{dim2}
Let $k$ be a field of characteristic $\neq 2$. If $k = k^2 := \{
q^2 \, | \, q\in k \}$ then, up to an isomorphism, there exist
four $2$-dimensional Jacobi algebras over $k$. These are the
Jacobi algebras denoted by $J^{2, \, 1} $, $J^{2, \, 2}$, $J^{2,
\, 3}$, $J^{2, \, 4}$ having $\{1, \, x\}$ as a basis and the
multiplication and the bracket defined by:
$$
J^{2, \, 1}: \,\, x^2 = 0, \,\,\, [1, \, x] = 0; \qquad J^{2, \,
2}: \,\, x^2 = 0, \,\,\, [1, \, x] = 1
$$
$$
J^{2, \, 3}: \,\, x^2 = x, \,\,\, [1, \, x] = 0; \qquad J^{2, \,
4}: \,\, x^2 = 0, \,\,\, [x, \, 1] = x
$$
If $k \neq k^2$, besides the four Jacobi algebras listed above
there exists another one-parameter $2$-dimensional Jacobi algebra
$J^2_d$ defined by:
$$J^2_d:  \qquad  x^2 = d, \qquad [1, \, x] = 0 $$
for all $d \in S$, where $S \subseteq k\setminus k^2$ a system of
representatives for the following equivalence relation on $k
\setminus k^2$: $d \equiv d'$ if and only if there exists $q \in
k^{*}$ such that $d = q^{2} d'$.

Indeed, we will fix $\{1, \, x\}$ as a basis in a two dimensional
Jacobi algebra. The proof follows from the classical
classification of $2$-dimensional Lie algebras \cite{lie} and from
the well known classification of $2$-dimensional associative
algebras \cite{study} (for arbitrary fields see \cite[Corollary
4.5]{am-2013d}). Indeed, the classification follows by a routine
computation based on checking the compatibility condition
\equref{jac1}. We only mention that the algebra defined by the
multiplication $x^2 = x$ (or $x^2 = d$, for some $d \in S$, if $k
\neq k^2$) together with the Lie bracket $[1, \, x] = 1$ or $[1,
\, x] = x$ is not a Jacobi algebra since the compatibility
condition \equref{jac1} fails for $a = b = c := x$. We observe
that $J^{2, \, 2}$ and $J^{2, \, 4}$ are Jacobi non-Poisson
algebras, $(J^{2, \, 2})_{\rm poss} = 0$ and $(J^{2, \, 4})_{\rm
poss} \cong k_0$. Moreover, the Jacobi algebras $J^{2, \, 2}$ and
$J^{2, \, 4}$ are isomorphic as associative algebras as well as
Lie algebras but they are not isomorphic as Jacobi algebras.

In particular, if $k = \CC$ there are four isomorphism classes of
$2$-dimensional Jacobi algebras, while if $k = \RR$ there exist
five types of $2$-dimensional Jacobi algebras, the fifth one being
the Jacobi algebra $J^2_{-1}$.
\end{example}

\begin{example} \exlabel{dim3}
Using the classical classification of $3$-dimensional associative
(resp. Lie) algebras from \cite{study} (resp. \cite{Jacob, lie})
over the complex field $\CC$ and the same strategy as in
\exref{dim2} we can prove that, up to an isomorphism, there are
exactly $11$ families of complex Jacobi algebras of dimension $3$:
they have $\{1, \, x, \, y\}$ as a basis and the multiplication
and the bracket as listed in Table \ref{10}.
\begin{center}
\begin{table}
\begin{tabular}{ | l | c | c | }
\hline
Jacobi algebra & Algebra structure & Lie bracket\\
\hline

$J^{3}_{1}$ & $x^2 = y^{2} = xy = yx = 0$ & abelian \\\hline

$J^{3}_{2}$ & $x^2 = y$,\, $y^{2} = 0$,\, $xy = yx = 0$ & abelian
\\\hline

$J^{3}_{3}$ & $x^2 = x$, \,$y^2 = 0$,\, $xy = yx = 0$  & abelian
\\\hline

$J^{3}_{4}$ &  $x^2 = x$, \,$y^2 = y$, \,$xy = yx = 0$ & abelian \\
\hline

$J^{3}_{5}$ & $x^2 = y^{2} = xy = yx = 0$ & [x,\,1] = x \\\hline

$J^{3}_{6}$ & $x^2 = y^{2} = xy = yx = 0$ & [x,\,y] = x\\
\hline

$J^{3}_{7}$ & $x^2 = y^{2} = xy = yx = 0$ & [x,\,1] = x+y, \, [y,\, 1] = y\\
\hline

${}^{u}J^{3}_{8}$,\, $u \in k^{*}$ & $x^2 = y^{2} = xy = yx = 0$ & [x,\,1] = x, \, [y,\, 1] = u\,y\\
\hline

$J^{3}_{9}$ &  $x^2 = y$, \,$y^2 = 0$, \,$xy = yx = 0$ &  [x,\,1] = x, \, [y,\, 1] = 2\,y\\
\hline

$J^{3}_{10}$ &  $x^2 = y$, \,$y^2 = 0$, \,$xy = yx = 0$ &  [x,\,1] = $2^{-1}$\,x, \, [y,\, 1] = y\\
\hline

$J^{3}_{11}$ &  $x^2 = x$, \,$y^2 = 0$, \,$xy = yx = 0$ &  [y,\, 1] = y\\
\hline
\end{tabular}
\caption{Jacobi algebras of dimension $3$ over $\CC$.} \label{10}
\end{table}
\end{center}
\end{example}

Subsequent to the problem of classifying Jacobi algebras of a
given dimension is the following question: \emph{for a given
algebra (resp. Lie algebra) $A$, describe and classify all Lie
brackets $[-, \, -]$ (resp. all possible multiplications $m_A$)
which endow $A$ with a Jacobi algebra structure}. Some examples
are given below:

\begin{examples} \exlabel{ciclic}
$(1)$ Let $C_{n}$ be a cyclic group of order $n \geq 2$ generated by
$c$. Then the group algebra $k[C_{n}]$ can be made into a Jacobi
algebra as follows:

$(1)$ If ${\rm char} \, (k) = 0$  or ${\rm char} \, (k) = p > 0$
and $(p, \, n-1) = 1$, then the only Lie bracket which makes
$k[C_{n}]$ into a Jacobi algebra is the trivial one, i.e.
$k[C_{n}] = k[C_{n}]_0$;

$(2)$ If ${\rm char} \, (k) = p \, | \, n-1 $, then any $y \in
k[C_{n}]$ induces a unique Lie bracket given by $[c^{i}, \, c^{j}]
= (j - i) \, c^{j + i - 1} \, y$, for all $i$, $j = 1, 2, ..., n$
which makes $k[C_{n}]$ into a Jacobi algebra.

Indeed, let $[-, \, -]$ be a Lie bracket that endows the group
algebra $k[C_{n}]$ with a Jacobi algebra structure and we denote
$y := [1, c]$. By using \equref{jac1} and the induction we obtain
that $[c^{i}, \, c^{j}] = (j - i) \, c^{j + i - 1} \, y$, for all
$i$, $j = 1, 2, ..., n$. Therefore, since $[c^{n}, \, c] = y$ we
obtain $(n-1) \,y = 0$, which brings us to the two cases
considered above. If ${\rm char} \, (k) = p$ and $p | n-1$, then
it can be easily seen that Jacobi's identity is also fulfilled.

$(2)$ Let $C_{\infty}$ be the infinite cyclic group generated by
$c$. Then, for any $y \in k[C_{\infty}]$, the group algebra
$k[C_{\infty}]$ admits a Jacobi algebra structure with the bracket
given by $[c^{i}, \, c^{j}] = (j - i) \, c^{j + i - 1} \, y$, for
all $i$, $j \in \ZZ$.

$(3)$ Let ${\rm sl} \, (2, \, \CC)$ be the complex special linear
algebra of dimension $3$. Since ${\rm sl} \, (2, \, \CC)$ is
perfect as a Lie algebra, a careful analysis of the Lie brackets
on $3$-dimensional Jacobi algebras given in table \ref{10} brings
us to the conclusion that the Lie algebra ${\rm sl} \, (2, \,
\CC)$ can not be endowed with an algebra structure to make it into
a Jacobi agebra.
\end{examples}

The following construction is the algebraic counterpart of
conformal deformations of Jacobi manifolds \cite{marle}.
Mutatis-mutandis it can be seen as the Jacobi version of the
Drinfel'd twist deformation for quantum groups \cite{Dri}.

\begin{proposition}\prlabel{conformal}
Let $A = (A, \, m_A, \, [-, \, -])$ be a Jacobi algebra and $u \in
U (A)$ an invertible element of $A$. Then $A_u := (A, \, m_A, \,
[-, \, -]_u)$ is a Jacobi algebra called the \emph{$u$-conformal
deformation} of $A$, where the bracket $[-, \, -]_u$ is given for
any $x$, $y\in A$ by:
\begin{equation} \eqlabel{brackconf}
\left [x, \, y \right]_u := u^{-1} \, \left[u \,x, \,\, u\,y
\right]
\end{equation}
\end{proposition}

\begin{proof}
The bilinear map $[-, \, -]_u$ is a Lie bracket on $A$ since for
any $x$, $y$, $z \in A$ we have:
\begin{eqnarray*}
&&\bigl[x,\,[y,\,z]_{u}\bigl]_{u} +
\bigl[y,\,[z,\,x]_{u}\bigl]_{u}
+ \bigl[z,\,[x,\,y]_{u}\bigl]_{u} \, =\\
&=& u^{-1} \bigl[ux, \, [uy, \, uz]\bigl] + u^{-1} \bigl[uy, \,
[uz, \, ux]\bigl] + u^{-1} \bigl[uz, \, [ux, \, uy]\bigl] = 0
\end{eqnarray*}
Now, the compatibility condition \equref{jac1} is equivalent to $
[uxy, \, uz] = [ux, \, uz] \, y + x \, [uy,\, uz] - xy \, [u, \,
uz]$, for all  $x$, $y$, $z \in A$. The right hand side gives:
\begin{eqnarray*}
&&[ux, \, uz] \, y + x \, \underline{[uy, \, uz]} - xy \, [u, \, uz]=\\
&\stackrel{\equref{jac1}} {=}& [ux, \, uz] \, y + x \, \bigl(u
\,[y, \, uz] + [u, \, uz] \, y - uy \, [1_{A},\, uz]\bigl)
- xy \, [u, \, uz]\\
&=& [ux, \, uz] \, y + u\,x\,[y, \, uz] - u \, x\, y [1_{A},\, uz]
\stackrel{\equref{jac1}} {=} [uxy, \, uz]
\end{eqnarray*}
as needed, where the last equality follows by applying
\equref{jac1} for the bracket $[-,\,-]$ in $ux$, $y$ and
respectively $uz$.
\end{proof}

Describing and classifying all $u$-conformal deformations of a
given Jacobi algebra $A$ is an interesting question that will be
addressed somewhere else. In what follows we provide an example
which shows that a $u$-conformal deformation of a Jacobi algebra
can be a Poisson algebra.

\begin{example}\exlabel{conforexp}
Let $A$ be the $3$-dimensional Jacobi non-Poisson algebra over of
field of characteristic $\neq 2$ having $\{1, \, x, \, y \}$ as a
basis and the multiplication and the non-trivial brackets given by
$x^2 := 0$, \, $xy = yx := x$, \, $y^2 := 2 y -1$, \, $[1, \, y] =
[x, \, y] := x$ ($A$ is the Jacobi algebra denoted by $J^{3, \,
5}_{2, \, 0, \, -1, -1}$ in \prref{J21} below). The group $U(A)$
of units of $A$ coincides with the set of elements of the form $u
= \alpha + \beta\, x + \gamma \, y$, with $\alpha + \gamma \neq 0$
and the space of all $u$-conformal deformations of $A$ is in
bijection with the three-parameter Jacobi algebras having the
bracket defined by:
$$
[1, \, x]_u = -\gamma \, x, \,\,\,\,\,  [1, \, y]_u = (\alpha +
\beta + \gamma) \, x, \,\,\,\,\, [x, \, y]_u = (\alpha + 2 \gamma)
\, x
$$
for all $\alpha$, $\beta$, $\gamma \in k$ such that $\alpha +
\gamma \neq 0$. In particular, the $u$-conformal deformation of
$A$ corresponding to $u : = \alpha \, (1 - x)$, for any $\alpha
\in k^*$ is a Poisson algebra.
\end{example}

\subsection*{(Bi)Modules, integrals and Frobenius Jacobi algebras}
We shall introduce the representations of a Jacobi algebra using
the equivalent notion of modules.

\begin{definition}\delabel{moduleJac}
Let $A$ be a Jacobi algebra. A right \emph{Jacobi $A$-module} is a
vector space $V$ equipped with two bilinear maps $\triangleleft :
V \times A \to V$ and $\leftharpoonup : V \times A \to V$ such
that $(V, \, \triangleleft)$ is a right $A$-module, $(V, \,
\leftharpoonup)$ is a right Lie $A$-module satisfying the
following two compatibility conditions for any $a$, $b\in A$ and
$x\in V$:
\begin{eqnarray}
x \leftharpoonup (ab) &=& (x \leftharpoonup a) \triangleleft b +
(x \leftharpoonup b) \triangleleft a - (x \leftharpoonup 1_A) \triangleleft
(ab) \eqlabel{Jmod1}\\
x  \triangleleft [a, \, b] &=& (x \triangleleft a) \leftharpoonup
b - (x  \leftharpoonup b) \triangleleft a + (x \triangleleft a)
\triangleleft [1_A, \, b] \eqlabel{Jmod2}
\end{eqnarray}
A right \emph{Jacobi $A$-bimodule} is a right Jacobi $A$-module
which in addition satisfies the following compatibility condition
for any $a$, $b\in A$ and $x\in V$:
\begin{eqnarray}
x \leftharpoonup (ab) &=& (x \triangleleft a) \leftharpoonup b +
(x \triangleleft b) \leftharpoonup a \, - (x \triangleleft (ab)) \leftharpoonup
1_A \eqlabel{Jmod3}
\end{eqnarray}
Let ${\mathcal J} \Mm^A_A$ (resp. ${\mathcal J} {\mathcal B}^A_A$)
be the category of right Jacobi $A$-modules (resp. Jacobi
$A$-bimodules) having as morphisms the linear maps which are
compatible with both actions.
\end{definition}

The categories $^A_A{}{\mathcal J} {\mathcal M}$ (resp.
$^A_A{}{\mathcal J} {\mathcal B}$) of left Jacobi $A$-(bi)modules
are defined analogously and there exists an isomorphism of
categories $^A_A{}{\mathcal J} {\mathcal M} \cong {\mathcal J}
\Mm^A_A$ and $^A_A{}{\mathcal J} {\mathcal B} \cong {\mathcal J}
{\mathcal B}^A_A$.

\begin{remarks} \relabel{vbbimodule}
$(1)$ The compatibility conditions \equref{Jmod1}-\equref{Jmod2}
defining the category ${\mathcal J} \Mm^A_A$ are the Jacobi
version of the compatibility conditions \equref{Pmod1} defining
Poisson modules over a Poisson algebra: they are precisely axioms
(J2) and (J4) from the construction of the unified product
(\thref{1}). On the other hand, axiom \equref{Jmod3} defining
Jacobi bimodules has another explanation which will be given in
\leref{dualul} below.

$(2)$ Rephrasing \deref{moduleJac} in terms of representations can
be done as follows: a \emph{representation} of a Jacobi algebra
$A$ on a vector space $V$ is a pair $(\psi, \, \varphi)$
consisting of a representation $\psi$ of the associative algebra
$A$ on $V$, that is an algebra map $\psi: A \to {\rm End}_k (A)$
and a representation $\varphi$ of a Lie algebra $A$ on $V$, i.e. a
Lie algebra map $\varphi : A \to {\rm gl} (V)$, satisfying the
following two compatibilities for any $a$, $b\in A$:
\begin{eqnarray}
\varphi (ab) &=&  \psi (b) \circ \varphi (a) - \psi(a) \circ \varphi (b) + \psi(ab) \circ \varphi(1_A) \eqlabel{repJ1}\\
\psi \bigl([a, \, b]\bigl) &=& \psi(a) \circ \varphi (b) - \varphi(b) \circ \psi(a) + \psi \bigl([1_A, \, b] \, a\bigl) \eqlabel{repJ2}
\end{eqnarray}
Representations of a Jacobi algebra $A$ and right Jacobi
$A$-modules are two different ways of describing the same
structure: more precisely, there exists an equivalence of
categories ${\mathcal J} \Mm^A_A \cong {\rm Rep} (A)$, where ${\rm
Rep} (A)$ is the category of representations of $A$ with the
obvious morphisms. The one-to-one correspondence between right
Jacobi $A$-modules $(\triangleleft, \, \leftharpoonup)$ on $V$ and
representations $(\psi, \, \varphi)$ of $A$ on $V$ is given by the
two-sided formulas: $\psi(a) (x) :=: x\triangleleft a$ and
$\varphi (a) (x) :=: - x \leftharpoonup a$, for all $a\in A$ and
$x\in V$.
\end{remarks}

\begin{examples} \exlabel{modulek}
$(1)$ Any Jacobi algebra $A$ has a canonical structure of a Jacobi
$A$-bimodule via: $x \triangleleft a := xa$ and $x \leftharpoonup
a := [x, \, a]$, for all $x$, $a \in A$. Indeed, for these
actions, axioms \equref{Jmod1} and \equref{Jmod2} are both
equivalent to the compatibility condition \equref{jac1} defining a
Jacobi algebra. On the other hand, axiom \equref{Jmod3} is
equivalent to
$$
[x, \, ab] = [xa, \, b] + [xb, \, a] - [xab, \, 1_A]
$$
which follows trivially by applying several times the
compatibility condition \equref{jac1}.

$(2)$ Any vector space $V$ can be viewed as a Jacobi $A$-bimodule
via the trivial actions: $x \triangleleft a := x$ and $x
\leftharpoonup a := 0$, for all $x \in V$, $a \in A$. We shall
denote by $V_0$ the vector space $V$ equipped with these trivial
actions.

$(3)$ There exist a bijection between the set of all right Jacobi
$A$-module structures $(\triangleleft, \, \leftharpoonup)$ that
can be defined on $k$ and the set of all pairs $(\Lambda, \lambda)
\in A^* \times A^*$, consisting of an algebra map $\Lambda: A \to
k$ and a Lie algebra map $\lambda : A \to k$ satisfying the
following two compatibility conditions for any $a$, $b\in A$:
\begin{eqnarray}
\lambda(ab) = \lambda(a)\, \Lambda (b) + \lambda(b) \, \Lambda(a)
- \lambda(1_A) \, \Lambda (ab), \quad \Lambda([a, \, b]) =
\Lambda(a) \, \Lambda([1_A, \, b]) \eqlabel{modflag2b}
\end{eqnarray}
The bijection is given such that the Jacobi $A$-module structure
$(\triangleleft, \, \leftharpoonup)$ on $k$ associated to
$(\Lambda, \lambda) \in A^* \times A^*$ is defined by $x
\triangleleft a := x \, \Lambda(a)$ and $x \leftharpoonup a := x
\,  \lambda(a)$, for all $a\in A$, $x\in k$. The actions
associated to such a pair $(\Lambda, \lambda)$ endow $k$ with a
Jacobi $A$-bimodule structure since \equref{Jmod3} also holds
thanks to the first compatibility condition of \equref{modflag2b}.
\end{examples}

The proof of the next lemma provides a motivation for introducing
axiom \equref{Jmod3} in \deref{moduleJac}: without it the linear
dual of a right Jacobi $A$-module is not necessarily a right
Jacobi $A$-module.

\begin{lemma}\lelabel{dualul}
Let $A$ be a Jacobi algebra and $(V, \, \triangleleft, \,
\leftharpoonup) \in {\mathcal J} {\mathcal B}^A_A$ a Jacobi
$A$-bimodule. Then the $k$-linear dual $V^* = (V^*, \,
\blacktriangleleft, \, \curvearrowleft) \in {\mathcal J} {\mathcal
B}^A_A$ is a Jacobi $A$-bimodule, where the actions
$\blacktriangleleft$, $\curvearrowleft$ are defined for any
$v^*\in V^*$, $a\in A$ and $x\in V$ by:
\begin{equation}
(v^* \blacktriangleleft a) (x) := v^* (x \triangleleft a), \qquad
(v^* \curvearrowleft a) (x) := - v^* (x \leftharpoonup a)
\end{equation}
In particular, there exists a well defined contravariant functor
$(-)^* : {\mathcal J} {\mathcal B}^A_A \to {\mathcal J} {\mathcal
B}^A_A$ which for finite dimensional Jacobi $A$-bimodules provides
a duality of categories.
\end{lemma}

\begin{proof}
$(V^*, \, \curvearrowleft)$ is a right Lie $A$-module and $(V^*,
\, \blacktriangleleft)$ is a right $A$-module since $A$ is a
commutative algebra. A straightforward computation shows that the
compatibility condition \equref{Jmod1} holds for $(V^*, \,
\blacktriangleleft, \, \curvearrowleft)$ if and only if
\equref{Jmod3} holds for $(V, \, \triangleleft, \,
\leftharpoonup)$ and similar \equref{Jmod3} holds for $(V^*, \,
\blacktriangleleft, \, \curvearrowleft)$ if and only if
\equref{Jmod1} holds for $(V, \, \triangleleft, \,
\leftharpoonup)$. Finally, \equref{Jmod2} for $(V^*, \,
\blacktriangleleft, \, \curvearrowleft)$ is equivalent to
\equref{Jmod2} for $(V, \, \triangleleft, \, \leftharpoonup)$.
\end{proof}

We shall view the Jacobi algebra $A$ as an object in ${\mathcal J}
{\mathcal B}^A_A$ via the actions $x \triangleleft a := xa$ and $x
\leftharpoonup a := [x, \, a]$, for all $x$, $a\in A$. It follows
from \leref{dualul} that the $k$-linear dual $A^*$ is a Jacobi
$A$-bimodule whose actions take the following form:
\begin{equation}\eqlabel{actdualA}
(a^* \blacktriangleleft a) (x) := a^* (ax), \qquad
(a^* \curvearrowleft a) (x) := a^* ([a, \, x])
\end{equation}
From now on we will see $A$ and $A^*$ as objects in ${\mathcal J}
{\mathcal B}^A_A$ via the above structures.

\begin{definition}\delabel{deffrojac}
A Jacobi algebra $A$ is called \emph{Frobenius} if there exists an
isomorphism $A \cong A^*$ in ${\mathcal J} {\mathcal B}^A_A$, i.e.
an isomorphism of right $A$-modules and right Lie $A$-modules.
\end{definition}

Any Frobenius Jacobi algebra is Frobenius as an associative
algebra and as a Lie algebra. Now we introduce the concept of
integral on a Jacobi algebra: as in the case of Hopf
algebras \cite{radford}, integrals will be intimately related
to the property of being Frobenius.

\begin{definition}\delabel{integrale}
An \emph{integral} on a Jacobi algebra $A$ is an element $\nu \in A^*$ such that
\begin{equation} \eqlabel{defint}
\nu \bigl( \left[a, \, b\right] \, c \bigl) = \nu \bigl( a\, \left[b, \, c\right] \bigl)
\end{equation}
for all $a$, $b$, $c\in A$. We denote by $\int_A$ the space of
integrals on a Jacobi algebra $A$. An integral $\nu$ is called
\emph{non-degenerate} if: $\nu (a x ) = 0$, for all $x\in A$
implies $a = 0$.
\end{definition}

If $\nu$ is an integral on $A$ then by taking $b = c = 1_A$ in
\equref{defint} we obtain that $\nu \bigl( [a, \, 1] \bigl) = 0$.

\begin{examples} \exlabel{exeint}
$(1)$ If $A$ is an abelian Jacobi algebra, then $\int_A = A^*$.

$(2)$ Let $A$ be a unital Poisson algebra and $\nu \in \int_A$.
Then by taking $c = 1_A$ in \equref{defint} we obtain that $\nu
(x) = 0$, for any $x \in A' := [A, \, A]$. In particular, it
follows that $\int_A = 0$, for any Poisson algebra $A$ which is
perfect as a Lie algebra.

$(3)$ The trace map ${\rm Tr} : {\rm M}_n(k) \to k$ satisfies
\equref{defint} since ${\rm Tr} \bigl([A, \, B] \,C
\bigl) = {\rm Tr} \bigl( A \, [B, \, C] \bigl)$, for all $n\times
n$-matrices $A$, $B$, $C$. Thus, for any finite dimensional Jacobi
algebra $A$ which is also a subalgebra of ${\rm M}_n(k)$ and a Lie
subalgebra of ${\rm gl} (n, \, k)$, the restriction of the trace
map ${\rm Tr}_{|\, A}$ is an integral on $A$.

$(4)$ Let $u\in U(A)$ be an invertible element of a Jacobi algebra
$A$. Then the map
$$
\int_A  \longrightarrow  \int_{A_u}, \qquad \nu \mapsto \nu_u :=
\nu(u^2 \, \cdot)
$$
is a bijection between the integrals on $A$ and those on the
$u$-conformal deformation $A_u$. Indeed, we can easily prove that
$\nu_u (a) := \nu (u^2 a)$ is an integral on $A_u$ for any
integral $\nu$ on $A$ and the inverse of the map $\nu \mapsto
\nu_u$ is given by $\mu \mapsto \mu(u^{-2} \, \cdot)$ -- we note
that $(A_u)_{u^{-1}} = A$. We can also prove that an integral
$\nu$ on $A$ remains an integral on $A_u$ if and only if $2 \, [u,
\, x] = 2u \, [1_A, \, x]$, for all $x\in A$.

$(5)$ Let $J^{2,4}$ be the Jacobi algebra of \exref{dim2}. Then
$\nu$ is an integral on $J^{2,4}$ if and only if $\nu (1) =
\alpha$ and $\nu (x) = 0$, for some $\alpha \in k$, i.e.
$\int_{J^{2,4}} \cong k$.

$(6)$ Let $J^{3}_{11}$ be the Jacobi algebra from Table \ref{10}.
Then $\nu$ is an integral on $J^{3}_{11}$ if and only if $\nu (1)
= \alpha$, \, $\nu (x) = \nu (y) = 0$, for some $\alpha \in k$,
i.e. $\int_{J^{3}_{11}} \cong k$.
\end{examples}

\begin{lemma} \lelabel{cesuntint}
Let $A$ be a Jacobi algebra. There exists a bijection between
$\int_A$ and the space of all (symmetric) associative, invariant
bilinear forms on $A$, i.e. bilinear maps $B : A \times A \to k$
satisfying the following compatibility conditions for any $a$,
$b$, $c\in A$:
\begin{equation} \eqlabel{formemetric}
B(ab, \, c) = B(a, \, bc), \qquad B([a, \, b], \, c) = B(a, \, [b,
\, c])
\end{equation}
\end{lemma}

\begin{proof}
Any bilinear form $B$ satisfying \equref{formemetric} is symmetric
since $A$ is a commutative algebra: $B(a, \, b) = B(1_A, \, ab) =
B(1_A, \, ba) = B(b, \, a)$, for all $a$, $b\in A$. If $\nu \in
\int_A$ is an integral on $A$, then $B_{\nu} (a, \, b) := \nu
(ab)$ is an associative, invariant, bilinear form on $A$;
conversely, if $B$ is such a form then $\nu_B : A \to k$, $\nu_B
(a) := B(a, \, 1_A) = B(1_A, \, a)$ is an integral on $A$ and the
correspondence $(\nu \mapsto B_{\nu}, \, B\mapsto \nu_B)$ is
clearly bijective.
\end{proof}

The equivalences $(1)-(2)-(4)$ in the theorem below can be seen as
the Jacobi versions of the classical characterization of Frobenius
algebras (\cite{CMZ2002}). The equivalence $(1)-(3)$ is the Jacobi
counterpart of Sullivan's theorem \cite{radford} for Hopf
algebras: a Hopf algebra $H$ is co-Frobenius if and only if there
exists a non-zero integral on $H$.

\begin{theorem}\thlabel{caracfrj}
Let $A$ be a Jacobi algebra. The following are equivalent:

$(1)$ $A$ is a Frobenius Jacobi algebra;

$(2)$ there exists a nondegenerate associative, invariant,
bilinear form on $A$;

$(3)$ there exists a nondegenerate integral on the Jacobi algebra
$A$;

Furthermore, if $A$ is finite dimensional then the above
statements are also equivalent to:

$(4)$ there exists a pair $(\nu, \, e = \sum e^1 \ot e^2)$,
consisting of an integral $\nu \in \int_A$ on $A$ and an element
$e = \sum e^1 \ot e^2 \in A\ot A$ such that for any $a\in A$ we
have:
\begin{equation} \eqlabel{casimir}
\sum ae^1 \ot e^2 = \sum e^1 \ot e^2 a, \qquad \sum \nu (e^1) e^2 = \sum e^1 \nu (e^2) = 1_A
\end{equation}
We call such a pair $(\nu, \, e = \sum e^1 \ot e^2) \in \int_A
\times \, (A\ot A)$ a \emph{Jacobi-Frobenius pair} and $\omega_A
:= \sum e^1 e^2 \in A$ the \emph{Euler-Casimir} element of $A$.
\end{theorem}

\begin{proof}
$(1) \Leftrightarrow (2)$ Follows from the one-to-one
correspondence between the set of all $k$-linear isomorphisms $f:
A \to A^*$ and the set of all nondegenerate bilinear forms $B : A
\times A \to k$ given by the two-sided formula $f (a) (b) :=: B(a,
b)$, for all $a$, $b\in A$. Under this bijection, the right
$A$-module (resp. right Lie $A$-module) maps $f: A \to A^*$
correspond to those bilinear forms $B: A\times A \to k$ that
satisfy the left (resp. right) hand part of \equref{formemetric}.

$(2) \Leftrightarrow (3)$ Follows from \leref{cesuntint} since
under the bijective correspondence $(\nu \mapsto B_{\nu}, \,
B\mapsto \nu_B)$ nondegenerate forms on $A$ correspond to
nondegenerate integrals on $A$. We note that the isomorphism of
right Jacobi $A$-modules $f = f_{\nu} : A \to A^*$ associated to a
nondegenerate integral $\nu$ is defined by $f_{\nu} (a) := \nu
\blacktriangleleft a$, i.e. $f_{\nu} (a) (x) = \nu (ax)$, for all
$a$, $x\in A$. Conversely any nondegenerate integral $\nu = \nu_f$
on $A$ arises from an isomorphism of right Jacobi $A$-modules $f:
A \to A^*$ via $\nu_f := f(1_A)$.

$(1) \Leftrightarrow (4)$ Let $\{e_i, \, e_i^* \, | \, i = 1,
\cdots, n \}$ be a dual basis of $A$ and $f: A \to A^*$ and
isomorphism of Jacobi $A$-bimodules. Then, $\nu := f(1_A)$ is a
Jacobi integral and by a straightforward computations we can see
that $(f(1_A), \, e:= \sum_{i=1}^n \, e_i \ot f^{-1} (e_i^*)$ is a
Jacobi-Frobenius pair. Conversely, if $(\nu, \, e = \sum e^1 \ot
e^2)$ is a Jacobi-Frobenius pair then the map $f = f_{\nu}: A \to
A^*$, $f (a) (b) := \nu (ab)$, for all $a$, $b\in A$ is an
isomorphism of Jacobi $A$-bimodules with the inverse $f^{-1} : A^*
\to A$ given by $f^{-1} (a^*) := \sum a^* (e^1) \, e^2$, for all
$a^* \in A^*$.
\end{proof}

\begin{remarks} \relabel{comfrob}
$(1)$ Integrals on a unital Poisson algebra $A$ are defined
exactly as in \deref{integrale} and a Poisson algebra $A$ is
called Frobenius if there exists an isomorphism of Poisson modules
$A \cong A^*$. The Poisson version of \thref{caracfrj} has the
same statement and we can rephrase this characterization by saying
that a Poisson algebra is Frobenius if and only if it is Frobenius
when viewed as a Jacobi algebra.

$(2)$ Let $\nu \in \int_A $ be a nondegenerate integral on a
finite dimensional Jacobi algebra $A$. It follows from
\thref{caracfrj} that $(A^*, \blacktriangleleft)$ is free as a
right $A$-module having $\{\nu \}$ as a basis, i.e. $A^* = \nu
\blacktriangleleft A$. This is the Jacobi version of a well-know
result for Hopf algebras (\cite[Corollary 10.6.6]{radford}).

$(3)$ Using $(1) \Leftrightarrow (2)$ of \thref{caracfrj} we
obtain that the abelian Jacobi algebra $A_0$ is Frobenius as a
Jacobi algebra if and only if $A$ is Frobenius as an associative
algebra.

$(4)$ Let $A$ be a finite dimensional Frobenius Jacobi algebra
with the Jacobi-Frobenius pair $(\nu, \, e = \sum e^1 \ot e^2)$.
In the case of associative Frobenius algebras the element
$\omega_A := \sum e^1 e^2 \in A$ does not depend on the choice of
the Jacobi-Frobenius pair and it is called in \cite{abr} the
\emph{Euler (characteristic)} element of $A$ since it is the
algebraic counterpart of the Euler class of a connected, oriented,
finite dimensional and compact manifold. In the context of finite
dimensional Lie algebras equipped with an invariant bilinear form,
the same element viewed in the enveloping algebra, is called the
Casimir element \cite[Proposition 11]{bour}. These explain the
terminology used in $(4)$ of \thref{caracfrj}.
\end{remarks}

\begin{examples}\exlabel{exjafrone}
$(1)$ Let $C_n$ be the cyclic group of order $n$ and $k$ a field
such that ${\rm char} \, (k) = p \, | \, n-1 $. Let $A :=k[C_n]$
be the Jacobi algebra with the bracket $[c^i, \, c^j] := (j-i)
c^{j+i}$, for all $i$, $j = 0, \cdots, n-1$. Then $\int_{k[C_n]} =
0$. Indeed, let $\nu \in \int_{k[C_n]}$. By applying
\equref{defint} in $a = b := c^0 = 1$ and $c := c^i$, we obtain
$\nu (c^i) = 0$, for any $i = 0, \cdots, n-1$. Since any group
algebra $k[G]$ of a finite group $G$ is Frobenius \cite{kadison},
we obtain that $k[C_n]$ is Frobenius as an associative algebra and
using \thref{caracfrj} we obtain that it is not Frobenius as a
Jacobi algebra.

$(2)$ Let $u \in U(A)$ be an invertible element of a Jacobi
algebra $A$. Then, the $u$-conformal deformation $A_u$ is a
Frobenius Jacobi algebra if and only if $A$ is a Frobenius Jacobi
algebra. The result follows from \thref{caracfrj} and $(4)$ of
\exref{exeint} since the bijection there preserves the
non-degeneration of integrals.

$(3)$ The equivalence $(1)-(3)$ of \thref{caracfrj} provides a
very efficient criterion for deciding when a given Jacobi algebra
is Frobenius. For instance, the only $2$-dimensional Frobenius
Jacobi algebras are $J^{2, 1}$, $J^{2, 3}$ and $J^{2}_d$: for each
of them the linear map $\nu$ defined by $\nu (1) := 0$ and $\nu
(x) : = 1$ is a non-degenerate integral. In the same manner, we
can easily prove that among the eleven types of $3$-dimensional
complex Jacobi algebras given in \exref{dim3} there are only three
Frobenius Jacobi algebras, namely $J^{3}_{2}$, $J^{3}_{3}$ and
$J^{3}_{4}$. For each of them the linear map $\nu$ defined by $\nu
(1) := a_1$, $\nu (x) := a_2$ and $\nu (y) := a_3$ is a
non-degenerate integral, for all $a_1$, $a_2$, $a_3 \in k$ such
that $a_2 a_3 \neq 0$ and $a_1 \neq a_2 + a_3$.
\end{examples}

We end the section with the following two questions:

\textbf{Question 1:} \emph{Does there exist a finite dimensional Jacobi
algebra $A$ which is Frobenius both as an associative algebra and
as a algebra Lie, but is not Frobenius as a Jacobi algebra?}

\textbf{Question 2:} \emph{Let $A$ be a finite dimensional Frobenius
Jacobi algebra such that the Euler-Casimir element $\omega_A$ is
invertible in $A$. Is every right Jacobi $A$-module completely
reducible (i.e. is $A$ semisimple as a Jacobi algebra)?}

\section{Unified products for Jacobi algebras}\selabel{unifiedprod}
In this section we shall answer the extending structures
ES-problem for Jacobi algebras: i.e. we shall describe and
classify all Jacobi algebras containing a given Jacobi algebra $A$
as a subalgebra of a fixed codimension. We start by explaining
what we mean by classification within the ES-problem. Let $A$ be a
Jacobi algebra, $E$ a vector space such that $A$ is a subspace of
$E$ and let $V$ be a fixed complement of $A$ in $E$, i.e. $V$ is a
subspace of $E$ such that $E = A + V$ and $A \cap V = \{0\}$. Let
${\mathcal J} \, (A, \, E)$ be the category whose objects are all
Jacobi algebra structures $(\cdot_E, [- , - ]_E)$ that can be
defined on $E$ such that $A$ becomes a Jacobi subalgebra of $(E,
\cdot_E, [- , - ]_E)$. A morphism $\varphi: (\cdot_E, [- , - ]_E)
\to (\cdot'_E, [- , - ]'_E)$ in ${\mathcal J} \, (A, \, E)$ is a
morphism of Jacobi algebras $\varphi: (E, \cdot_E, \{-, \, -\}_E)
\to (E, \cdot'_E, \{-, \, -\}^{'}_E)$ which stabilizes $A$ and
co-stabilizes $V$, i.e. the diagram
\begin{eqnarray} \eqlabel{diagrama1}
\xymatrix {& A \ar[r]^{i} \ar[d]_{Id} & {E}
\ar[r]^{\pi} \ar[d]^{\varphi} & V \ar[d]^{Id}\\
& A \ar[r]^{i} & {E}\ar[r]^{\pi } & V}
\end{eqnarray}
is commutative, where $\pi : E \to V$ is the canonical projection
of $E = A + V$ on $V$ and $i: A \to E$ is the inclusion map. In
this case we say that the Jacobi algebra structures $(\cdot_E, [-
, - ]_E)$ and $(\cdot'_E, [- , - ]'_E)$ on $E$ are
\emph{cohomologous} and we denote this by $(\cdot_E, \{-, \,
-\}_E) \approx (\cdot'_E, \{-, \, -\}^{'}_E)$. Any linear map
$\varphi$ which makes diagram \equref{diagrama1} commutative is
bijective, thus the category ${\mathcal J} \, (A, \, E)$ is a
groupoid, i.e. any morphism is an isomorphism. In particular, we
obtain that $\approx$ is an equivalence relation on the set of
objects of ${\mathcal J} \, (A, \, E)$ and we denote by ${\rm
Extd}_{{\mathcal J}} \, (E, \, A)$ the set of all equivalence
classes, i.e. ${\rm Extd}_{{\mathcal J}} \, (E, \, A) := {\mathcal
J} \, (A, \, E)/\approx$. ${\rm Extd}_{{\mathcal J}} \, (E, \, A)$
is the classifying object for the ES-problem: it classifies all
Jacobi algebra structures that can be defined on $E$ containing
$A$ as a Jacobi subalgebra up to an isomorphism that stabilizes
$A$ and co-stabilizes $V$. The answer to the ES-problem will be
provided by explicitly computing ${\rm Extd}_{{\mathcal J}} \, (E,
\, A)$ for a given Jacobi algebra $A$ and a vector space $E$. From
geometrical point of view this means to give the decomposition of
the groupoid ${\mathcal J} \, (A, \, E)$ into connected components
and to indicate a 'point' in each such component. The main result
of this section proves that ${\rm Extd}_{{\mathcal J}} (E, \, A)$
is parameterized by a non-abelian cohomological type object
${\mathcal J}{\mathcal H}^{2} \, (V, \, A)$ that will be
explicitly constructed and the bijection between ${\mathcal
J}{\mathcal H}^{2} \, (V, \, A)$ and ${\rm Extd}_{{\mathcal J}}
(E, \, A)$ will be indicated.

\begin{definition}\delabel{exdatum}
Let $A$ be a Jacobi algebra and $V$ a vector space. An
\textit{extending datum of $A$ through $V$} is a system $\Upsilon
(A, V) = \bigl(\triangleleft, \, \triangleright, \, f, \, \cdot,
\, \leftharpoonup, \, \rightharpoonup, \, \theta, \, \{-, \, -\}
\bigl)$ consisting of eight bilinear maps
\begin{eqnarray*}
&& \triangleleft : V \times A \to V, \quad \triangleright : V
\times A \to A, \quad f: V\times V \to A, \quad \cdot \, : V\times
V \to V \\
&& \leftharpoonup: V \times A \to V, \quad \rightharpoonup : V
\times A \to A, \quad \theta: V\times V \to A, \quad \{-, \, -\} :
V\times V \to V
\end{eqnarray*}

Let $\Upsilon (A, V) = \bigl(\triangleleft, \, \triangleright, \,
f, \, \cdot, \, \leftharpoonup, \, \rightharpoonup, \, \theta, \,
\{-, \, - \} \bigl)$ be an extending datum of a Jacobi algebra $A$
through a vector space $V$. We denote by $ A \, \ltimes_{\Upsilon
(A, V)} V = A \, \ltimes V$ the vector space $A \, \times V$
together with the multiplication $\bullet$ and the bracket $[-, \,
-]$ defined by:
\begin{eqnarray}
(a, \, x) \bullet (b, \, y) &:=& \bigl( ab + x \triangleright b +
y \triangleright a + f(x, y), \,\, x\triangleleft b + y \triangleleft a  +
x \cdot y \bigl) \eqlabel{multunifJ} \\
\left [(a, x), \, (b, y) \right ] &:=& \bigl( [a, \, b] + x
\rightharpoonup b - y \rightharpoonup a + \theta (x, y), \,\,
x\leftharpoonup b - y \leftharpoonup a + \{x, \, y \} \bigl)
\eqlabel{brackunifJ}
\end{eqnarray}
for all $a$, $b \in A$ and $x$, $y \in V$. The object $A\ltimes V$
is called the \textit{unified product} of $A$ and $\Upsilon (A,
V)$ if it is a Jacobi algebra with the multiplication defined by
\equref{multunifJ}, the unit $(1_A, \, 0_V)$ and the bracket given
by \equref{brackunifJ}. In this case the extending datum $\Upsilon
(A, V)$ is called a \textit{Jacobi extending structure} of $A$
through $V$.
\end{definition}

The next theorem provides the necessary and sufficient conditions
that need to be fulfilled by an extending datum $\Upsilon (A, V)$
such that $A \ltimes V$ is a unified product.

\begin{theorem}\thlabel{1}
Let $A = (A, \, m_A, \, [-, \, -])$ be a Jacobi algebra, $V$ a
vector space and $\Upsilon (A, V) = \bigl(\triangleleft, \,
\triangleright, \, f, \, \cdot, \, \leftharpoonup, \,
\rightharpoonup, \, \theta, \, \{-, \, - \} \bigl)$ an extending
datum of $A$ through $V$. Then $A \ltimes V$ is a unified product
if and only if the following compatibilities hold:

$(J0)$ \, $\bigl(\triangleleft, \, \triangleright, \, f, \,
\cdot)$ is an algebra extending system of the associative algebra
$A$ through $V$ and $\bigl(\leftharpoonup, \, \rightharpoonup, \,
\theta, \, \{-, \, - \} \bigl)$ is a Lie extending system of the
Lie algebra $A$ through $V$;

$(J1)$ \, $x \rightharpoonup (ab) = (x \rightharpoonup a) \, b +
(x \leftharpoonup a) \triangleright b + a \, (x \rightharpoonup b)
+ (x \leftharpoonup b) \triangleright a - ab \, (x \rightharpoonup
1_A) \, - $

$ \qquad \qquad \qquad \,\,\,\,\,\,\, - \, (x \leftharpoonup 1_A)
\triangleright (ab) $

$(J2)$ \, $x \leftharpoonup (ab) = (x \leftharpoonup a)
\triangleleft b + (x \leftharpoonup b) \triangleleft a - (x
\leftharpoonup 1_A) \triangleleft (ab)$

$(J3)$ \, $x \, \triangleright [a, \, b] = [x \triangleright a, \,
b] + (x \triangleleft a) \rightharpoonup b - a \, (x
\rightharpoonup b) - (x \leftharpoonup b) \triangleright a + (x
\triangleright a) [1_A, \, b] + (x \triangleleft a) \triangleright
[1_A, \, b]$

$(J4)$ \, $x  \triangleleft [a, \, b] = (x \triangleleft a)
\leftharpoonup b - (x  \leftharpoonup b) \triangleleft a + (x
\triangleleft a) \triangleleft [1_A, \, b] $

$(J5)$ \, $ \{x, \, y\} \triangleright a  = \theta (x
\triangleleft a, \, y) - a \, \theta (x, \, y) + f (y
\leftharpoonup a, \, x) - f (x \triangleleft a, \, y
\leftharpoonup 1_A)  -  y \rightharpoonup (x \triangleright a) + $

\,\, $ \qquad \,\, + \, x \triangleright (y \rightharpoonup a) -
(x \triangleright a) (y \rightharpoonup 1_A) - (x \triangleleft a)
\triangleright (y \rightharpoonup 1_A) - (y \leftharpoonup 1_A)
\triangleright (x \triangleright a) $

$(J6)$ \, $ \{x, \, y\} \triangleleft a = \{x \triangleleft a, \,
y \} - y \leftharpoonup (x \triangleright a) + x \triangleleft (y
\rightharpoonup a) + (y \leftharpoonup a) \cdot x \, - $

\,\, $ \qquad \,\, - (x \triangleleft a) \triangleleft (y
\rightharpoonup 1_A) - (y \leftharpoonup 1_A) \triangleleft (x
\triangleright a) - (x \triangleleft a) \cdot (y \leftharpoonup
1_A) $

$(J7)$ \, $ (x \cdot y) \rightharpoonup a = x \triangleright (y
\rightharpoonup a) + y \triangleright (x \rightharpoonup a) + f (x
\leftharpoonup a, \, y) + f (x, \, y \leftharpoonup a) - [f(x, \,
y), \, a] \, - $

\,\, $  \qquad \,\, - f(x, \, y) [1_A, \, a] - (x \cdot y)
\triangleright [1_A, \, a] $

$(J8)$ \, $ (x \cdot y) \leftharpoonup  a =  x \cdot (y
\leftharpoonup a) + (x \leftharpoonup a) \cdot y + x \triangleleft
(y \rightharpoonup a) + y  \triangleleft (x \rightharpoonup a) -
(x\cdot y) \triangleleft [1_A, \, a] $

$(J9)$ \, $ \theta (x \cdot y, \, z) = x \triangleright \theta (y,
\, z) + y \triangleright \theta (x, \, z) + z \rightharpoonup f(x,
\, y) + f \bigl(\{x, \, z\}, \, y \bigl) + f \bigl(x, \, \{y, \,
z\}\bigl) \, + $

\,\, $  \qquad \,\,  + f(x, \, y) (z \rightharpoonup 1_A) +
(x\cdot y) \triangleright (z \rightharpoonup 1_A) + (z
\leftharpoonup 1_A) \triangleright f(x, \, y) + f(x\cdot y, \, z
\leftharpoonup 1_A)$

$(J10)$ $ \{x\cdot y, \, z \} = x \cdot \{y, \, z \} + \{x, \, z\}
\cdot y + z \leftharpoonup f(x, \, y) + x \triangleleft \theta (y,
\, z) + y \triangleleft \theta (x, \, z) \, + $

\,\, $ \qquad \,\, + \, (x\cdot y) \triangleleft (z
\rightharpoonup 1_A) + (z\leftharpoonup 1_A) \triangleleft f(x, \,
y) + (x \cdot y) \cdot (z \leftharpoonup 1_A) $

for all $a$, $b \in A$, $x$, $y$, $z \in V$.
\end{theorem}

\begin{proof}
We have already noticed in Preliminaries that $(A \ltimes \, V, \,
\bullet)$ is a commutative algebra with unit $(1_A, 0)$ if and
only if $\bigl(\triangleleft, \, \triangleright, \, f, \, \cdot)$
is an algebra extending system of the algebra $A$ through $V$ and
$(A \ltimes \, V, \, [-, \, - ])$ is a Lie algebra if and only
$\bigl(\leftharpoonup, \, \rightharpoonup, \, \theta, \, \{-, \, -
\} \bigl)$ is a Lie extending system of the Lie algebra $A$
through $V$. These are the assumptions from (J0) which from now on
we assume to be fulfilled. Then, $(A \ltimes \, V, \, \bullet, \,
[-, \, - ])$ is a Jacobi algebra if and only if the following
compatibility holds for any $a$, $b$, $c\in A$ and $x$, $y$, $z\in
V$
\begin{eqnarray*}
\left[ (a, x) \bullet (b, y), \, (c, z)  \right] &=& (a, x)
\bullet \left[ (b, y), \, (c, z) \right] + \left [ (a,
x), \, (c, z) \right] \bullet (b, y) - \\
&& - \, (a, x) \bullet (b, y) \bullet \bigl( \left[1_A, \, c
\right] - z \rightharpoonup 1_A, \, -z \leftharpoonup 1_A \bigl)
\eqlabel{005}
\end{eqnarray*}
If we denote the last equation by (J), the proof relies on a
detailed analysis of this identity. Since in $A \times V$ we have
$(a, x) = (a, 0) + (0, x)$ it follows that (J) holds if and only
if it holds for all generators of $A \times V$, i.e. for the set
$\{(a, \, 0) ~|~ a \in A\} \cup \{(0, \, x) ~|~ x \in V\}$.
However, since the computations are rather long but
straightforward we will only indicate the main steps of the proof,
the details being available upon request. First, we notice that
(J) holds for the triple $(a, 0)$, $(b, 0)$, $(c, 0)$ since $A$ is
a Jacobi algebra. The left hand side of (J) evaluated in $(a, 0)$,
$(b, 0)$, $(0, x)$ is equal to $\bigl( - x \rightharpoonup (ab),
\, - x \leftharpoonup (ab) \bigl)$ while the right hand side of
(J) evaluated in the same triple comes down to:
\begin{eqnarray*}
&& - \bigl( (x \rightharpoonup a) \, b + (x \leftharpoonup a)
\triangleright b, \,\, (x \leftharpoonup a) \triangleleft b \bigl)
\, - \, \bigl( a\, (x \rightharpoonup b) + (x\leftharpoonup b)
\triangleright a, \,\, (x \leftharpoonup b) \triangleleft a \bigl)
\, + \\
&& + \, \bigl( ab \, (x\rightharpoonup 1_A) + (x \leftharpoonup
1_A) \triangleright (ab), \,\, (x \leftharpoonup 1_A)
\triangleleft (ab) \bigl)
\end{eqnarray*}
We obtain that (J) holds for the triple $(a, 0)$, $(b, 0)$, $(0,
x)$ if and only if (J1) and (J2) hold. Similar computations show
the following: (J) holds for the triple $(a, 0)$, $(0, x)$, $(b,
0)$ if and only if (J3) and (J4) hold; (J) holds for the triple
$(a, 0)$, $(0, x)$, $(0, y)$ if and only if (J5) and (J6) hold;
(J) holds for the triple $(0, x)$, $(0, y)$, $(a, 0)$ if and only
if (J7) and (J8) hold and finally, (J) holds for the triple $(0,
x)$, $(0, y)$, $(0, z)$ if and only if (J9) and (J10) hold.
Moreover, since $A$ is commutative, we observe that the Jacobi
compatibility \equref{jac1} holds for the triple $(a, b, c)$ if
and only if it holds for the triple $(b, a, c)$. Based on this
remark we obtain that (J) holds for the triple $(0, x)$, $(a, 0)$,
$(b, 0)$ whenever it holds for $(a, 0)$, $(0, x)$, $(b, 0)$ and
(J) holds for the triple $(0, x)$, $(a, 0)$, $(0, y)$ whenever it
holds for $(a, 0)$, $(0, x)$, $(0, y)$. The proof is now finished.
\end{proof}

From now on a Jacobi extending structure of a Jacobi algebra $A$
through a vector space $V$ will be viewed as a system $\Upsilon
(A, V) = \bigl(\triangleleft, \, \triangleright, \, f, \, \cdot,
\, \leftharpoonup, \, \rightharpoonup, \, \theta, \, \{-, \, - \}
\bigl)$ satisfying the compatibility conditions (J0)-(J10). We
denote by ${\mathcal J} {\mathcal E} (A, V)$ the set of all Jacobi
algebra extending structures of $A$ through $V$. \thref{1} takes a
simplified form at the level of Poisson algebras:

\begin{corollary}\colabel{1}
Let $A = (A, \, m_A, \, [-, \, -])$ be a unital Poisson algebra,
$V$ a vector space and $\Upsilon (A, V) = \bigl(\triangleleft, \,
\triangleright, \, f, \, \cdot, \, \leftharpoonup, \,
\rightharpoonup, \, \theta, \, \{-, \, - \} \bigl)$ an extending
datum of $A$ through $V$. Then $A \ltimes V = (A\ltimes V,
\bullet, [-, \, -])$ is a Poisson algebra with $(1_A, \, 0_V)$ as
a unit if and only if the following compatibilities hold for any
$a$, $b \in A$, $x$, $y$, $z \in V$:

$(P0)$ \, $\bigl(\triangleleft, \, \triangleright, \, f, \,
\cdot)$ is an algebra extending system of the associative algebra
$A$ trough $V$ and $\bigl(\leftharpoonup, \, \rightharpoonup, \,
\theta, \, \{-, \, - \} \bigl)$ is a Lie extending system of the
Lie algebra $A$ trough $V$

$(P1)$ \, $x \rightharpoonup (ab) = (x \rightharpoonup a) \, b +
(x \leftharpoonup a) \triangleright b + a \, (x \rightharpoonup b)
+ (x \leftharpoonup b) \triangleright a $

$(P2)$ \, $x \leftharpoonup (ab) = (x \leftharpoonup a)
\triangleleft b + (x \leftharpoonup b) \triangleleft a $

$(P3)$ \, $x \, \triangleright [a, \, b] = [x \triangleright a, \,
b] + (x \triangleleft a) \rightharpoonup b - a \, (x
\rightharpoonup b) - (x \leftharpoonup b) \triangleright a $

$(P4)$ \, $x  \triangleleft [a, \, b] = (x \triangleleft a)
\leftharpoonup b - (x  \leftharpoonup b) \triangleleft a$

$(P5)$ \, $ \{x, \, y\} \triangleright a  = \theta (x
\triangleleft a, \, y) - a \, \theta (x, \, y) + f (y
\leftharpoonup a, \, x) -  y \rightharpoonup (x \triangleright a)
+ x \triangleright (y \rightharpoonup a)$

$(P6)$ \, $ \{x, \, y\} \triangleleft a = \{x \triangleleft a, \,
y \} - y \leftharpoonup (x \triangleright a) + x \triangleleft (y
\rightharpoonup a) + (y \leftharpoonup a) \cdot x $

$(P7)$ \, $ (x \cdot y) \rightharpoonup a = x \triangleright (y
\rightharpoonup a) + y \triangleright (x \rightharpoonup a) + f (x
\leftharpoonup a, \, y) + f (x, \, y \leftharpoonup a) - [f(x, \,
y), \, a]$

$(P8)$ \, $ (x \cdot y) \leftharpoonup  a =  x \cdot (y
\leftharpoonup a) + (x \leftharpoonup a) \cdot y + x \triangleleft
(y \rightharpoonup a) + y  \triangleleft (x \rightharpoonup a)$

$(P9)$ \, $ \theta (x \cdot y, \, z) = x \triangleright \theta (y,
\, z) + y \triangleright \theta (x, \, z) + z \rightharpoonup f(x,
\, y) + f \bigl(\{x, \, z\}, \, y \bigl) + f \bigl(x, \, \{y, \,
z\}\bigl)$

$(P10)$ $ \{x\cdot y, \, z \} = x \cdot \{y, \, z \} + \{x, \, z\}
\cdot y + z \leftharpoonup f(x, \, y) + x \triangleleft \theta (y,
\, z) + y \triangleleft \theta (x, \, z)$
\end{corollary}

An extending datum $\Upsilon (A, V) = \bigl(\triangleleft, \,
\triangleright, \, f, \, \cdot, \, \leftharpoonup, \,
\rightharpoonup, \, \theta, \, \{-, \, - \} \bigl)$ of a unital
Poisson algebra $A$ through a vector space $V$ satisfying the
axioms (P0)-(P10) is called a \emph{Poisson extending structure of
$A$ through $V$} and we denote by ${\mathcal P} {\mathcal E} (A,
V)$ the set of all Poisson extending structures of $A$ through
$V$.

Let $\Upsilon (A, V) = \bigl(\triangleleft, \, \triangleright, \,
f, \, \cdot, \, \leftharpoonup, \, \rightharpoonup, \, \theta, \,
\{-, \, - \} \bigl) \in {\mathcal J} {\mathcal E} (A, V)$ be a
Jacobi extending structure of a Jacobi algebra $A$ through a
vector space $V$. Then $A$ is a Jacobi subalgebra of the unified
product $A \ltimes V$ through the identification $A \cong i_A (A)
= A \times \{0\}$, where $i_A: A \to A \ltimes V$, $i_A (a) = (a,
\, 0)$ is the canonical injection. Conversely, the following
result provides the answer to the description part of the
ES-problem:

\begin{proposition}\prlabel{classif}
Let $A$ be a Jacobi algebra, $E$ a vector space containing $A$ as
a subspace and $(\ast_E, \, [-, \, -]_E)$ a Jacobi algebra
structure on $E$ such that $A$ is a subalgebra of $(E, \, \ast_E,
\, [-, \, -]_E)$. Then there exists a Jacobi extending structure
$\Upsilon (A, V) = \bigl(\triangleleft, \, \triangleright, \, f,
\, \cdot, \, \leftharpoonup, \, \rightharpoonup, \, \theta, \,
\{-, \, - \} \bigl)$ of $A$ through a subspace $V$ of $E$ and an
isomorphism of Jacobi algebras $(E, \ast_E, [-, \, -]_E ) \cong A
\ltimes V$ that stabilizes $A$ and co-stabilizes $V$.
\end{proposition}

\begin{proof}
Since $k$ is a field, there exists a linear map $p: E \to A$ such
that $p(a) = a$, for all $a \in A$. Then $V := \rm{Ker}(p)$ is a
complement of $A$ in $E$. Using the retraction $p$, we define the
extending datum $\Upsilon (A, V) = \bigl(\triangleleft, \,
\triangleright, \, f, \, \cdot, \, \leftharpoonup, \,
\rightharpoonup, \, \theta, \, \{-, \, - \} \bigl)$ of $A$ through
$V$ by the following formulas for any $a\in A$ and $x$, $y\in V$:
\begin{eqnarray*}
x \triangleright a &:=& p (x \ast_E a), \qquad \,\,\, x \triangleleft a
:= x\ast_E a  - p (x\ast_E a) \\
f(x, y) &:=& p (x \ast_E y), \qquad \,\,\,\,\, x \cdot y := x \ast_E y - p (x \ast_E y) \\
x \rightharpoonup a &:=& p \bigl([x, \, a]_E \bigl), \quad
\,\,\,\,\,
x \leftharpoonup a := [x, \, a]_E - p \bigl([x, \, a]_E\bigl)\\
\theta (x, y) &:=& p \bigl([x, \, y]_E \bigl), \quad \,\,\,\,\, \{x, \,
y\} := [x, \, y]_E - p \bigl([x, \, y]_E\bigl)
\end{eqnarray*}
Then by arguments similar to those used for Lie algebras in
\cite[Theorem 2.4]{am-2013} and associative algebras in
\cite[Theorem 2.4]{am-2013d} we can prove that $\Upsilon (A, V) =
\bigl(\triangleleft, \, \triangleright, \, f, \, \cdot, \,
\leftharpoonup, \, \rightharpoonup, \, \theta, \, \{-, \, - \}
\bigl)$ is a Jacobi extending structure of $A$ through $V$ and the
linear map $\varphi: A \ltimes V \to (E, \ast_E, [-, \, -]_E )$,
$\varphi(a, x) := a + x$, is an isomorphism of Jacobi algebras
that stabilizes $A$ and co-stabilizes $V$, i.e. the following
diagram is commutative:
\begin{eqnarray*} \eqlabel{diagrama}
\xymatrix {& A \ar[r]^{i} \ar[d]_{Id} & {A \ltimes V}
\ar[r]^{q} \ar[d]^{\varphi} & V \ar[d]^{Id}\\
& A \ar[r]^{i} & {E}\ar[r]^{\pi } & V}
\end{eqnarray*}
where $q: A \ltimes V \to V$, $q (a, x) := x$ is the canonical
projection.
\end{proof}

\prref{classif} reduces the classification of all Jacobi algebra
structures on $E$ that contain $A$ as a subalgebra to the
classification of all unified products $A \ltimes V$, associated
to all Jacobi extending structures $\Upsilon (A, V) =
\bigl(\triangleleft, \, \triangleright, \, f, \, \cdot, \,
\leftharpoonup, \, \rightharpoonup, \, \theta, \, \{-, \, - \}
\bigl)$, for a fixed complement $V$ of $A$ in $E$. First we need
the following technical result:

\begin{proposition} \prlabel{echiaabb}
Let $A$ be a Jacobi algebra, $\Upsilon (A, V) =
\bigl(\triangleleft, \, \triangleright, \, f, \, \cdot, \,
\leftharpoonup, \, \rightharpoonup, \, \theta, \, \{-, \, - \}
\bigl)$ and $\Upsilon' (A, V) = \bigl(\triangleleft', \,
\triangleright', \, f', \, \cdot', \, \leftharpoonup', \,
\rightharpoonup', \, \theta', \, \{-, \, - \}' \bigl)$ two Jacobi
extending structures of $A$ through $V$. Let $A\ltimes V$ and $A
\ltimes' V$ be the unified products associated to the Jacobi
extending structures $\Upsilon (A, V)$ and respectively $\Upsilon'
(A, V)$. The following are equivalent:

$(1)$ There exists $\psi: A \ltimes V \to A \ltimes ' V$ a
morphism of Jacobi algebras which stabilizes $A$ and co-stabilizes
$V$;

$(2)$ $\triangleleft' = \triangleleft$, \, $\leftharpoonup' \, =
\, \leftharpoonup$ and there exists a linear map $r: V \to A$ such
that $\triangleright'$, $f'$, \, $\cdot'$, \, $\rightharpoonup'$,
\,  $\theta'$ and $\{-, \, - \}'$ are implemented by $r$ via the
following formulas for any $a \in A$, $x$, $y \in V$:
\begin{eqnarray*}
x \triangleright ' a &=& x \triangleright a + r(x \triangleleft a) - r(x) \, a, \qquad \,\, \quad x \cdot
' y = x \cdot y  - x \triangleleft r(y) - y \triangleleft r(x)\\
f ' (x, y) &=& f(x,\, y)  + r(x \cdot y) + r(x) r(y) - x \triangleright r(y)
- r\bigl(x \triangleleft r(y)\bigl) - y \triangleright r(x) - r\bigl(y \triangleleft r(x) \bigl)  \\
x \rightharpoonup ' a &=& x \rightharpoonup a + r\bigl(x
\leftharpoonup a \bigl) - \, [r(x), \, a], \qquad \{x, y\}' = \{x,
y\} - x \leftharpoonup r (y) + y \leftharpoonup r(x) \\
\theta'(x, y) &=& \theta (x, y) + r \bigl(\{x, \, y\}\bigl) +
\, [r(x), \, r(y)] + y \rightharpoonup r (x) - x \rightharpoonup r
(y) + \\
&& + \, r\bigl(y \leftharpoonup r (x)\bigl) \, - \, r\bigl(x
\leftharpoonup r (y) \bigl)
\end{eqnarray*}
\end{proposition}

\begin{proof}
There exists a bijection between the set of all linear maps $\psi:
A \times V \to A \times V$ which stabilize $A$ and co-stabilize
$V$ and the set of all linear map $r : V \to A$ given as follows:
$\psi \mapsto r_{\psi}$, where $r_{\psi} (x) := \psi (0, x)$
respectively $r \mapsto \psi_r$, where $\psi_r (a, \, x) := (a +
r(x), \, x)$, for all $a\in A$ and $x\in V$. We denote by $\psi_r
= \psi$, the linear map associated to $r: V \to A$. We prove that
$\psi = \psi_r : A \ltimes V \to A \ltimes ' V$ is a morphism of
Jacobi algebras if and only if the compatibility conditions from
$(2)$ hold. Indeed, first we can see that $\psi \bigl( (a, 0)
\bullet (0, x) \bigl) = \psi (a, 0) \bullet' \, \psi (0, x)$ if
and only if $\triangleleft' = \triangleleft$ and $x \triangleright
' a = x \triangleright a + r(x \triangleleft a)$. Taking these two
compatibilities into account, we can easily see that $\psi \bigl(
(0, x) \bullet (0, y) \bigl) = \psi (0, x) \bullet' \psi (0, y)$
if and only if $x \cdot ' y = x \cdot y  - x \triangleleft r(y) -
y \triangleleft r(x)$ and $f ' (x, y) = f(x,\, y)  + r(x \cdot y)
+ r(x) r(y) - x \triangleright r(y) - r\bigl(x \triangleleft
r(y)\bigl) - y \triangleright r(x) - r\bigl(y \triangleleft r(x)
\bigl)$, for all $x$, $y\in V$. This shows that $\psi = \psi_r : A
\ltimes V \to A \ltimes ' V$ is a morphism of associative algebras
if and only if $\triangleleft' = \triangleleft$ and the first
three compatibility conditions of $(2)$ hold. In a similar
fashion, we can prove that $\psi = \psi_r : A \ltimes V \to A
\ltimes ' V$ is a morphism of Lie algebras if and only if
$\leftharpoonup' \, = \, \leftharpoonup$ and the last three
compatibility conditions of $(2)$ hold.
\end{proof}

The Jacobi algebra morphism $\psi_r : A \times V \to A\times V$
defined in the proof of \prref{echiaabb} is bijective. This allows
us to introduce the following equivalence relation :

\begin{definition}\delabel{echiaabbc}
Let $A$ be a Jacobi algebra and $V$ a vector space. Two Jacobi
extending structures $\Upsilon (A, V) = \bigl(\triangleleft, \,
\triangleright, \, f, \, \cdot, \, \leftharpoonup, \,
\rightharpoonup, \, \theta, \, \{-, \, - \} \bigl)$ and $\Upsilon'
(A, V) = \bigl(\triangleleft', \, \triangleright', \, f', \,
\cdot', \, \leftharpoonup', \, \rightharpoonup', \, \theta', \,
\{-, \, - \}' \bigl)$ are called \emph{cohomologous}, and we
denote this by $\Upsilon (A, V) \approx \Upsilon'(A, V)$, if
$\triangleleft' = \triangleleft$, \, $\leftharpoonup' \, = \,
\leftharpoonup$ and there is a linear map $r: V \to A$ such that
$\triangleright'$, $f'$, \, $\cdot'$, \, $\rightharpoonup'$, \,
$\theta'$ and $\{-, \, - \}'$ are implemented by $r$ via the
formulas given in $(2)$ of \prref{echiaabb}.
\end{definition}

The theoretical answer to the ES-problem for Jacobi algebras now
follows:

\begin{theorem}\thlabel{main1}
Let $A$ be a Jacobi algebra, $E$ a vector space which contains $A$
as a subspace and $V$ a complement of $A$ in $E$. Then $\approx $
is an equivalence relation on the set ${\mathcal J} {\mathcal E}
(A, V)$ of all Jacobi algebra extending structures of $A$ through
$V$. If we denote by ${\mathcal J}{\mathcal H}^{2} \, (V, \, A) :=
{\mathcal J} {\mathcal E} (A, V)/ \approx $, then the
map\footnote{ $\overline{\bigl(\triangleleft, \, \triangleright,
\, f, \, \cdot, \, \leftharpoonup, \, \rightharpoonup, \, \theta,
\, \{-, \, - \} \bigl)}$ denotes the equivalence class of
$\bigl(\triangleleft, \, \triangleright, \, f, \, \cdot, \,
\leftharpoonup, \, \rightharpoonup, \, \theta, \, \{-, \, - \}
\bigl)$ via $\approx$.}
$$
{\mathcal J}{\mathcal H}^{2} \, (V, \, A) \to {\rm
Extd}_{{\mathcal J}} \, (E, \, A), \qquad
\overline{\bigl(\triangleleft, \, \triangleright, \, f, \, \cdot,
\, \leftharpoonup, \, \rightharpoonup, \, \theta, \, \{-, \, - \}
\bigl)} \, \longmapsto \, A \ltimes V
$$
is bijective, where $A \ltimes V$ is the unified product
associated to $A$ and $\bigl(\triangleleft, \, \triangleright, \,
f, \, \cdot, \, \leftharpoonup, \, \rightharpoonup, \, \theta, \,
\{-, \, - \} \bigl)$.
\end{theorem}

\begin{proof}
\prref{echiaabb} proves that $\Upsilon (A, V) \approx \Upsilon'(A,
V)$ if and only if there exists an isomorphism of Jacobi algebras
$\psi: A \ltimes V \to A \ltimes ' V$ which stabilizes $A$ and
co-stabilizes $V$. This shows that $\approx$ is an equivalence
relation on ${\mathcal J} {\mathcal E} (A, V)$. The last part
follows from this observation together with \thref{1} and
\prref{classif}.
\end{proof}

\begin{remark} \relabel{possversES} The Poisson version of
\thref{main1} has the following form. Let $A$ be a unital Poisson
algebra, $E$ a vector space which contains $A$ as a subspace, $V$
a complement of $A$ in $E$ and let ${\mathcal P} {\mathcal E} (A,
V)$ be the set of all Poisson extending structures of $A$ through
$V$ in the sense of \coref{1}. We denote by ${\mathcal P}{\mathcal
H}^{2} \, (V, \, A) := {\mathcal P} {\mathcal E} (A, V)/ \approx
$, where $\approx$ is the equivalence relation on ${\mathcal P}
{\mathcal E} (A, V)$ defined exactly as in \deref{echiaabbc} and
by ${\rm Extd}_{{\mathcal P}} \, (E, \, A)$ the set of all
equivalence classes of isomorphism of all Poisson algebra
structures which can be defined on $E$ that contain and stabilize
$A$ as a Poisson subalgebra and co-stabilize $V$. Then the map
$$
{\mathcal P}{\mathcal H}^{2} \, (V, \, A) \to {\rm
Extd}_{{\mathcal P}} \, (E, \, A), \qquad
\overline{\overline{\bigl(\triangleleft, \, \triangleright, \, f,
\, \cdot, \, \leftharpoonup, \, \rightharpoonup, \, \theta, \,
\{-, \, - \} \bigl)}} \, \longmapsto \, A \ltimes V
$$
is bijective.
\end{remark}

Computing the classifying object ${\mathcal J}{\mathcal H}^{2} \,
(V, \, A)$ for a given Jacobi algebra $A$ and a given vector space
$V$ is a highly nontrivial problem. However, the first step in
computing ${\mathcal J}{\mathcal H}^{2} \, (V, \, A)$ is suggested
by the first part of \deref{echiaabbc}: if two Jacobi extending
structures $\Upsilon (A, V) = \bigl(\triangleleft, \,
\triangleright, \, f, \, \cdot, \, \leftharpoonup, \,
\rightharpoonup, \, \theta, \, \{-, \, - \} \bigl)$ and $\Upsilon'
(A, V) = \bigl(\triangleleft', \, \triangleright', \, f', \,
\cdot', \, \leftharpoonup', \, \rightharpoonup', \, \theta', \,
\{-, \, - \}' \bigl)$ are cohomologous, then we must have
$\triangleleft' = \triangleleft$, $\leftharpoonup' \, = \,
\leftharpoonup$. Thus, in order to compute ${\mathcal J}{\mathcal
H}^{2} \, (V, \, A)$ we can fix the right $A$-module action
$\triangleleft$ and the right Lie $A$-module action
$\leftharpoonup$ such that $(V, \, \triangleleft, \,
\leftharpoonup)$ is a right Jacobi module as defined in
\deref{moduleJac} -- we observe that the compatibility conditions
\equref{Jmod1} and \equref{Jmod2} of \deref{moduleJac} coincide
with axioms (J2) and (J4) of \thref{1} defining the Jacobi
extending structures. Hence, we can decompose the object
${\mathcal J}{\mathcal H}^{2} \, (V, \, A)$ as follows. Let $A$ be
a Jacobi algebra, $V$ a vector space and let $(V, \,
\triangleleft, \, \leftharpoonup)$ be a fixed right Jacobi
$A$-module. Let ${\mathcal J} {\mathcal E}_{(\triangleleft, \,
\leftharpoonup)} (A, V)$ be the set of all \emph{$(\triangleleft,
\, \leftharpoonup)$-Jacobi extending stuctures} of $A$ through
$(V, \, \triangleleft, \, \leftharpoonup)$ which is the set of all
6-tuples $\bigl(\triangleright, \, f, \, \cdot, \,
\rightharpoonup, \, \theta, \, \{-, \, -\} \bigl)$ consisting of
six bilinear maps such that $\bigl(\triangleleft, \,
\triangleright, \, f, \, \cdot, \, \leftharpoonup, \,
\rightharpoonup, \, \theta, \, \{-, \, - \} \bigl)$ is a Jacobi
extending structure of $A$ through $V$. Two elements
$\bigl(\triangleright, \, f, \, \cdot, \, \rightharpoonup, \,
\theta, \, \{-, \, -\} \bigl)$ and $\bigl(\triangleright', \, f',
\, \cdot', \, \rightharpoonup', \, \theta', \, \{-, \, -\}'
\bigl)$ of ${\mathcal J} {\mathcal E}_{(\triangleleft, \,
\leftharpoonup)} (A, V)$ are $(\triangleleft, \,
\leftharpoonup)$-\emph{cohomologous} and we denote this by
$\bigl(\triangleright, \, f, \, \cdot, \, \rightharpoonup, \,
\theta, \, \{-, \, -\} \bigl) \, \approx_l \,
\bigl(\triangleright', \, f', \, \cdot', \, \rightharpoonup', \,
\theta', \, \{-, \, -\}' \bigl)$ if $\bigl(\triangleleft, \,
\triangleright, \, f, \, \cdot, \, \leftharpoonup, \,
\rightharpoonup, \, \theta, \, \{-, \, - \} \bigl) \approx
\bigl(\triangleleft, \, \triangleright', \, f', \, \cdot', \,
\leftharpoonup, \, \rightharpoonup', \, \theta', \, \{-, \, - \}'
\bigl)$. Then $\approx_l$ is an equivalence relation on ${\mathcal
J} {\mathcal E}_{(\triangleleft, \, \leftharpoonup)} (A, V)$ and
we denote by ${\mathcal J}{\mathcal H}^{2} \, ( (V, \,
\triangleleft, \, \leftharpoonup), \, A)$ the quotient set
${\mathcal J} {\mathcal E}_{(\triangleleft, \, \leftharpoonup)}
(A, V)/ \approx_l $. \thref{main1} and the above considerations
provide the following decomposition of ${\mathcal J}{\mathcal
H}^{2} \, (V, \, A)$:

\begin{corollary}\colabel{descompunere}
Let $A$ be a Jacobi algebra and $V$ a vector space. Then
\begin{equation} \eqlabel{descompunere}
{\mathcal J}{\mathcal H}^{2} \, (V, \, A) \, = \,
\sqcup_{(\triangleleft, \, \leftharpoonup)} \,\, {\mathcal
J}{\mathcal H}^{2} \, \bigl( (V, \, \triangleleft, \,
\leftharpoonup), \, A \bigl)
\end{equation}
where the coproduct in the right hand side is in the category of
sets over all possible right Jacobi $A$-module structures
$(\triangleleft, \, \leftharpoonup)$ on $V$.
\end{corollary}

The decomposition given by \equref{descompunere} is a very
important step in computing the classifying object ${\mathcal
J}{\mathcal H}^{2} \, (V, \, A)$. However, even computing every
object ${\mathcal J}{\mathcal H}^{2} \, \bigl( (V, \,
\triangleleft, \, \leftharpoonup), \, A \bigl)$, for a fixed right
Jacobi $A$-module $(V, \, \triangleleft, \, \leftharpoonup)$ is a
problem far from being trivial. However the decomposition is
important as some of the components in the right hand side of
\equref{descompunere} might be equal to the empty set as the
following example shows - several explicit examples of computing
${\mathcal J}{\mathcal H}^{2} \, (V, \, A)$ are provided in
\seref{exemple}.

\begin{example} \exlabel{cazvid}
We consider the trivial right Jacobi $A$-module structure on $V$,
that is $x \triangleleft a := x$ and $x \leftharpoonup a := 0$,
for all $x \in V$, $a \in A$. This right Jacobi $A$-module
structure on $V$ was denoted by $V_0$. If $V \neq \{0\}$, then
${\mathcal J}{\mathcal H}^{2} \, \bigl( V_0, \, A \bigl)  =
\emptyset $. We will prove that in fact the corresponding set
${\mathcal J} {\mathcal E}_{(\triangleleft, \, \leftharpoonup :=
0)} \, (A, V)$ is empty. Indeed, if $\bigl(\triangleright, \, f,
\, \cdot, \, \rightharpoonup, \, \theta, \, \{-, \, - \} \bigl)
\in {\mathcal J} {\mathcal E}_{(\triangleleft, \, \leftharpoonup
:= 0)} \, (A, V)$, then taking into account that $\triangleleft$
acts trivially on $V$, if follows from axiom (A3) that $x \cdot y
= x + x \cdot y$, for all $x$, $y\in V$, thus $V = 0$ and we have
reached a contraction.
\end{example}

\section{Flag Jacobi algebras. Examples}\selabel{exemple}
In this section we will test the efficiency of \thref{main1} and
the decomposition given by \equref{descompunere} in order to
compute ${\mathcal J}{\mathcal H}^{2} \, (V, \, A)$ for the class
of Jacobi algebras defined below:

\begin{definition} \delabel{flagex}
Let $A$ be a Jacobi algebra and $E$ a vector space containing $A$
as a subspace. A Jacobi algebra structure on $E$ is called a
\emph{flag extending structure} of $A$ to $E$ if there exists a
finite chain of Jacobi subalgebras of $E$
\begin{equation} \eqlabel{lant}
E_0 := A \subset E_1 \subset \cdots \subset E_m := E
\end{equation}
such that $E_i$ has codimension $1$ in $E_{i+1}$, for all $i = 0,
\cdots, m-1$. A Jacobi algebra that is a flag extending structure
of the Jacobi algebra $k = k_0$ is called a \emph{flag Jacobi
algebra}.
\end{definition}

All flag extending structures of a given Jacobi algebra $A$ to a
vector space $E$ can be completely described by a recursive
reasoning where the key step is the case when $A$ has codimension
$1$ in $E$. This step will provide the description and
classification of all unified products $A\ltimes k$ of Jacobi
algebras. Then, by replacing $A$ with each of these unified
products $A\ltimes k$, we go on with the recursive process in $m$
steps, where $m$ is the codimension of $A$ in $E$. The tool that
will play the key role in the description of flag extending
structures is the following:

\begin{definition} \delabel{tehnica}
Let $A$ be a Jacobi algebra. A \emph{Jacobi flag datum} of $A$ is
a $6$-tuple $(\Lambda, \, \Delta, \, f_0, \, u, \, \lambda, \,
D)$, consisting of four $k$-linear maps $\Lambda$, $\lambda : A
\to k$, $\Delta$, $D: A \to A$ and two elements $f_0\in A$ and
$u\in k$ satisfying the following compatibilities for any $a$, $b
\in A$:

(JF0) \, $(\Lambda, \, \Delta, \, f_0, \, u)$ is a flag datum of
the associative algebra $A$ and $(\lambda, \, D)$ is a twisted
derivation of the Lie algebra $A$;

(JF1) \, $D(ab) = D(a) \, b + a\, D(b) + \lambda(a) \, \Delta(b) +
\lambda(b) \, \Delta (a) - \lambda(1_A) \Delta (ab) - ab \, D(1_A)
$

(JF2) \, $\lambda(ab) = \lambda(a)\, \Lambda (b) + \lambda(b) \,
\Lambda(a) - \lambda(1_A) \, \Lambda (ab)$

(JF3) \, $\Delta([a, \, b]) = [\Delta(a), \, b] + \Lambda(a) \,
D(b) - a \, D(b) - \lambda(b)\, \Delta (a) + \Delta(a) [1_A, \, b]
+ \Lambda(a) \, \Delta( [1_A, \, b])$

(JF4) \, $\Lambda([a, \, b]) = \Lambda(a) \, \Lambda([1_A, \, b])$

(JF5)  \, $ \Delta \bigl(D(a)\bigl) - D\bigl( \Delta(a) \bigl) \, =
\Delta(a) D(1_A) + \Lambda(a) \, \lambda(1_A) \, f_0 - \lambda(a)
\, f_0 + \Lambda(a) \, \Delta \bigl( D(1_A)\bigl) + $

$ \qquad \qquad \qquad \qquad \qquad \qquad \,\,\, + \,
\lambda(1_A)\, \Delta^2(a)$

(JF6) \, $\Lambda\bigl(D(a)\bigl) - \lambda\bigl(\Delta(a) \bigl)
\, = \Lambda(a) \, \Lambda \bigl(D(1_A) \bigl) + \lambda(1_A) \,
\Lambda(a)\, u - \lambda (a) \, u$

(JF7) \, $2 \, \Delta \bigl(D(a)\bigl) + 2 \, \lambda(a) \, f_0 =
u \, D(a) + [f_0, \, a] + f_0 \, [1_A, \, a] + u\, \Delta([1_A, \,
a])$

(JF8) \, $2\, \Lambda \bigl( D(a) \bigl) \, = - \lambda(a) \, u + u
\, \Lambda ([1_A, \, a]) $

(JF9) \, $D(f_0) + f_0 \, D(1_A) + u\, \Delta \bigl(D(1_A)\bigl) +
\lambda(1_A) \, \Delta(f_0) + u\, \lambda(1_A) \, f_0 = 0$

(JF10) \, $\lambda(f_0) + u\, \Lambda \bigl( D(1_A)\bigl) +
\lambda (1_A) \, \Lambda(f_0) + u^2 \, \lambda(1_A) = 0$

We denote by ${\mathcal J} {\mathcal F} \, (A)$ the set of all
Jacobi flag datums of $A$.
\end{definition}

For further computations we point out that for any $(\Lambda, \,
\Delta, \, f_0, \, u, \, \lambda, \, D) \in {\mathcal F} {\mathcal
J} \, (A)$ we have $\Delta (1_A) = 0$ and $\Lambda (1_A) = 1$.

\begin{proposition}\prlabel{unifdim1}
Let $A$ be a Jacobi algebra. Then there exists a bijection between the set ${\mathcal J}
{\mathcal E} \, (A, \, k)$ of all Jacobi extending structures of
$A$ through $k$ and ${\mathcal J} {\mathcal F} \, (A)$ given such
that the unified product $A \ltimes_{(\Lambda, \, \Delta, \, f_0,
\, u, \, \lambda, \, D)} \,\, k$ corresponding to the Jacobi flag
datum $(\Lambda, \, \Delta, \, f_0, \, u, \, \lambda, \, D) \in
{\mathcal J} {\mathcal F} \,(A)$, denoted by $A_{(\Lambda, \,
\Delta, \, f_0, \, u, \, \lambda, \, D)}$, is the vector space $A
\times k$ with the Jacobi algebra structure given for any $a$, $b
\in A$ and $x$, $y\in k$ by:
\begin{eqnarray}
(a, x) \bullet (b, y) &=& \bigl(ab + x\Delta(b) + y \Delta(a) + xy
\, f_0, \,\, x\Lambda(b) + y\Lambda(a) + xy \, u \bigl)
\eqlabel{002b}
\\
\left[ (a, x), \, (b, y) \right] &=& \bigl( \left[a, \, b \right]
+ x D(b) - y D(a), \,\, x\lambda(b) - y \lambda(a) \bigl)
\eqlabel{002bc}
\end{eqnarray}

Furthermore, a Jacobi algebra $B$ contains $A$ as a Jacobi subalgebra of
codimension $1$ if and only if $B \cong A_{(\Lambda, \, \Delta, \,
f_0, \, u, \, \lambda, \, D)}$, for some Jacobi flag datum
$(\Lambda, \, \Delta, \, f_0, \, u, \, \lambda, \, D)$ of $A$.
\end{proposition}

\begin{proof}
We have to compute the set of all bilinear maps
\begin{eqnarray*}
&& \triangleleft : k \times A \to k, \quad \triangleright : k
\times A \to A, \quad f: k\times k \to A, \quad \cdot \, : k\times
k \to k \\
&& \leftharpoonup: k \times A \to k, \quad \rightharpoonup : k
\times A \to A, \quad \theta: k\times k \to A, \quad \{-, \, -\} :
k\times k \to k
\end{eqnarray*}
satisfying the compatibility conditions (J0)-(J10) in \thref{1}.
To start with, the first part of axiom $(J0)$ tells us that
$(\triangleleft, \, \triangleright, \, f, \, \cdot)$ is an algebra
extending system of $A$ through $k$ and \exref{flagalg} proves
that there is a bijection between the set of all such maps
$(\triangleleft, \, \triangleright, \, f, \, \cdot)$ and the set
of all $4$-tuples $(\Lambda, \, \Delta, \, f_0, \, u) \in A^*
\times {\rm End}_k (A) \times A \times k$ that are flag datums of
the associative algebra $A$: the bijection is given such that the
algebra extending system $\bigl(\triangleleft, \, \triangleright,
\, f, \, \cdot \bigl)$ associated to $(\Lambda, \, \Delta, \, f_0,
\, u) \in {\mathcal F}(A)$ is defined by the formulas
\equref{bijflag}. Secondly, the last assertion of $(J0)$ tells us
that $(\leftharpoonup, \, \rightharpoonup, \, \theta, \, \{-, \,
-\})$ is a Lie extending system of the Lie algebra $A$ through the
vector space $k$ and \exref{flagLie} shows that there is a
bijection between this set and the space ${\rm TwDer} (A)$ of all
twisted derivations of $A$; the bijection is given such that the
Lie extending system $\bigl(\leftharpoonup, \, \rightharpoonup, \,
\theta, \{-, \, -\} \bigl)$ associated to a twisted derivation
$(\lambda, D) \in {\rm TwDer} (A)$ is defined by the formulas
\equref{bijflaglie}. Hence, there exists a bijection between the
set of all bilinear maps $(\triangleleft, \, \triangleright, \, f,
\, \cdot, \, \leftharpoonup, \, \rightharpoonup, \, \theta, \,
\{-, \, -\})$ satisfying (J0) and the set of all $6$-tuples
$(\Lambda, \, \Delta, \, f_0, \, u, \, \lambda, \, D)$ as defined
in \deref{tehnica} satisfying the compatibility conditions (JF0)
and the bijection is given by the formulas
\equref{bijflag}-\equref{bijflaglie}. The rest of the proof is a
long but straightforward computation which shows that, under this
bijection, the compatibility conditions (J1)-(J10) of \thref{1}
take the equivalent forms given by (JF1)-(JF10). Finally, the
Jacobi algebra defined by \equref{002b}-\equref{002bc} is exactly
the associated unified product $A \ltimes k$ defined by
\equref{multunifJ}-\equref{brackunifJ} written in this context.
Finally, the last statement follows from the first part and
\prref{classif}.
\end{proof}

Let $\{e_i \, |\, i\in I\}$ be a basis of a Jacobi algebra $A$ and
$(\Lambda, \, \Delta, \, f_0, \, u, \, \lambda, \, D) \in
{\mathcal J} {\mathcal F} \,(A)$ a Jacobi flag datum. Then
$A_{(\Lambda, \, \Delta, \, f_0, \, u, \, \lambda, \, D)}$ is the
Jacobi algebra having $\{E, \, e_i \, |\, i\in I\}$ as a basis
with the multiplication and the bracket defined for any $i\in I$
by:
\begin{eqnarray}
e_i \bullet e_j &:=& e_i \cdot_A e_j, \quad E \bullet e_i = e_i
\bullet E := \Delta (e_i) + \Lambda(e_i) \, E, \quad E^2 :=
f_0 + u\, E \eqlabel{terente3} \\
\left[e_i, \, e_j \right] &:=& \left[e_i, \, e_j \right]_A, \quad
\left[E, \, e_i \right] := D(e_i) + \lambda(e_i) \, E
\eqlabel{terente3b}
\end{eqnarray}
where $\cdot_A$ (resp. $[-, \, -]_A$) is the multiplication (resp.
the bracket) on $A$. The existence of these Jacobi algebras
depends on the Jacobi algebra $A$. An interesting fact is the
following:

\begin{corollary}\colabel{prefect}
Let $A$ be a perfect Poisson algebra, i.e. $[A, \, A] = A$. Then,
there is no Jacobi algebra which contains $A$ as a subalgebra of
codimension $1$.
\end{corollary}

\begin{proof}
It follows from \prref{unifdim1} if we prove that the set
${\mathcal J} {\mathcal F} \,(A)$ is empty. Indeed, let $(\Lambda,
\, \Delta, \, f_0, \, u, \, \lambda, \, D) \in {\mathcal J}
{\mathcal F} \,(A)$. Since $[1_A, \, a] = 0$, for all $a \in A$,
it follows from axiom (JF4) that $\Lambda ([a, \, b]) = 0$, for
all $a$, $b\in A$. As $A$ is perfect as a Lie algebra, we obtain
that $\Lambda (x) = 0$, for any $x\in A$, contrary to axiom (FA1)
which ensures that $\Lambda (1_A) = 1$.
\end{proof}

\begin{remark} \relabel{schurinv}
A basic invariant of a finite dimensional Lie algebra
$\mathfrak{g}$ is the Schur invariant defined by $\alpha
(\mathfrak{g}) :=$ the maximal dimension of an abelian subalgebra
of $\mathfrak{g}$. There is a vast literature devoted to computing
this number for several classes of Lie algebras such as
(semi)simple, (super)solvable, etc. - see \cite{BC, CT} and their
references. Lie algebras for which $\alpha (\mathfrak{g}) = {\rm
dim} (\mathfrak{g}) - 1 $ are fully described in \cite{gor}: below
we give the Jacobi algebra version of this result. For a finite
dimensional Jacobi algebra $J$ we define the Schur invariant by
the formula:
$$
\alpha (J) := {\rm max}\, \{ \, {\rm dim}\,(A) ~|~ A\,\, {\rm
is}\,\, {\rm an}\,\, {\rm abelian} \,\, {\rm Jacobi} \,\, {\rm
subalgebra}\,\, {\rm of}\,\, J\}
$$
Using \prref{unifdim1} we obtain that  a $(n+1)$-dimensional
Jacobi algebra $J$ has $\alpha (J) = {\rm dim} (J) - 1 $ if and
only if $J \cong A_{(\Lambda, \, \Delta, \, f_0, \, u, \, \lambda,
\, D)}$, where $A$ is an $n$-dimensional algebra with basis $\{e_i
\, |\, i = 1, \cdots, n \}$ and $A_{(\Lambda, \, \Delta, \, f_0,
\, u, \, \lambda, \, D)}$ is the Jacobi algebra having $\{E, \,
e_i \, |\, i = 1, \cdots, n \}$ as a basis and the multiplication
and the bracket is given for any $i$, $j = 1, \cdots n$ by:
\begin{eqnarray*}
e_i \bullet e_j &:=& e_i \cdot_A e_j, \quad E \bullet e_i = e_i
\bullet E := \Delta (e_i) + \Lambda(e_i) \, E, \quad E^2 :=
f_0 + u\, E \\
\left[E, \, e_i \right] &:=& D(e_i) + \lambda(e_i) \, E
\end{eqnarray*}
for some $6$-tuple $(\Lambda, \, \Delta, \, f_0, \, u, \, \lambda,
\, D) \in {\mathcal J} {\mathcal F} \,(A)$ with $(D, \, \lambda)
\neq (0, 0)$. We mention that in this case the axioms (JF0)-(JF10)
which need to be fulfilled by the $6$-tuples $(\Lambda, \, \Delta,
\, f_0, \, u, \, \lambda, \, D)$ take a simplified form as the Lie
bracket on $A$ is trivial.
\end{remark}

Now we will classify the algebras $A_{(\Lambda, \, \Delta, \, f_0,
\, u, \, \lambda, \, D)}$ by providing the first explicit
classification result of the ES-problem for Jacobi algebras:

\begin{theorem}\thlabel{clasdim1}
Let $A$ be a Jacobi algebra of codimension $1$ in the vector space
$E$. Then there exist a bijection
\begin{equation}\eqlabel{flagformmare}
{\rm Extd}_{{\mathcal J}} \, (E, \, A) \cong {\mathcal J}{\mathcal
H}^{2} \, (k, \, A) \cong  {\mathcal J} {\mathcal F} \,(A)/\equiv
\end{equation}
where $\equiv $ is the equivalence relation on the set ${\mathcal
J} {\mathcal F} \, (A)$ of all Jacobi flag datums of $A$ defined
as follows: $(\Lambda, \, \Delta, \, f_0, \, u, \, \lambda, \, D)
\equiv (\Lambda', \, \Delta', \, f_0', \, u', \, \lambda', \, D')$
if and only if $\Lambda' = \Lambda$, $\lambda' = \lambda$, $u' =
u$ and there exists $\alpha \in A$ such that for any $a \in A$ we
have:
\begin{eqnarray}
\Delta' (a) &=& \Delta (a) + \Lambda(a) \alpha - a \alpha \eqlabel{ecfl1} \\
f_0' &=& f_0 + \alpha^2 + u \alpha - 2 \Lambda(\alpha) \alpha - 2
\Delta (\alpha) \eqlabel{ecfl2} \\
D'(a) &=& D(a) + \lambda(a) \alpha - \left[\alpha, \, a
\right]\eqlabel{ecfl3}
\end{eqnarray}
The bijection between ${\mathcal J} {\mathcal F} \,(A)/\equiv$ and
${\rm Extd}_{{\mathcal J}} \, (E, \, A)$ is given by $
\overline{(\Lambda, \, \Delta, \, f_0, \, u, \, \lambda, \, D)}
\mapsto A_{(\Lambda, \, \Delta, \, f_0, \, u, \, \lambda, \, D)}$,
where $\overline{(\Lambda, \, \Delta, \, f_0, \, u, \, \lambda, \,
D)}$ is the equivalence class of $(\Lambda, \, \Delta, \, f_0, \,
u, \, \lambda, \, D)$ via the relation $\equiv$.
\end{theorem}

\begin{proof}
Let $(\Lambda, \, \Delta, \, f_0, \, u, \, \lambda, \, D)$,
$(\Lambda', \, \Delta', \, f_0', \, u', \, \lambda', \, D') \in
{\mathcal J} {\mathcal F} \, (A)$ and $\Upsilon (A, V)$,
respectively $\Upsilon' (A, V)$ be the corresponding Jacobi
algebra extending structures. Since ${\rm dim}_k (V) = 1$, any
linear map $r: V \to A$ is uniquely determined by an element
$\alpha \in A$ such that $r(x) = \alpha$, where $\{x\}$ is a basis
in $V$. We can easily see that the compatibility conditions from
\prref{echiaabb} applied to Jacobi flag datums take precisely the
form given in the statement and hence the proof follows from
\thref{main1} and \prref{unifdim1}.
\end{proof}

\begin{remark} \relabel{practica}
In practice, in order to compute the quotient set ${\mathcal J}
{\mathcal F} \,(A)/\equiv$ constructed in \thref{clasdim1} we
shall use the decomposition \equref{descompunere} obtained in
\coref{descompunere} by going through the following steps. First
of all, we shall fix a pair $(\Lambda, \lambda) \in A^* \times
A^*$, consisting of an associative algebra map $\Lambda: A \to k$
and a Lie algebra map $\lambda : A \to k$ satisfying the
compatibility conditions \equref{modflag2b}. Secondly, for a given
such pair $(\Lambda, \lambda)$ we fix a scalar $u \in k$ and
compute the set ${\mathcal J} {\mathcal F}_{(\Lambda,
\lambda)}^{u} \,(A)$ consisting of all triples $(\Delta, \, f_0,
\, D) \in {\rm End}_k (A) \times A \times {\rm End}_k (A)$ such
that $(\Lambda, \, \Delta, \, f_0, \, u, \, \lambda, \, D)$ is a
Jacobi flag datum of $A$. Two triples $(\Delta, \, f_0, \, D)$ and
$(\Delta', \, f_0', \, D') \in {\mathcal J} {\mathcal
F}_{(\Lambda, \lambda)}^{u}  \,(A)$ are equivalent and we denote
this by $(\Delta, \, f_0, \, D) \equiv_{(\Lambda, \lambda)}^{u} \,
(\Delta', \, f_0', \, D')$ if and only if there exists $\alpha \in
A$ such that \equref{ecfl1}-\equref{ecfl3} hold. Finally, we
compute the quotient set ${\mathcal J} {\mathcal F}_{(\Lambda,
\lambda)}^{u} \, (A)/\equiv_{(\Lambda, \lambda)}^{u}$ --
${\mathcal J}{\mathcal H}^{2} \, (k, \, A)$ will be the coproduct
of these quotients sets over all triples $(\Lambda, \lambda, u)$
-- and then we list the isomorphism classes of the associated
Jacobi algebras $A_{(\Lambda, \, \Delta, \, f_0, \, u, \, \lambda,
\, D)}$ using \equref{terente3}-\equref{terente3b}. To conclude,
using \thref{clasdim1} and \coref{descompunere}, we obtain:
\end{remark}

\begin{corollary}\colabel{desc2fagmare}
Let $A$ be a Jacobi algebra. Then:
\begin{equation} \eqlabel{descompflag2}
{\mathcal J}{\mathcal H}^{2} \, (k, \, A) \cong \,
\sqcup_{(\Lambda, \lambda)} \, \Bigl(\sqcup_u \,\, \bigl(
{\mathcal J} {\mathcal F}_{(\Lambda, \lambda)}^{u}  \,
(A)/\equiv_{(\Lambda, \lambda)}^{u} \bigl) \Bigl)
\end{equation}
where the coproducts in the right hand side are made in the
category of sets over all possible pairs $(\Lambda, \lambda)$
consisting of an associative algebra map $\Lambda: A \to k$ and a
Lie algebra map $\lambda : A \to k$ satisfying \equref{modflag2b}
and over all scalars $u\in k$.
\end{corollary}

\begin{remark} \relabel{flagpoiss}
Let $A$ be a Poisson algebra. The Poisson algebra version of
\thref{clasdim1} and \coref{desc2fagmare} for computing ${\mathcal
P}{\mathcal H}^{2} \, (k, \, A)$ defined in \reref{possversES} are
obtained as follows. First, we define the set ${\mathcal P}
{\mathcal F} \, (A)$ of all \emph{Poisson flag datums} of $A$: it
coincides with the set of all $6$-tuples $(\Lambda, \, \Delta, \,
f_0, \, u, \, \lambda, \, D)$, consisting of four $k$-linear maps
$\Lambda$, $\lambda : A \to k$, $\Delta$, $D: A \to A$ and two
elements $f_0\in A$ and $u\in k$ satisfying the following
compatibilities for any $a$, $b \in A$:

(PF0) \, $(\Lambda, \, \Delta, \, f_0, \, u)$ is a flag datum of
the associative algebra $A$ and $(\lambda, \, D)$ is a twisted
derivation of the Lie algebra $A$;

(PF1) \, $D(ab) = D(a) \, b + a\, D(b) + \lambda(a) \, \Delta(b) +
\lambda(b) \, \Delta (a) $

(PF2) \, $\lambda(ab) = \lambda(a)\, \Lambda (b) + \lambda(b) \,
\Lambda(a) $

(PF3) \, $\Delta([a, \, b]) = [\Delta(a), \, b] + \Lambda(a) \,
D(b) - a \, D(b) - \lambda(b)\, \Delta (a)$

(PF4) \, $\Lambda([a, \, b]) = 0$, \quad $D(f_0) = 0$, \quad
$\lambda(f_0) = 0$

(PF5) \, $ D\bigl( \Delta(a) \bigl) - \Delta \bigl(D(a)\bigl) \, =
\lambda(a) \, f_0 $, \quad $\lambda\bigl(\Delta(a) \bigl) - \Lambda\bigl(D(a)\bigl) \, = \lambda (a) \, u$

(PF6) \, $2 \, \Delta \bigl(D(a)\bigl) + 2 \, \lambda(a) \, f_0 =
u \, D(a) + [f_0, \, a]$, \quad $2\, \Lambda \bigl( D(a) \bigl) \, =
- \lambda(a) \, u $

Then there exists a bijection ${\mathcal P}{\mathcal H}^{2} \, (k,
\, A) \, \cong \, {\mathcal P} {\mathcal F} \,(A)/\equiv $, where
$\equiv $ is the equivalence relation defined exactly as in
\thref{clasdim1}, but on the set ${\mathcal P} {\mathcal F} \,
(A)$ of all Poisson flag datums of the Poisson algebra $A$.
\end{remark}

Now, we will illustrate by examples the efficiency of
\coref{desc2fagmare} in classifying flag Jacobi algebras. The
strategy followed will be that of \reref{practica} imposed by the
decomposition of ${\mathcal J}{\mathcal H}^{2} \, (k, \, A)$ given
in \equref{descompflag2}. Moreover, if $A$ is a Poisson algebra we
will also describe ${\mathcal P}{\mathcal H}^{2} \, (k, \, A)$ in
order to illustrate the difference between it and ${\mathcal
J}{\mathcal H}^{2} \, (k, \, A)$. The model is given below and we
make the following convention: all undefined bracket or
multiplication of two elements of a basis is zero. Let $J^{2, \,
1}$ be the $2$-dimensional Jacobi algebra of \exref{dim2}.

\begin{proposition} \prlabel{J21}
Let $k$ be a field of characteristic $\neq 2$. Then:
\begin{eqnarray}
{\mathcal J}{\mathcal H}^{2} \, (k, \, J^{2, \, 1} ) \, &\cong &
\, (k \times k^* \times k) \, \sqcup \, (k^* \times k^2) \, \sqcup
\, k^* \sqcup k^* \, \sqcup \, k^4 \eqlabel{formmar}
\end{eqnarray}
and the equivalence classes of all $3$-dimensional flag Jacobi
algebra over $J^{2, \, 1}$ are the Jacobi algebras $J^{3,
1}_{(\lambda_1, \, \lambda_2, \, u)}$, \, $J^{3, 2}_{(\lambda_1,
\, u, \, f)}$, \, $J^{3, 3}_{\delta}$, \, $J^{3, 4}_{u}$, \,
$J^{3, 5}_{(u, \, f, \, d_1, \, d_2)}$ having $\{1, \, x, \, y \}$
as a basis and the multiplication and the bracket defined in Table
\ref{1}.
\begin{center}
\begin{table}
\begin{tabular}{ | l | c | c | }
\hline
Jacobi algebra & Algebra structure & Lie bracket\\
\hline

$J^{3, 1}_{(\lambda_1, \, \lambda_2, \, u)}$ & $x^2 = 0$ & $\left[
y, \, 1 \right] = - 2^{-1}\lambda_1 u + \lambda_1 y$ \\

$(\lambda_1, \lambda_2, u) \in k \times k^* \times k$ &  $yx = xy
= 2^{-1} \, u \, x$ & $\left[y, \,
x \right] = - 2^{-1}\lambda_2 u + \lambda_2 y$ \\
& $y^2 = - 4^{-1} \, u^2 + u \, y$ &  \\\hline

$J^{3, 2}_{(\lambda_1, \, u, \, f)}$ & $x^2 = 0$ &  \\
$(\lambda_1, u, f) \in k^* \times k^2$ & $yx = xy = 2^{-1} \, u \,
x$ & $\left[y, \, 1
\right] = - 2^{-1} \, \lambda_1 u + \lambda_1 \, y$ \\
& $y^2 = - 4^{-1} \, u^2 + f \, x+ u \, y$ & \\\hline

$J^{3, 3}_{\delta}$, \, $\delta \in k^{*}$ & $x^2 = 0$, \,$y^2 = \delta^2$ & {\rm abelian} \\
& $yx = xy = \delta \, x$ & \\\hline

$J^{3, 4}_{u}$,\, $u \in k^{*}$ &  $x^2 = 0$, \,$y^2 = u\,y$ & {\rm abelian} \\
& $yx = xy = 0$ & \\\hline

$J^{3, 5}_{(u, \, f, \, d_1, \, d_2)}$ & $x^2 = 0$ & $\left[y, \,
1 \right] = d_1 x$  \\
$(u, \, f, d_1, d_2) \in k^4$ & $yx = xy = 2^{-1}u \, x$ &
$\left[y, \, x \right] = d_2 x$  \\
& $y^2 = -4^{-1} u^2 + f \, x + u \, y$ & \\\hline
\end{tabular}
\caption{Flag Jacobi algebras over $J^{2, \, 1}$.} \label{1}
\end{table}
\end{center}
\end{proposition}

\begin{proof}
To start with, we should notice that there is only one algebra map
$\Lambda: J^{2, \, 1} \to k$ namely the one given by $\Lambda (1)
= 1$ and $\Lambda(x) = 0$. We denote $\lambda (1) = \lambda_1$ and
$\lambda (x) = \lambda_2$, for some $(\lambda_1, \, \lambda_2) \in
k^2$. Then $\lambda: J^{2, \, 1} \to k$ is a Lie algebra map (the
Lie brackets on $J^{2, \, 1}$ and $k$ are both abelian) and
moreover \equref{modflag2b} is fulfilled. Thus, the set of all
pairs $(\Lambda, \lambda)$ consisting of an algebra map $\Lambda:
J^{2, \, 1} \to k$ and a Lie algebra map $\lambda : J^{2, \, 1}
\to k$ satisfying \equref{modflag2b} is in bijection with $k^2$.
From now on we fix a pair $(\lambda_1, \, \lambda_2) \in k^2$ and
a scalar $u\in k$. We will describe the set ${\mathcal J}
{\mathcal F}_{(\Lambda, \lambda)}^{u}  \,(J^{2, \, 1})$ consisting
of all triples $(\Delta, \, f_0, \, D)$ such that $(\Lambda, \,
\Delta, \, f_0, \, u, \, \lambda, \, D)$ is a Jacobi flag datum of
$J^{2, \, 1}$. We identify the $k$-linear maps $\Delta$ and $D :
J^{2, \, 1} \to J^{2, \, 1} $ with their associated matrices with
respect to the basis $\{1, \, x\}$:
\begin{eqnarray*}
\Delta = \left( \begin{array}{cc} \Delta_{11} & \Delta_{12}\\
\Delta_{21} & \Delta_{22} \end{array}\right) \qquad
D = \left( \begin{array}{cc} d_{11} & d_{12}\\
d_{21} & d_{22} \end{array}\right) \qquad f_0 = f_1 + f_2 \, x
\end{eqnarray*}
for some $(\Delta_{ij}) \in k^4$, $(d_{ij}) \in k^4$, $(f_1, \,
f_2) \in k^2$. We have to determine all these scalars such that
axioms (JF0)-(JF10) are fulfilled for a fixed triple $(\lambda_1,
\, \lambda_2, \, u) \in k^3$. Since this comes down to a laborious
computation we will only provide a sketch of it.
For instance, we can easily obtain that axiom (JF0) holds for a fixed
triple $(\lambda_1, \, \lambda_2, \, u) \in k^3$ if and only if
$(\Delta, \, f_0, \, D)$ is given by
\begin{eqnarray}
\Delta = \left( \begin{array}{cc} 0 & 0 \\
0 & \delta \end{array}\right) \qquad
D = \left( \begin{array}{cc} d_{11} & d_{12}\\
d_{21} & d_{22} \end{array}\right) \qquad f_0 = \delta (\delta -
u) + f \, x  \eqlabel{muica}
\end{eqnarray}
for some $\delta$, $f \in k$ and $(d_{ij}) \in k^4$ such that
$\lambda_1 d_{12} = \lambda_2 d_{11}$ and $\lambda_1 d_{22} =
\lambda_2 d_{21}$. We continue in similar manner by testing the
remaining compatibilities: for instance (JF2) and (JF4) hold
automatically for such a triple $(\Delta, \, f_0, \, D)$ while we
can easily see that (JF1) (resp. (JF3)) holds if and only if
$d_{12} = - \lambda_2 \delta$ (resp. $d_{11} = - \lambda_1
\delta$). On the other hand we can prove that (JF5) holds if and
only if $\lambda_2 \, f = 0$ and $\lambda_2 \, \delta ( 2 \delta -
u) = 0$. Moreover, (JF8) holds if and only if $d_{11} = - 2^{-1}
\, \lambda_1 u$ and $d_{12} = - 2^{-1} \, \lambda_2 u$). By
eliminating the redundant compatibilities we can conclude that
axioms (JF0)-(JF10) are fulfilled if and only if the triple
$(\Delta, \, f_0, \, D)$ given by \equref{muica} satisfies the
following equations:
\begin{eqnarray}
&& d_{11} = - \lambda_1 \delta = - 2^{-1} \, \lambda_1 u, \quad
d_{12} = - \lambda_2 \delta = - 2^{-1} \, \lambda_2 u, \quad
\lambda_1 d_{22} = \lambda_2 d_{21} \eqlabel{mui1} \\
&& \lambda_2 f = \lambda_2 (2 \, \delta - u) = \lambda_1 \delta (2
\, \delta - u ) =  f d_{22} = 0 \eqlabel{mui2} \\
&& d_{21} (2 \, \delta - u) = d_{22} (2 \, \delta - u) = 0, \quad
\lambda_1 ( 2 \, \delta^2 + \delta \, u - u^2) = 0 \eqlabel{mui3}
\end{eqnarray}
Now, based on \equref{mui2} we can decompose the
set of all Jacobi flag datums of $J^{2, \, 1}$ as follows
$$
{\mathcal J} {\mathcal F} \,(J^{2, \, 1}) = J_1 \sqcup J_2 \sqcup
J_3
$$
where $J_1$ are those flag datums corresponding to $\lambda_2 \neq
0$, $J_2$ are associated to the case $\lambda_2 = 0$ and
$\lambda_1 \neq 0$ while $J_3$ correspond to the case $\lambda_2 =
\lambda_1 = 0$.

In what follows we only provide details for the first case, namely
the one corresponding to $\lambda_2 \neq 0$ -- the other two cases
are settled using an analogous treatment. If we denote $d_{22} :=
d \in k$, then equations \equref{mui1}-\equref{mui3} hold if and
only if $f = 0$, $\delta = 2^{-1} \, u$ and $d_{21} =
\lambda_2^{-1} \lambda_1 d$. Thus, $J_1 \cong k \times k^* \times
k^2$ and the bijection is given such that the Jacobi flag datum
$(\Lambda, \, \Delta, \, f_0, \, u, \, \lambda, \, D) \in
{\mathcal J} {\mathcal F} \,(J^{2, \, 1})$ associated to
$(\lambda_1, \, \lambda_2, \, u, \, d) \in k \times k^* \times
k^2$ is given by
\begin{eqnarray}
&& \Lambda (1) = 1, \,\, \Lambda (x) = 0, \quad \lambda(1) =
\lambda_1, \, \, \lambda(x) = \lambda_2 \neq 0 \eqlabel{muicaa1} \\
&& \Delta = \left( \begin{array}{cc} 0 & 0 \\
0 & 2^{-1} u \end{array}\right) \quad
D = \left( \begin{array}{cc} - 2^{-1} \, \lambda_1 u & - 2^{-1} \, \lambda_2 u \\
\lambda_2^{-1} \lambda_1 d & d \end{array}\right) \quad f_0 = -
4^{-1} \, u^2 \eqlabel{muicaa2}
\end{eqnarray}
Now, an elementary computation shows that the equivalence relation
of \thref{clasdim1} given by \equref{ecfl1}-\equref{ecfl3} on the
Jacobi flag datums written for the set of triples $k \times k^*
\times k^2$ becomes: $(\lambda_1, \, \lambda_2, \, u, \, d) \equiv
(\lambda_1', \, \lambda_2', \, u', \, d')$ if and only if
$\lambda_1' = \lambda_1$, $\lambda_2' = \lambda_2$, $u' = u$ and
there exists $q \in k$ such that $d' = d + \lambda_2 q$. Since,
$\lambda_2 \neq 0$, there exists such a $q$, namely $q :=
\lambda_2^{-1} (d' - d)$. This shows that $(\lambda_1, \,
\lambda_2, \, u, \, d) \equiv (\lambda_1, \, \lambda_2, \, u, \,
0)$, for any $d\in k$ and the quotient set $k \times k^* \times
k^2/\equiv \,\, \cong \, k \times k^* \times k \times \{0\} \cong
k \times k^* \times k$, which is the first component of ${\mathcal
J}{\mathcal H}^{2} \, (k, \, J^{2, \, 1} )$ in formula
\equref{formmar}. The $3$-parameter Jacobi algebra $J^{3,
1}_{(\lambda_1, \, \lambda_2, \, u)}$ is precisely the unified
product $J^{2, \, 1} \ltimes k$ associated to the Jacobi flag
datum corresponding to $(\lambda_1, \, \lambda_2, \, u, \, 0)$ via
the formulas \equref{muicaa1}-\equref{muicaa2}, with the
multiplication and the bracket as defined by
\equref{terente3}-\equref{terente3b}. In a similar fashion we can prove that $J_2 \cong k^* \times k^3$
and the bijection is given such that the Jacobi flag datum
$(\Lambda, \, \Delta, \, f_0, \, u, \, \lambda, \, D) \in
{\mathcal J} {\mathcal F} \,(J^{2, \, 1})$ associated to
$(\lambda_1, \, u, \, f, \, d) \in k^* \times k^3$ is given by
\begin{eqnarray}
&& \Lambda (1) = 1, \,\, \Lambda (x) = 0, \quad \lambda(1) =
\lambda_1 \neq 0, \, \, \lambda(x) = 0 \eqlabel{muicaa12} \\
&& \Delta = \left( \begin{array}{cc} 0 & 0 \\
0 & 2^{-1} u \end{array}\right) \quad
D = \left( \begin{array}{cc} - 2^{-1} \, \lambda_1 u & 0 \\
d & 0 \end{array}\right) \quad f_0 = - 4^{-1} \, u^2 + f\, x
\eqlabel{muicaa22}
\end{eqnarray}
and $k^* \times k^3/\equiv \,\, \cong k^* \times k^2$ since
$(\lambda_1, \, u, \, f, \, d) \cong (\lambda_1', \, u', \, f', \,
d')$ if and only if $\lambda_1' = \lambda_1$, $u' = u$, $f' = f$
and there exists $q\in k$ such that $d' = d + \lambda_1 q$. The
Jacobi algebra $J^{3, 2}_{(\lambda_1, \, u, \, f)}$ is the unified
product $J^{2, \, 1} \ltimes k$ associated to $(\lambda_1, \, u,
\, f, \, d:= 0)$. Finally, one can show in an analogous manner
that $J_3 \cong (k^* \sqcup k^*) \, \sqcup \, k^4$ and the
corresponding Jacobi algebras are the last three families of Table
\ref{1}: $J^{3, 3}_{\delta}$, $J^{3, 4}_{u}$ and respectively
$J^{3, 5}_{(u, \, f, \, d_1, \, d_2)}$.
\end{proof}

\begin{remarks} \relabel{comentechi}
$(1)$ $J^{2, \, 1}$ is a Poisson algebra and hence we can also
compute the classifying object ${\mathcal P}{\mathcal H}^{2} \,
(k, \, J^{2, \, 1} )$. By testing which of the Jacobi algebras
listed in \prref{J21} are Poisson algebras, we obtain that
$$
{\mathcal P}{\mathcal H}^{2} \, (k, \, J^{2, \, 1} ) \, \cong \,
(k^* \times k) \, \sqcup \, k^* \sqcup k^* \, \sqcup \, k^3
$$
and the equivalence classes of all $3$-dimensional flag Poisson
algebras over $J^{2, \, 1}$ are the following four families of
Poisson algebras $J^{3, 1}_{(0, \, \lambda_2, \, u)}$, $J^{3,
3}_{\delta}$, $J^{3, 4}_{u}$ and $J^{3, 5}_{(u, \, f, \, 0, \,
d_2)}$ of Table \ref{1}.

$(2)$ Similar to \prref{J21} we can describe all flag Jacobi
algebras over any $2$-dimensional Jacobi algebra listed in
\exref{dim2}. In some cases the computations are straightforward:
for instance we can immediately see that  ${\mathcal J}{\mathcal
H}^{2} \, (k, \, J^{2, \, 2} ) = \varnothing$ and, if $k\neq k^2$,
then ${\mathcal J}{\mathcal H}^{2} \, (k, \, J^{2}_d ) =
\varnothing$, for all $d \in S \subset k \setminus k^2$. Indeed,
if $\Lambda : J^{2, \, 2} \to k$ is an algebra map then, $\Lambda
(x) = 0$ since $x^2 = 0$ in $J^{2, \, 2}$. By applying axiom (JF4)
for $a := x$ and $b := 1$, we obtain that $\Lambda(1) = 0$ and we
have reached a contradiction as $\Lambda$ is a unitary algebra
map. Thus, we obtained that there is no 3-dimensional Jacobi
algebra containing $J^{2, \, 2}$ or $J^{2}_d$ as Jacobi
subalgebras. The remaining two Jacobi algebras of \exref{dim2} can
be treated in a similar manner and are left to the reader.
\end{remarks}

The recursive algorithm can be continued in order to describe all
$4$-dimensional flag Jacobi algebras. For instance, computations
similar to those performed in \prref{J21} give the following:

\begin{example} \exlabel{una4dim}
Consider the Jacobi algebra $J^{3}_{11}$ described in Table $1$. Then:
$$
{\mathcal J}{\mathcal H}^{2} \, (k, \, J^{3}_{11} ) \, \cong
\, \{*\} \, \sqcup \, k \, \sqcup \, k^3 \, \sqcup \, k^2 \, \sqcup \, k^2 \, \sqcup \, (k^* \setminus \{1\} )
\, \sqcup \, k \, \sqcup \, k \, \sqcup \, k \, \sqcup \, k \, \sqcup \,
(k^* \setminus \{1\} )
$$
where $\{*\}$ is the singleton set. The equivalence classes of all 4-dimensional flag Jacobi
algebras over $J^{3}_{11}$ are the Jacobi algebras with basis
$\{1, \, x, \, y, \, z\}$ and the multiplication and the bracket defined below:
\begin{eqnarray*}
&J^{4,1}:& \quad x^{2} = x,\, y^{2} = 0,\, z^{2} = 0,\, xy=yx=0,\, zy=yz= 0,\, zx=xz= z,\\
&& \quad  [y,\,1] = y,\, [z,\, 1] =  2\,z,\, [z,\,x] = 2;\\
&J^{4,2}_{\gamma}:& \quad x^{2} = x,\, y^{2} = 0,\, z^{2} = 0,\, xy=yx=0,\, zy=yz= 0,\, zx=xz=\gamma\, y + z,\\
&& \quad  [y,\,1] = y,\, [z,\, 1] =  z,\, [z,\,x] = -\gamma\, y + z,\,\, {\rm where} \, \gamma \in k;\\
&J^{4,3}_{\alpha,\, u,\, v}:& \quad x^{2} = x,\, y^{2} = 0,\, xy=yx=0,\, zx=xz=\alpha\, - \alpha\, x + z, \, zy=yz= -\, \alpha\, y,\\
&& \quad z^{2} = \alpha^{2} + \alpha\,u + v\,x + uz,\, [y,\,1] = y,\, [z,\,y] = \alpha\, y,\,\, {\rm where} \, (\alpha, u, v) \in k^3;\\
&J^{4,4}_{\alpha,\, b}:& \quad x^{2} = x,\, y^{2}
= 0,\, xy=yx=0,\, zx=xz=0,\, \quad zy=yz= \alpha\, y, \\
&& \quad z^{2} = -\,\alpha^{2} +
\alpha^{2}\,x + 2\, \alpha\,z,\, [y,\,1] = y,\, [z,\,y] = b y,\,\, {\rm where} \,\, (\alpha, \, b) \in k^2;\\
&J^{4,5}_{\alpha, \, u}:& \quad x^{2} = x,\, y^{2}
= 0,\, xy=yx=0,\, zx=xz=0,\, \quad zy=yz= \alpha\, y, \\
&& \quad z^{2} = \alpha \, (\alpha - u) - \alpha(\alpha - u)\,x + u\, z,\, [y,\,1] = y,\, [z,\,y] = - \alpha \, y, \, {\rm where} \, (\alpha, u) \in k^2;\\
&J^{4,6}_{\lambda}:& \quad x^{2} = x,\, y^{2}
= 0,\, xy=yx=0,\, zx=xz= x,\, \quad zy=yz= 0,\, z^{2} = x,\\
&& \quad [y,\,1] = y,\, [z,\, 1] = -\, \lambda\,x + \lambda\, z,\,\, {\rm where} \,\, \lambda \in k^{*}-\{1\};\\
\end{eqnarray*}
\begin{eqnarray*}
&J^{4,7}_{a}:& \quad x^{2} = x,\, y^{2}
= 0,\, xy=yx=0,\, zx=xz= x,\, \quad zy=yz= 0,\, z^{2} = x,\\
&& \quad [y,\,1] = y,\, [z,\, 1] = - x + ay + z,\,\, {\rm where} \,\, a \in k;\\
&J^{4,8}_{b}:& \quad x^{2} = x,\, y^{2}
= 0,\, xy=yx=0,\, zx=xz= x,\, \quad zy=yz= 0,\, z^{2} = x + by,\\
&& \quad [y,\,1] = y,\, [z,\, 1] = -\,2^{-1}\,x + 2^{-1}\, z,\,\, {\rm where} \,\, b \in k;\\
&J^{4,9}_{a}:& \quad x^{2} = x,\, y^{2}
= 0,\, xy=yx=0,\, zx=xz= x,\, zy=yz= 2^{-1}\,y,\\
&& \quad z^{2} = -4^{-1} + 4^{-1}\,x + ay +z,\, [y,\,1] = y, \, [z,\,y]=-\,2^{-1}y,\\
&& \quad [z,\, 1] = -\,4^{-1}-\,4^{-1}\,x + 2^{-1}\, z,\,\, {\rm where} \, a \in k;\\
&J^{4,10}_{a}:& \quad x^{2} = x,\, y^{2}
= 0,\, xy=yx=0,\, zx=xz= x,\, zy=yz= 2^{-1}\,y,\\
&& \quad z^{2} = -4^{-1} + 4^{-1}\,x + z,\, [y,\,1] = y,\, [z,\, y] = -2^{-1}\,y\\
&& \quad [z,\, 1] = -\,2^{-1} - 2^{-1}\,x + ay + z,\,\, {\rm where} \, a \in k;\\
&J^{4,11}_{\lambda}:& \quad x^{2} = x,\, y^{2}
= 0,\, xy=yx=0,\, zx=xz= x,\, zy=yz= 2^{-1}\,y,\\
&& \quad z^{2} = -4^{-1} + 4^{-1}\,x + z,\, [y,\,1] = y,\, [z,\, y] = -2^{-1}\,y\\
&& \quad [z,\, 1] = -\,2^{-1}\,\lambda - 2^{-1}\, \lambda\,x + \lambda\, z,\,\, {\rm where} \,\, \lambda \in k^{*}-\{1\}.\\
\end{eqnarray*}
\end{example}

\section{Bicrossed products for Poisson algebras. Applications}\selabel{cazurispeciale}
In this section we deal with a special case of the unified product
for Poisson algebras, namely the bicrossed product and its main
applications. Throughout this section the associative algebras are
commutative but not necessarily unital. Let $P$ and $Q$ be two
given Poisson algebras. A Poisson algebra $R$ \emph{factorizes}
through $P$ and $Q$ if $P$, $Q$ are Poisson subalgebras of $R$
such that $R = P + Q$ and $P \cap Q = \{0\}$. In this case $Q$ is
called a \emph{Poisson complement} of $P$ in $R$ or a
\emph{$P$-complement of $R$}.

We recall from \cite{LW, majid} that a \emph{matched pair} of Lie
algebras is a system $\bigl(P, \, Q, \, \leftharpoonup, \,
\rightharpoonup \bigl)$ consisting of two Lie algebras $P$ and $Q$
and two bilinear maps $\leftharpoonup: Q \times P \to Q$ and
$\rightharpoonup : Q \times P \to P$ such that $(Q, \,
\leftharpoonup)$ is a right Lie $P$-module, $(P, \,
\rightharpoonup)$ is a left Lie $Q$-module such that for any $a$,
$b\in P$ and $x$, $y\in Q$
\begin{eqnarray}
x \rightharpoonup [a, \, b]_P &=& [x \rightharpoonup a, \, b]_P +
[a, \, x \rightharpoonup b]_P + (x \leftharpoonup a)
\rightharpoonup b - (x \leftharpoonup b) \rightharpoonup a \eqlabel{mp2la}\\
\{x, \, y \}_Q \leftharpoonup a &=& \{x, \, y \leftharpoonup a
\}_Q + \{x \leftharpoonup a, \, y \}_Q + x \leftharpoonup (y
\rightharpoonup a) - y \leftharpoonup (x \rightharpoonup a)
\eqlabel{mp2la}
\end{eqnarray}

The associative algebra counterpart of the matched pair was
introduced in \cite[Definition 3.6]{am-2013d} for non-commutative
algebras. A slightly more general definition can be found in
\cite{a-2013} where the unitary assumption on the algebras is
dropped. In the case of commutative algebras we arrive at the
following simplified definition which originates in \cite{bai}: a
\emph{matched pair} of commutative algebras is a system $\bigl(P,
\, Q, \, \triangleleft, \, \triangleright)$ consisting of two
commutative algebras $P$ and $Q$ and two bilinear maps
$\triangleleft : Q \times P \to Q$ and $\triangleright : Q \times
P \to P$ such that $(Q, \, \triangleleft)$ is a right $P$-module,
$(P, \, \triangleright)$ is a left $Q$-module satisfying the
following compatibility conditions for any $a$, $b\in P$ and $x$,
$y\in Q$:
\begin{eqnarray*}
(x y) \triangleleft a = x \triangleleft (y
\triangleright a) + x (y \triangleleft a),  \qquad
x \triangleright (ab) = a (x \triangleright b) + (x
\triangleleft b) \triangleright a \eqlabel{mp1b}
\end{eqnarray*}
It is worth pointing out that these axioms are exactly what
remains from (A1)-(A6) of \deref{comexdatum} if we ask for
$\Omega(P, Q) = \bigl(\triangleleft, \, \triangleright, \, f := 0,
\, \cdot \bigl)$ to be an algebra extending system of the algebra
$P$ by a vector space $Q$, where $f := 0$ is the trivial map. The
concept of a matched pair of Poisson algebras was recently
introduced in \cite[Theorem 1]{NiBai} - we recall the definition
following our notations and terminology since it will be a special
case of the axioms (P1)-(P10) appearing in \coref{1}.

\begin{definition}\delabel{matchedposs}
A \emph{matched pair} of Poisson algebras is a system $(P, \, Q,
\, \triangleleft, \, \triangleright, \, \leftharpoonup, \,
\rightharpoonup)$ consisting of two Poisson algebras $P$ and $Q$
and four bilinear maps
\begin{eqnarray*}
\triangleleft : Q \times P \to Q, \quad \triangleright : Q \times
P \to P, \quad \leftharpoonup: Q \times P \to Q, \quad
\rightharpoonup : Q \times P \to P
\end{eqnarray*}
such that $\bigl(P, \, Q, \, \triangleleft, \, \triangleright)$ is
a matched pair of commutative algebras, $\bigl(P, \, Q, \,
\leftharpoonup, \, \rightharpoonup \bigl)$ is a matched pair of
Lie algebras satisfying the following compatibility conditions for
any $a$, $b\in P$ and $x$, $y\in Q$:
\begin{eqnarray}
x \rightharpoonup (ab) &=& (x \rightharpoonup a) \, b + (x
\leftharpoonup a) \triangleright b + a \, (x \rightharpoonup b)
+ (x \leftharpoonup b) \triangleright a  \eqlabel{mp2p} \\
x \leftharpoonup (ab) &=& (x \leftharpoonup a)
\triangleleft b + (x \leftharpoonup b) \triangleleft a \eqlabel{mp3p} \\
x \, \triangleright [a, \, b]_P &=& [x \triangleright a, \, b]_P +
(x \triangleleft a) \rightharpoonup b - a \, (x
\rightharpoonup b) - (x \leftharpoonup b) \triangleright a \eqlabel{mp4p} \\
x  \triangleleft [a, \, b]_P &=& (x \triangleleft a)
\leftharpoonup b - (x  \leftharpoonup b) \triangleleft a \eqlabel{mp5p} \\
\{x, \, y\}_Q \triangleright a  &=&
x \triangleright (y \rightharpoonup a) - y \rightharpoonup (x \triangleright a) \eqlabel{mp6p} \\
 \{x, \, y\}_Q \triangleleft a &=& \{x \triangleleft a, \,
y \}_Q - y \leftharpoonup (x \triangleright a) + x \triangleleft
(y
\rightharpoonup a) + (y \leftharpoonup a)  x \eqlabel{mp7p} \\
(x y) \rightharpoonup a &=& x \triangleright (y
\rightharpoonup a) + y \triangleright (x \rightharpoonup a) \eqlabel{mp8p} \\
(x  y) \leftharpoonup  a &=&  x  (y \leftharpoonup a) + (x
\leftharpoonup a)  y + x \triangleleft (y \rightharpoonup a) + y
\triangleleft (x \rightharpoonup a) \eqlabel{mp9p}
\end{eqnarray}
\end{definition}

The axioms defining a matched pair of Poisson algebras in
\deref{matchedposs} are derived from axioms (P1)-(P10) of
\coref{1} if we ask for $\Upsilon (P, Q) = \bigl(\triangleleft, \,
\triangleright, \, f :=0, \, \cdot, \, \leftharpoonup, \,
\rightharpoonup, \, \theta :=0, \, \{-, \, - \} \bigl)$ to be a
Poisson extending structure of the Poisson algebra $P$ through
$Q$, where the cocycles $f$ and $\theta$ are both trivial: $f (x,
y) = \theta (x, y) := 0$, for all $x$, $y \in Q$.

Let $\bigl(P, \, Q, \, \triangleleft, \, \triangleright, \, \,
\leftharpoonup, \, \rightharpoonup \bigl)$ be a matched pair of
Poisson algebras. Then $P \bowtie Q := P \times Q$ is a Poisson
algebra with the multiplication and the bracket defined for any
$a$, $b \in P$ and $x$, $y \in Q$ by:
\begin{eqnarray}
(a, \, x) \bullet (b, \, y) &:=& \bigl( ab + x \triangleright b +
y \triangleright a, \,\, x\triangleleft b + y \triangleleft a  +
x y \bigl) \eqlabel{multbpp1} \\
\left [(a, x), \, (b, y) \right ] &:=& \bigl( [a, \, b]_P + x
\rightharpoonup b - y \rightharpoonup a, \,\, x\leftharpoonup b -
y \leftharpoonup a + \{x, \, y \}_Q \bigl) \eqlabel{multbpp2}
\end{eqnarray}
called the \emph{bicrossed product} associated to the matched pair
$\bigl(P, \, Q, \, \triangleleft, \, \triangleright, \, \,
\leftharpoonup, \, \rightharpoonup \bigl)$. The bicrossed product
of Poisson algebras is exactly the unified product $P \ltimes Q$
associated to a Poisson extending structure of the Poisson algebra
$P$ through $Q$ having both cocyles trivial $f = \theta :=0$. The
bicrossed product is the construction responsible for the
so-called \emph{factorization problem} which formulated at the
level of Poisson algebras comes down to the following: for two
given Poisson algebras $P$ and $Q$ describe and classify all
Poisson algebras which factorize through $P$ and $Q$. The next
result is \cite[Theorem 1]{NiBai}. For the reader's convenience we
also include a short and different proof of it based on
\prref{classif}.

\begin{proposition} \prlabel{factprob}
Let $P$ and $Q$ be two given Poisson algebras. A Poisson algebra
$R$ factorizes trough $P$ and $Q$ if and only if there exists
$\bigl(P, \, Q, \, \triangleleft, \, \triangleright, \,
\leftharpoonup, \, \rightharpoonup \bigl)$ a matched pair of
Poisson algebras such that $R \cong P \bowtie Q$, an isomorphism
of Poisson algebras.
\end{proposition}

\begin{proof}
First we observe that any bicrossed product $P \bowtie Q$
factorizes through $P \cong P \times \{0\}$ and $Q \cong \{0\}
\times Q$ and, via this identification, $P$ and $Q$ are Poisson
subalgebras of $P \bowtie Q$ such that $P \bowtie Q = P + Q$ and
$P \cap Q = \{0\}$. Conversely, assume that $R = (R, \cdot_R, [-,
\, -]_R)$ is a Poisson algebra that factorizes through $P$ and
$Q$. Let $\iota : P \to R$ be the inclusion map and $\pi: R \to P$
the canonical $k$-linear retraction of $\iota$, i.e. $\pi (p + q)
: = p$. Since $Q = {\rm Ker}(\pi)$ is a Poisson subalgebra of $R$
we obtain that the cocycles $f$ and $\theta : Q\times Q \to P$
constructed in the proof of \prref{classif} are both trivial: i.e.
$f(x, y) := \pi (x\cdot_R y) = 0$ and $\theta (x, y) := \pi ([x,
y]_R) = 0$, for all $x$, $y\in Q$. Thus, the actions
$(\triangleleft, \, \triangleright, \, \leftharpoonup, \,
\rightharpoonup)$ defined for any $a \in P$ and $x\in Q$ by:
\begin{eqnarray*}
x \triangleright a &:=& \pi (x \cdot_R a), \qquad \, x
\triangleleft a
:= x\cdot_R a  - \pi (x\cdot_R a) \\
x \rightharpoonup a &:=& \pi \bigl([x, \, a]_R \bigl), \quad x
\leftharpoonup a := [x, \, a]_R - \pi \bigl([x, \, a]_R\bigl)
\end{eqnarray*}
make $\bigl(P, \, Q, \, \triangleleft, \, \triangleright, \,
\leftharpoonup, \, \rightharpoonup \bigl)$ a matched pair of
Poisson algebras while $\varphi: P \bowtie Q \to R$, $\varphi (a,
x): = a + x$ becomes an isomorphism of Poisson algebras.
\end{proof}

As we have seen in the proof of \prref{factprob}, if a Poisson
algebra $R$ factorizes through $P$ and $Q$ then we can construct a
matched pair of Poisson algebras as follows:
\begin{eqnarray}
x \triangleright a + x \triangleleft a = x a, \qquad x
\rightharpoonup a + x \leftharpoonup a = [x, \, a]
\end{eqnarray}
for all $a \in P$ and $x \in Q$. Throughout, the above matched
pair will be called \emph{the canonical matched pair} associated
with the factorization of $R$ through $P$ and $Q$.
\prref{factprob} allows for a computational reformulation of the
factorization problem as follows: for two given Poisson algebras
$P$ and $Q$ describe and classify all bicrossed products $P
\bowtie Q$ associated to all possible matched pairs $\bigl(P, \,
Q, \, \triangleleft, \, \triangleright, \, \leftharpoonup, \,
\rightharpoonup \bigl)$. The problem is far from being a trivial
one.

\begin{example}\exlabel{anaex1}
Let $k_{0}$ be the trivial $1$-dimensional Poisson algebra with
basis $\{X\}$ and $\mathcal{H}$ the $3$-dimensional Heisenberg Lie
algebra with basis $\{H_1, H_2, H_3\}$ and the bracket defined by
$[H_1,\, H_2] = H_3$. $\mathcal{H}$ admits a Poisson algebra
structure with the associative multiplication given by $H_{1}^{2}
= H_3$. It can be easily seen by a straightforward computation
that the bicrossed products corresponding to all matched pairs of
Poisson algebras $\bigl(k_{0},\, \mathcal{H}, \, \triangleleft, \,
\triangleright, \, \leftharpoonup, \, \rightharpoonup \bigl)$ are
the $4$-dimensional Poisson algebras with basis $\{X,\, H_{1}, \,
H_{2}, \, H_{3}\}$ listed below:
\begin{eqnarray*}
\mathcal{H}^{1}_{(\beta, \, \alpha, \, \mu, \, \eta)}: &&
X^{2}=0,\, H_{1}^{2} = H_{3}, \, XH_{1} = H_{1}X = \alpha\,
H_{3},\, XH_{2} = H_{2}X = \beta\, H_{3}\\
&& [H_{1}, \, H_{2}] = H_{3}, \, [H_{1}, \, X] = \mu\, H_{3}, \,
[H_{2}, \, X] = \eta H_{3}, \,\, {\rm for} \,\, (\beta, \alpha, \mu, \eta) \in k^{*} \times k^{3}\\
\mathcal{H}^{2}_{(\xi,\, \mu,\, \eta)}: && X^{2}=0,\, H_{1}^{2} =
H_{3}, \, [H_{1}, \, H_{2}] = H_{3}, \, [H_{1}, \, X] = \xi\, X +
\mu\, H_{2} + \eta\, H_{3}\\
&& {\rm for} \,\, (\xi,\, \mu,\, \eta) \in k^{3}\\
\end{eqnarray*}
\begin{eqnarray*}
\mathcal{H}^{3}_{(\eta,\, \mu,\, \tau)}: && X^{2}=0,\, H_{1}^{2} =
H_{3}, \,  [H_{1}, \, H_{2}] = H_{3}, \, [H_{2}, \, X] = \eta\,
H_{3}\\
&& [H_{1}, \, X] = \mu\, H_{2} + \tau\, H_{3}, \,\, {\rm for} \,\, (\eta,\, \mu,\, \tau) \in k^{*} \times k^{2}\\
\mathcal{H}^{4}_{(\alpha,\, \mu, \, \tau, \eta)}: && X^{2}=0,\,
H_{1}^{2} = H_{3}, \, XH_{1}=H_{1}X = \alpha\, H_{3}, \, [H_{1},
\, H_{2}] = H_{3} \\
&& [H_{1}, \, X] = \mu\, H_{2} + \tau\, H_{3}, \, [H_{2}, \, X] =
\eta\, H_{3}, \,\, {\rm for} \,\, (\alpha,\, \mu, \, \tau, \eta)
\in k^{*} \times k^{3}\\
\mathcal{H}^{5}_{(\xi,\, \mu,\, \eta)}: && X^{2}=0,\, H_{1}^{2} =
H_{3}, \, [H_{1}, \, H_{2}] = H_{3}\\
&& \left[H_{1},\, X\right] = \mu \, H_{3},\, [H_{2},\, X]= \xi \,
X + \eta \,H_{3},
\,\,  {\rm for} \,\, (\xi,\, \mu,\, \eta) \in k^{*} \times k^{2}\\
\mathcal{H}^{6}_{(\xi,\, \gamma,\, \mu,\, \eta)}: && X^{2}=0,\,
H_{1}^{2} = H_{3}, \, [H_{1}, \, H_{2}] = H_{3}, \, \left[H_{1},\,
X\right] = \xi \, X + \xi \, \gamma^{-1} \, \eta \, H_{2} + \mu \,
H_{3} \\
&& [H_{2}, X] = \gamma \, X + \eta \, H_{2} + \xi^{-1} \, (\xi +
1)\, \eta \, H_{3}, \,\,  {\rm for} \,\, (\xi,\, \gamma,\, \mu,\,
\eta) \in (k^{*})^{2} \times k^{2}
\end{eqnarray*}
\end{example}

The rest of the section we deal with the converse of the
factorization problem, called the bicrossed descent (or the
classification of complements) problem and whose statement was
given in the introduction. First we need to introduce the
following concept:

\begin{definition} \delabel{defmap}
Let $\bigl(P, \, Q, \, \triangleleft, \, \triangleright, \,
\leftharpoonup, \, \rightharpoonup \bigl)$ be a matched pair of
Poisson algebras. A $k$-linear map $r: Q \to P$ is called
a \emph{deformation map} of the above matched pair if the
following compatibilities hold for all $p$, $q \in Q$:
\begin{eqnarray}
r(p) r(q) - r(pq) &=& r \bigl(q \triangleleft r(p) + p
\triangleleft r(q)\bigl) - q \triangleright r(p) - p
\triangleright r(q) \eqlabel{ana1}\\
r\bigl([p,\,q]\bigl) - \bigl[r(p),\, r(q)\bigl] &=& r\bigl(q
\leftharpoonup r(p) - p \leftharpoonup r(q)\bigl) + p
\rightharpoonup r(q) - q \rightharpoonup r(p \eqlabel{ana2})
\end{eqnarray}
\end{definition}

We denote by $\mathcal{D}\mathcal{M}\bigl(P,\, Q ~|~
(\triangleleft, \, \triangleright, \, \leftharpoonup, \,
\rightharpoonup)\bigl)$ the set of all deformation maps of the
matched pair $\bigl(P, \, Q, \, \triangleleft, \, \triangleright,
\, \leftharpoonup, \, \rightharpoonup \bigl)$.

\begin{example}\exlabel{defmap1}
Consider the following matched pair between $k_{0}$ and
$\mathcal{H}$:
\begin{eqnarray*}
 H_{1} \triangleleft X = H_{3}, \,\,\, H_{2}
\triangleleft X = H_{3}, \,\,\, H_{1} \leftharpoonup X = H_{3},
\,\,\, H_{2} \leftharpoonup X = H_{3}
\end{eqnarray*}
where the undefined actions are all equal to $0$. The
corresponding bicrossed product is the Poisson algebra
$\mathcal{H}^{1}_{(1, \, 1,\, 1, \, 1)}$ from \exref{anaex1}. Any
deformation map associated to the above matched pair is given as
follows:
$$
r_{(a_{1},\,a_{2})} : \mathcal{H} \to k_{0}, \quad
r_{(a_{1},\,a_{2})}(h_{1}) = a_{1} \, X, \quad
r_{(a_{1},\,a_{2})}(h_{2}) = a_{2}\, X, \quad
r_{(a_{1},\,a_{2})}(h_{3}) = 0
$$
for some $a_{1}$, $a_{2} \in k$.
\end{example}

The next result shows that to any deformation map $r$ we can
associate a new Poisson algebra called the $r$-deformation and,
moreover, all complements of a given Poisson algebra extension $P
\subset R$ can be described as $r$-deformations of a given
complement.

\begin{theorem}\thlabel{comp}
Let $P$ be a Poisson subalgebra of $R$ and $Q$ a given
$P$-complement of $R$ with the associated canonical matched pair
$\bigl(P, \, Q, \, \triangleleft, \, \triangleright, \,
\leftharpoonup, \, \rightharpoonup \bigl)$.

$(1)$ Let $r: Q \to P$ be a deformation map of the above matched
pair. Then $Q_{r}: = Q$, as a vector space, with the new Poisson
algebra structure defined for any $q$, $t \in Q$ by:
\begin{eqnarray}
q \cdot_{r} t &=& qt + t \triangleleft r(q) + q \triangleleft
r(t) \eqlabel{srt1}\\
\left[q,\, t\right]_{r} &=& [q, \, t] + q \leftharpoonup r(t) - t
\leftharpoonup r(q)\eqlabel{srt2}
\end{eqnarray}
is a Poisson algebra called the $r$-deformation of $Q$ and
$Q_{r}$ is a $P$-complement of $R$.

$(2)$ $\overline{Q}$ is a $P$-complement of $R$ if and only if
there exists an isomorphism of Poisson algebras $\overline{Q}
\cong Q_{r}$, for some deformation map $r: Q \to P$ of the above
canonical matched pair.
\end{theorem}

\begin{proof}
Let $P \bowtie Q$ be the bicrossed product associated to the
canonical matched pair $\bigl(P, \, Q, \, \triangleleft, \,
\triangleright, \, \leftharpoonup, \, \rightharpoonup \bigl)$. It
follows form \prref{factprob} that $R \cong P \bowtie Q$ as
Poisson algebras.

$(1)$ Although this claim can be proven by a very long but
straightforward computation, we will provide a different and more
natural approach. Given a deformation map $r: Q \to P$, we
consider $f_{r}: Q \to  P \bowtie Q $ to be the $k$-linear map
defined for all $q \in Q$ by:
$$f_{r}(q) = \bigl(r(q), \, q\bigl)$$
It turns out that $\widetilde{Q} : = {\rm Im}(f_{r})$ is a $P$
complement of $R \cong P \bowtie Q$. We start by proving that $\widetilde{Q}$ is a
Poisson subalgebra of $P \bowtie Q$. Indeed, for all $p$, $q \in
Q$ we have:
\begin{eqnarray*}
\bigl(r(p),\, p\bigl) \bigl(r(q), \, q\bigl) &=&
\bigl(\underline{r(p)r(q) + p \triangleright r(q) + q
\triangleright r(p)}, \, p \triangleleft
r(q) + q \triangleleft r(p) + pq\bigl)\\
&\stackrel{\equref{srt1}}{=}& \bigl(r(pq + p \triangleleft r(q) +
q \triangleleft r(p)),\, p \triangleleft r(q) + q \triangleleft
r(p) + pq\bigl)\\
\left[(r(p),\, p),\, (r(q),\, q)\right] &=& \bigl( \,
\underline{[r(p),\, r(q)] + p \rightharpoonup r(q) - q
\rightharpoonup r(p)}, \\
&& p \leftharpoonup r(q) - q \leftharpoonup r(p) + [p, \, q] \, \bigl)\\
&\stackrel{\equref{srt2}}{=}& \bigl(r([p, q] + p \leftharpoonup
r(q) - q \leftharpoonup r(p)), p \leftharpoonup r(q) - q
\leftharpoonup r(p) + [p, q]\bigl)
\end{eqnarray*}
Therefore $\widetilde{Q}$ is a Poisson subalgebra of $P \bowtie
Q$. Consider now $(p,\, q) \in P \cap \widetilde{Q}$. Since in
particular we have $(p,\, q) \in \widetilde{Q}$ then $p = r(q)$.
As we also have $(r(q),\, q) \in P$ we obtain $q = 0$ and thus $P
\cap \widetilde{Q} = \{0\}$. Furthermore, for any $(p, \, q) \in R
= P \bowtie Q$ we can write $(p, \, q) = (p - r(q),\,0) + (r(q),
\, q) \in P + \widetilde{Q}$. Hence, we have proved that
$\widetilde{Q}$ is a $P$-complement of $P \bowtie Q$. We are left to prove
that $Q_{r}$ and $\widetilde{Q}$ are isomorphic as Poisson
algebras. To this end, we denote by $\widetilde{f}_{r}$ the linear
isomorphism from $Q$ to $\widetilde{Q}$ induced by $f_{r}$. As we
will see, $\widetilde{f}_{r}$ is a Poisson algebra map if we
consider $Q$ endowed with the Poisson structures given by
\equref{srt1} and \equref{srt2}. Indeed, for all $q$, $t \in Q$ we
have:
\begin{eqnarray*}
\overline{f}_{r}(q \cdot_{r} t) &\stackrel{\equref{srt1}}{=}&
\overline{f}_{r}(qt + t \triangleleft r(q) + q \triangleleft
r(t))\\
&=& \bigl(\underline{r(qt + t \triangleleft r(q) + q \triangleleft
r(t))}, \,
qt + t \triangleleft r(q) + q \triangleleft r(t)\bigl)\\
&\stackrel{\equref{ana1}}{=}& \bigl(r(q)r(t) + q \triangleright
r(t) + t \triangleright r(q),\, qt + t \triangleleft r(q) + q
\triangleleft r(t)\bigl)\\
&\stackrel{\equref{multbpp1}}{=}& \bigl(r(q), \, q\bigl)
\bigl(r(t), \, t\bigl) = \overline{f}_{r}(q) \overline{f}_{r}(t)
\end{eqnarray*}
and
\begin{eqnarray*}
\overline{f}_{r}\bigl([q,\, t]_{r}\bigl)
&\stackrel{\equref{srt1}}{=}& \overline{f}_{r} \bigl([q, \, t] + q
\leftharpoonup r(t) - t \leftharpoonup r(q)\bigl)\\
&=& \bigl(\underline{r([q, \, t] + q \leftharpoonup r(t) - t
\leftharpoonup r(q))},\, [q, \, t] + q \leftharpoonup r(t) - t
\leftharpoonup r(q)\bigl)\\
&\stackrel{\equref{ana2}}{=}& \bigl(\left[r(q), \, r(t)\right] + q
\rightharpoonup r(t) - t \rightharpoonup r(q),\, [q, \, t] + q
\leftharpoonup r(t) - t \leftharpoonup r(q)\bigl)\\
&\stackrel{\equref{multbpp2}}{=}& \left[\bigl(r(q),\, q\bigl), \,
\bigl(r(t), \, t\bigl) \right] = \left[\overline{f}_{r}(q),\,
\overline{f}_{r}(t)\right]
\end{eqnarray*}
Hence we can conclude that $Q_{r}$ is a Poisson algebra and this
finishes the proof.

$(2)$ Let $\overline{Q}$ be an arbitrary $P$-complement of $R$. As
$R = P \oplus Q = P \oplus \overline{Q}$ we can find four
$k$-linear maps:
$$
a: Q \to P, \qquad b: Q \to \overline{Q}, \qquad
c: \overline{Q}
\to P, \qquad d: \overline{Q} \to Q
$$
such that for all $q \in Q$ and $t \in \overline{Q}$ we have:
\begin{eqnarray}
q = a(q) \oplus b(q), \qquad t = c(t) \oplus d(t) \eqlabel{ana3}
\end{eqnarray}
It follows by an easy computation that $b: Q \to \overline{Q}$ is
an isomorphism of vector spaces. We will denote by $\widetilde{b}:
Q \to P \bowtie Q$ the composition $\widetilde{b} := i \circ b$,
where $i: \overline{Q} \to R = P \bowtie Q$ is the canonical
inclusion. Thus, by \equref{ana3} we have $\widetilde{b}(t) =
(-a(t), \, t)$ for all $t \in Q$. We will prove that $r := -a$ is
a deformation map and $\overline{Q} \cong Q_{r}$. Indeed,
$\overline{Q} = {\rm Im}(b) = {\rm Im} (\widetilde{b})$ is a
Poisson subalgebra of $R = P \bowtie Q$ and we have:
\begin{eqnarray*}
\bigl(r(q),\, q\bigl) \bigl(r(t),\, t\bigl) =
&\stackrel{\equref{multbpp1}}{=}& \bigl(r(q) r(t) + q
\triangleright r(t) + t \triangleright r(q), \, q \triangleleft
r(t) + t \triangleleft r(q) + qt\bigl)\\
&=& \bigl(r(t'), t'\bigl)\\
\left[\bigl(r(q),\, q\bigl),\, \bigl(r(t),\, t\bigl)\right]
&\stackrel{\equref{multbpp2}}{=}& \bigl( \, [r(q), \, r(t)] + q
\rightharpoonup r(t) - t \rightharpoonup r(q), \\
&& q \leftharpoonup r(t) - t \leftharpoonup r(q) + [q, \, t]\bigl)
\, = \bigl(r(t''), \, t''\bigl)
\end{eqnarray*}
for some $t'$, $t'' \in Q$. Hence, we have:
\begin{eqnarray}
r(t') = r(q) r(t) + q \triangleright r(t) + t \triangleright
r(q),&&
\quad t' = q \triangleleft r(t) + t \triangleleft r(q) + qt \eqlabel{ana5}\\
r(t'') = [r(q), \, r(t)] + q \rightharpoonup r(t) - t
\rightharpoonup r(q),&& \quad t'' = q \leftharpoonup r(t) - t
\leftharpoonup r(q) + [q, \, t] \eqlabel{ana6}
\end{eqnarray}
By applying $r$ to the second part of \equref{ana5}, respectively
\equref{ana6}, we obtain that $r$ is a deformation map. Moreover,
by a straightforward computation using \equref{srt1},
\equref{srt2}, \equref{ana5} and respectively \equref{ana6} it
follows that $b: Q_{r} \to \overline{Q}$ is a Poisson algebra map
and the proof is now finished.
\end{proof}

\begin{examples}\exlabel{def20}
Let $k$ be an algebraically closed field of characteristic zero
and let $\mathcal{H}^{1}_{1,1,1,1}$ be the bicrossed product
described in \exref{anaex1}. For any $(a_{1},\,a_{2}) \in k^{2}$
consider $r_{(a_{1},\,a_{2})} : \mathcal{H} \to k_{0}$ the
associated deformation map described in \exref{defmap1}. Then, the
corresponding $r_{(a_{1},\,a_{2})}$-deformation
$\mathcal{H}_{r_{(a_{1},\,a_{2}})}$ of the Heisenberg Poisson
algebra has the associative algebra structure and the Lie bracket
given as follows:
\begin{eqnarray*}
\mathcal{H}_{r_{(a_{1},\,a_{2}})}: \quad H_{1}^{2} = (2a_{1}+1) \,
H_{3}, \,\,\, H_{1}\,H_{2} = (a_{1}+a_{2})\,H_{3}, \quad [H_{1},
\, H_{2}] = (a_{2} - a_{1} + 1)\,H_{3}
\end{eqnarray*}
If $a_{1} = a_{2} = 0$ then $\mathcal{H}_{r_{(0,\,0})}$ coincides
with $\mathcal{H}$. Moreover, for any $a \in k-\{2^{-1}\}$ the
$r_{(a,\, -a)}$-deformation $\mathcal{H}_{r_{(a,\, -a)}}$ is
isomorphic to the Heisenberg Poisson algebra $\mathcal{H}$, the
isomorphism of Poisson algebras being given by:
\begin{eqnarray*}
\varphi: \mathcal{H}_{r_{(a,\, -a)}} \to \mathcal{H}, \quad
\varphi(H_{1}) = \alpha\, H_{1}, \,\, \varphi(H_{2}) = (-2a+1)
\alpha^{-1} \, H_{2}, \,\, \varphi(H_{3}) = H_{3}
\end{eqnarray*}
where $\alpha$ is a square root of $(2a+1)$.
However, if we consider $a_{1} = 1$ and $a_{2} = 0$ then we obtain
the Poisson algebra $\mathcal{H}_{r_{(1,\,0)}}$ with the trivial
Lie bracket and the multiplication given by: $H_{1}^{2} = 3\,H_{3}, \, H_{1}\,H_{2} = H_{3} $.
Therefore, $\mathcal{H}$ and $\mathcal{H}_{r_{(1,\,0)}}$ are not
isomorphic as Poisson algebras as a consequence of not being
isomorphic as Lie algebras.
\end{examples}

As the previous example shows it, different deformation maps can
give rise to isomorphic deformations. Therefore, in order to
provide the classification of complements we need to introduce the
following:

\begin{definition}\delabel{equivDef}
Let $\bigl(P, \, Q, \, \triangleleft, \, \triangleright, \,
\leftharpoonup, \, \rightharpoonup \bigl)$ be a matched pair of
Poisson algebras. Two deformation maps $r$, $r': Q \to P$ are
called \emph{equivalent} and we will denote this by $r \sim r'$ if
there exists $\sigma : Q \to Q$ a linear automorphism of $Q$ such
that for any $q$, $t \in Q$ we have:
\begin{eqnarray*}\eqlabel{equiv111}
\sigma(qt) - \sigma(q)\sigma(t) &=& \sigma(q) \triangleleft
r'\bigl(\sigma(t)\bigl) + \sigma(t) \triangleleft
r'\bigl(\sigma(q)\bigl) - \sigma \bigl(q \triangleleft r(t) \bigl)
- \sigma \bigl(t \triangleleft r(q)\bigl)\\
\sigma\bigl([q, \, t]\bigl) - \bigl[\sigma(q), \, \sigma(t)\bigl]
&=& \sigma(q) \leftharpoonup r'\bigl(\sigma(t)\bigl) - \sigma(t)
\leftharpoonup r'\bigl(\sigma(q)\bigl) + \sigma \bigl(t
\leftharpoonup r(q)\bigl) - \sigma \bigl(q \leftharpoonup
r(t)\bigl)
\end{eqnarray*}
\end{definition}

The main result of this section which provides the answer to the
bicrossed descent problem now follows:

\begin{theorem}\thlabel{clasif100}
Let $P$ be a Poisson subalgebra of $R$, $Q$ a $P$-complement of
$R$ and $\bigl(P, \, Q, \, \triangleleft, \, \triangleright, \,
\leftharpoonup, \, \rightharpoonup \bigl)$ the associated
canonical matched pair. Then $\sim$ is an equivalence relation on
the set $\mathcal{D}\mathcal{M}\bigl(P,\, Q ~|~ (\triangleleft, \,
\triangleright, \, \leftharpoonup, \, \rightharpoonup)\bigl)$ and
the map
$$
\mathcal{H}\mathcal{A}^{2} \bigl(P,\, Q ~|~ (\triangleleft, \,
\triangleright, \, \leftharpoonup, \, \rightharpoonup)\bigl) :=
\mathcal{D}\mathcal{M}\bigl(P,\, Q ~|~ (\triangleleft, \,
\triangleright, \, \leftharpoonup, \, \rightharpoonup)\bigl)/\sim
\to \mathcal{F}(P,\, R), \quad \overline{r} \mapsto Q_{r}
$$ is
a bijection, where $\mathcal{F}(P,\,R)$ is the set of
isomorphism classes of all $P$-complements of $R$. In particular,
the factorization index of $P$ in $R$ is computed by the formula:
$$\left[R : P\right]^{f} = |\mathcal{H}\mathcal{A}^{2} \bigl(P,\, Q ~|~
(\triangleleft, \, \triangleright, \, \leftharpoonup, \,
\rightharpoonup)\bigl)|$$
\end{theorem}

\begin{proof}
Two deformation maps $r$ and $r'$ are equivalent in the sense of
\deref{equivDef} if and only if the corresponding Poisson algebras
$Q_{r}$ and $Q_{r'}$ are isomorphic. The conclusion follows by
\thref{comp}.
\end{proof}

In what follows, for any $a \in k$, we denote by $\mathcal{H}_{a}$
the deformation of the Heisenberg Poisson algebra described in
\exref{def20} associated to the pair $(0,\,a) \in k^{2}$. More
precisely, the Poisson algebra structure on $\mathcal{H}_{a}$ is
as follows:
$$
\mathcal{H}_{a}: \quad H_{1}^{2} = H_{3},\,\,
H_{1}\,H_{2} = a\,H_{3},\,\, [H_{1}, \, H_{2}] = (a + 1)\,H_{3}
$$
Our next result provides a classification result for the
$k_{0}$-complements of $\mathcal{H}^{1}_{1,1,1,1}$ from
\exref{anaex1}. In particular, the deformations $\mathcal{H}_{a}$
provide an infinite family of non-isomorphic three dimensional
Poisson algebras - for a similar result in the setting of Lie
algebras see \cite{AMSigma}.

\begin{proposition} \prlabel{ultimapr}
Let $k$ be an algebraically closed field of characteristic zero
and $a$, $b \in k-\{-1,\, -2^{-1},\, 0 \}$. Then $\mathcal{H}_{a}$
and $\mathcal{H}_{b}$ are isomorphic as Poisson algebras if and
only if $a=b$ or $a = -b(2b+1)^{-1}$. In particular, the
factorization index $[\mathcal{H}^{1}_{1,1,1,1} : k_{0}]^{f}$ is
infinite.
\end{proposition}

\begin{proof}
Suppose $\varphi: \mathcal{H}_{a} \to \mathcal{H}_{b}$ is a
Poisson algebra isomorphism, where $\varphi(H_{1}) =
\Sigma_{i=1}^{3} \alpha_{i} H_{i}$, $\varphi(H_{2}) =
\Sigma_{i=1}^{3} \beta_{i} H_{i}$, $\varphi(H_{3}) =
\Sigma_{i=1}^{3} \gamma_{i} H_{i}$, $\alpha_{i}$, $\beta_{i}$,
$\gamma_{i} \in k$. Thus \thref{clasif100} implies $r_{(0,\,a)}
\sim r_{(0,\,b)}$ and we obtain:
\begin{eqnarray}
\gamma_{1} = \gamma_{2} = 0,\quad \alpha_{1}^{2} +
2\,\alpha_{1}\,\alpha_{2}\,b = \gamma_{3},\quad \beta_{1}^{2} + 2\, \beta_{1}\,\beta_{2}\,b = 0\eqlabel{com1.2}\\
\alpha_{1}\,\beta_{1} + (\alpha_{1}\,\beta_{2} + \alpha_{2}\,
\beta_{1}) \,b = a\, \gamma_{3},\quad (\alpha_{1}\,\beta_{2} -
\alpha_{2}\, \beta_{1}) \,(b+1) = (a+1)\,
\gamma_{3}\eqlabel{com1.4}
\end{eqnarray}
To start with, we point out that since $\varphi$ is an isomorphism
we must have $\gamma_{3} \neq 0$. The last part of \equref{com1.2}
implies $\beta_{1} = 0$ or $\beta_{1} + 2\, \beta_{2} \, b = 0$.
Assume first that $\beta_{1} = 0$. As $\varphi$ is an isomorphism
it follows that $\alpha_{1} \neq 0$. Then the first part of
\equref{com1.4} comes down to $\beta_{2} =
ab^{-1}\,\alpha_{1}^{-1}\, \gamma_{3}$. Using the second part of
\equref{com1.2} we obtain
$$
\beta_{2} = ab^{-1}\,\alpha_{1}^{-1}\,\underline{\gamma_{3}}
\stackrel{\equref{com1.2}}{=}
ab^{-1}\,\alpha_{1}^{-1}\,(\alpha_{1}^{2} +
2\,\alpha_{1}\,\alpha_{2}\,b) = ab^{-1}\,(\alpha_{1} +
2\,\alpha_{2}\,b)
$$
Using the second part of \equref{com1.2} and the above formulae
for $\beta_{2}$, the last part of \equref{com1.4} becomes:
$$
(a+1)\, \gamma_{3} = \alpha_{1}\,ab^{-1}\, (\alpha_{1} +
2\,\alpha_{2}\,b) (b+1) = ab^{-1}\, (\underline{\alpha_{1}^{2} +
2\,\alpha_{1}\,\alpha_{2}\,b})(b+1) \stackrel{\equref{com1.2}}{=}
ab^{-1}\, \gamma_{3} \, (b+1)
$$
As $\gamma_{3} \neq 0$ we obtain $ab^{-1}\,(b+1) = a+1$ which
implies $a=b$. Assume now that $\beta_{1} \neq 0$. Therefore, by
the last part of \equref{com1.2} we get $\beta_{1} = -2\,
\beta_{2}\, b$. Now using this formulae for $\beta_{1}$,
\equref{com1.4} becomes:
$$
-\,\beta_{2}\,b(\alpha_{1} + 2\,\alpha_{2}\,b) = a\,
\gamma_{3}, \quad
\beta_{2}\,(\alpha_{1} + 2\,\alpha_{2}\,b) = (a+1)(b+1)^{-1}\,
\gamma_{3}
$$
As $\gamma_{3} \neq 0$, we obtain $-a = b(a+1)(b+1)^{-1}$ which
gives $a = -b(2b+1)^{-1}$. Therefore we proved that
$\mathcal{H}_{a}$ and $\mathcal{H}_{b}$ are isomorphic Poisson
algebras if and only if $a=b$ or $a = -b(2b+1)^{-1}$. Together
with the fact algebraically closed fields are infinite we obtain
$[\mathcal{H}^{1}_{1,1,1,1} : k_{0}]^{f}$ is infinite. This
finishes the proof.
\end{proof}


\begin{thebibliography}{99}
\bibitem{abr}
Abrams, L. - The quantum Euler class and the quantum cohomology of
the Grassmannians, {\sl Is. J. Math.}, {\bf 117}(2000), 335-352.

\bibitem{a-2013}
Agore, A. L. - Classifying complements for associative algebras,
{\sl Linear Algebra Appl.}, {\bf 446} (2014), 345--355.

\bibitem{am-2011}
Agore, A.L. and Militaru, G. - Extending structures II: the
quantum version, {\sl J. Algebra} {\bf 336} (2011), 321--341.

\bibitem{am-2013a}
Agore, A.L., Militaru, G. - Classifying complements for Hopf
algebras and Lie algebras, {\sl J. Algebra} {\bf 391} (2013),
193--208.

\bibitem{am-2013}
Agore, A.L. and Militaru, G. - Extending structures for Lie
algebras, {\sl Monatsh. f\"{u}r Mathematik} {\bf 174} (2014),
169--193.

\bibitem{AMSigma}
Agore, A.L. and Militaru, G. -  Bicrossed products, matched pair
deformations and the factorization index for Lie algebras – {\sl
Symmetry Integr. Geom.} {\bf 10} (2014), 065, 16 pages.

\bibitem{am-2013d}
Agore, A.L., Militaru, G. - The extending structures problem for
algebras, arXiv:1305.6022.

\bibitem{bai}
Bai, C. - Double constructions of Frobenius algebras, Connes
cocycles and their duality, {\sl J. Noncommu. Geom.} {\bf 4}
(2010), 475–-530.

\bibitem{bour}
Bourbaki, N. - Lie groups and Lie algebres, Chap. 1-3, Springer,
Paris, 1989.

\bibitem{bohm}
B\"{o}hm, G., Nill, F. and Szlach\'{a}nyi, K. - Weak Hopf
algebras. I. Integral theory and C*-structure, {\sl J. Algebra}
{\bf 221} (1999), 385--438.

\bibitem{BW}
Brzezi\'nski, T., Wisbauer, R. - Corings and comodules, London
Math. Soc. Lect. Note Ser., {\bf 309} (2003), Cambridge University Press.


\bibitem{BC}
Burde, D., Ceballos, M. - Abelian ideals of maximal dimension for
solvable Lie algebras, {\sl J. Lie Theory} {\bf 22} (2012),
741–756.

\bibitem{CMZ2002}
Caenepeel, S., Militaru G. and Zhu, S. - Frobenius and Separable
Functors for Generalized Module Categories and Nonlinear
Equations, {\sl Lect. Notes Math.} {\bf 1787} Springer Verlag,
Berlin, 2002.

\bibitem{CT}
Ceballos, M., Towers, D.A. - On abelian subalgebras and ideals of
maximal dimension in supersolvable Lie algebras, {\sl J. Pure
Appl. Algebra} {\bf 218} (2014), 497-–503.

\bibitem{CM}
Chekhov, L. and Mazzocco, M. - Poisson Algebras of
Block-Upper-Triangular Bilinear Forms and Braid Group Action, {\sl
Comm. Math. Physics} {\bf 322} (2013), 49--71.


\bibitem{Dri}
Drinfeld, V. - Quantum groups, {\sl Proc. Intern. Congr. Math.,
Berkeley}, Vol. I (1987), 798 -- 820.

\bibitem{fig}
Figueroa-O'Farrill, J. M., Stanciu, S. - On the structure of
symmetric self-dual Lie algebras, {\sl J. Math. Phys.} {\bf 37}
(1996), 4121–-4134.


\bibitem{gava}
Gavarini, F. - The crystal duality principle: from Hopf algebras
to geometrical symmetries, {\sl J. Algebra} {\bf 285} (2005),
399--437.

\bibitem{gor}
Gorbatsevich, V. V. - On the level of some solvable Lie algebras,
{\sl Siberian Math. J.} {\bf 39} (1998), 872–-883.

\bibitem{grab92}
Grabowski, J. - Abstract Jacobi and Poisson structures, {\sl J.
Geom. Phys.} {\bf 9} (1992), 45–-73.

\bibitem{gra1993}
Grabowski, J. Marmo, G. and Perelomov, A. M. - Poisson structures:
towards a classification, {\sl Modern Phys. Lett. A} {\bf
8}(1993), 1719--1733.

\bibitem{grabM}
Grabowski, J. and Marmo, G. - The graded Jacobi algebras and
(co)homology, {\sl J. Phys. A} {\bf 36} (2003), 161--181.

\bibitem{gra2013}
Grabowski, J. -  Brackets, {\sl Int. J. Geom. Methods Mod. Phys.},
10(8):1360001, 45, 2013.

\bibitem{H}
Humphreys, J.E. - Introducution to Lie algebras and Representation
theory, Spinger, 1972.


\bibitem{Miu2}
Iovanov, M.C. - Generalized Frobenius algebras and Hopf algebras,
{\sl Canad. J. Math.} {\bf 66} (2014), 205–-240.

\bibitem{Jacob}
Jacobson, N. -  Lie Algebras, Interscience Publ., New York–London,
1962.

\bibitem{kadison}
Kadison, L. - New Examples of Frobenius Extensions, Univ. Lect.
Series {\bf 14} (1999), Amer. Math. Soc., Providence.

\bibitem{kat}
Kath, I. and Olbrich, M. - Metric Lie algebras with maximal
isotropic centre, {\sl Math. Z.} {\bf 246} (2004) pp 23--53.

\bibitem{Kock}
Kock, J. - Frobenius algebras and 2D topological quantum field
theories, London Math. Soc. {\bf 59}(2003), Cambridge University
Press.

\bibitem{kirillov}
Kirillov, A. - Local Lie Algebras, Russian Math. Surveys {\bf 31}
(1976), 55-75.

\bibitem{lambre}
Lambre, T., Zhou, G. and Zimmermann, A. - The Hochschild
cohomology ring of a Frobenius algebra with semisimple Nakayama
automorphism is a Batalin-Vilkovisky algebra, arXiv:1405.5325.

\bibitem{LPV}
Laurent-Gengoux, C., Pichereau, A. and Vanhaecke, P. - Poisson
Structures, Vol. 347, 2013, Springer.

\bibitem{lich}
Lichnerowicz, A. - Le vari\'{e}t\'{e}s de Jacobi et leurs
alg\`{e}bres de Lie associ\'{e}es, {\sl J. Math. Pures Appl.} {\bf
57(4)} (1978), 453--488.

\bibitem{lie}
Lie, S. - Theorie der Transformationsgruppen, Vol. 1–-3, Leipzig, 1888, 1890, 1893.

\bibitem{LWZ}
Lu, J., Wang, X. and Zhuang, G. - Universal enveloping algebras of
Poisson Hopf algebras, {\sl J. Algebra} {\bf 426} (2015), 92–-136.

\bibitem{LW}
Lu, J.H. and Weinstein, A. - Poisson Lie groups, dressing
transformations and Bruhat decompositions,  {\sl J. Differential
Geom.} {\bf 31}(1990), 501--526.

\bibitem{majid}
Majid, S. - Physics for algebraists: non-commutative and
non-cocommutative  Hopf algebras by a bicrossproduct construction,
{\sl J. Algebra} {\bf 130} (1990), 17--64.

\bibitem{marle}
Marle, C. M. - On Jacobi Manifolds and Jacobi boundles, In
\emph{Symplectic Geometry, Groupoids, and Integrable Systems},
Mathematical Sciences Research Institute Publications {\bf 20}
(1991), 227--246, Springer.

\bibitem{mazzola}
Mazzola, G. - The algebraic and geometric classification of
associative algebras of dimension five, {\sl Manuscripta Math.}
{\bf 27} (1979), 81-–101.

\bibitem{medina}
Medina, A.,  Revoy, Ph. - Alg\'{e}bres de Lie et produit scalaire
invariant, {\sl Ann. Sci. \`{E}cole Norm. Sup.}  {\bf 18} (1985),
553–-561.


\bibitem{NiBai}
Xiang Ni and Chengming Bai - Poisson bialgebras, {\sl J. Math. Phys.}, {\bf 54}, 023515 (2013).

\bibitem{pelc}
Pelc, O. - A New Family of Solvable Self-Dual Lie Algebras, {\sl
J.Math.Phys.} {\bf 38} (1997), 3832--3840.

\bibitem{popovici}
Popovych, R.O., Boyko, V.M., Nesterenko, M.O. and Lutfullin, M. W. -
Realizations of real low-dimensional Lie algebras, {\sl Journal
of Physics A: Mathematical and General} {\bf 36} (2003), 7337.


\bibitem{radford}
Radford, D. E. - Hopf algebras, World Scientific Publishing,
Series on knots and everything, vol. {\bf 49} (2012).

\bibitem{rovi}
Rovi, A. - Hopf algebroids associated to Jacobi algebras, {\sl
Int. J. Geom. Methods Mod. Phys.} {\bf 11}, 1450092 (2014).

\bibitem{study}
Study, E. - \"{U}ber Systeme complexer Zahlen und ihre Anwendungen
in der Theorie der Transformationsgruppen, {\sl Monatshefte
f\"{u}r Mathematik} {\bf 1} (1890), 283–-354.

\bibitem{ZVZ}
Zhu, C., Van Oystaeyen, F. and Zhang Y. - On the (co) homology of
Frobenius Poisson algebras, {\sl Journal of K-theory: K-theory and
its Applications to Algebra, Geometry, and Topology} {\bf 14}
(2014), 371--386.
\end{thebibliography}
\end{document}